\documentclass[10pt]{amsart}

\usepackage[margin=1.2in]{geometry}

\usepackage{amsmath,amssymb,amsfonts,amsthm}
\usepackage{bbm}
\usepackage{bbold}
\usepackage{mathtools}
\usepackage{xcolor}
\usepackage{tikz-cd}
\usepackage{hyperref}
\usepackage[disable]{todonotes}
\usepackage{enumitem}
\usepackage[capitalise]{cleveref} 
\usepackage{autonum}
\usepackage{subfiles}

\usepackage{macros/utilities}
\usepackage{macros/environments}
\usepackage{macros/diagram-macros}
\usepackage{macros/macros}

\author{Viktoriya Ozornova, Martina Rovelli, Tashi Walde}
\title{Cores and localizations of \texorpdfstring{$(\infty,\infty)$}{(infinity, infinity)}-categories}

\begin{document}
\maketitle

\begin{abstract}
  We consider $(\infty,d)$-categories in the limit $d\to \infty$
  via the core or localization functors
  that forget or invert higher non-invertible arrows, respectively.
  We compare the two resulting $(\infty,1)$-categories
  of $(\infty,\infty)$-categories
  and exhibit the localization-limit
  as a reflective localization of the core-limit.

  Furthermore, we study notions of invertibility
  that only emerge at $d=\infty$
  and give rise to intermediate reflective localizations between the two limits.
  For example,
  we characterize coinductive completion and the maps it inverts,
  and reveal a new
  separation beyond the one of Henry--Loubaton
  by showing that even the stronger
  $\omega$-completeness does not suffice to carve out the localization-limit.
\end{abstract}

\tableofcontents

\newpage

\section{Introduction}

\subsection{What is an \texorpdfstring{$(\infty,\infty)$}{(infinity, infinity)}-category?}

It is well understood that a (weak\footnote{
  Strict categories make no appearance in this article,
  so we keep the adjective ``weak'' implicit.
}) $(\infty,\infty)$-category $C$
is supposed to be a structure with objects $x,y,\dots$
$1$-arrows $f,g,\dots\colon x\to y$ between them,
$2$-arrows $f\to g$ between (parallel) $1$-arrows
and so on with $d$-arrows in all dimensions $d=0,1,\dots$ to infinity.
These arrows are supposed to be composable in a way that is
well-defined, is associative and 
satisfies various interchange laws,
but only up to coherent systems of higher isomorphisms.
This idea seemingly leads to a self-referential conundrum:
An isomorphism is supposed to be an arrow
$f\colon x\to y$ which admits an inverse $f'\colon y\to x$
in the sense that $\id_x=f'f$ and $ff'=\id_y$.
But since composition is only well-defined up to higher isomorphism,
these equalities should themselves be replaced with isomorphisms of the next higher level,
a notion that is yet to be defined;
and because the arrow dimension goes \emph{up} rather than down each time,
there is no easy inductive way to resolve this issue\footnote{
  Compare to the case of $(d,d)$-categories for finite $d$,
  where one can start with only equalities above dimension $d$
  and then explicitly bootstrap by \emph{downward} induction.
}.

Fortunately, there is an established way for how to treat infinite
coherent systems of higher arrows, as long as these are all \emph{invertible}:
this is just the notion of an $\infty$-groupoid
(or anima\footnote{
  To emphasize their fundamental status
  (as opposed to viewing them as constructed from more primitive objects such a sets)
  we will join the recent trend of calling them \emph{animae}
  rather than the more classical ``spaces''
  which carries a distinctly topological connotation that we wish to avoid.
})
that is very well understood both via models in classical mathematics
(e.g., combinatorially via Kan complexes or topologically via CW-complexes)
or to some extent even intrinsically
(e.g., via the language of homotopy type theory).
Especially from the latter perspective it is fruitful to think
of isomorphisms in an anima not as (a priori non-invertible) arrows
that just happen to be invertible,
but rather as the inbuilt notion of ``identification'' or ``equality''
in the logical sense:
no well-formed predicate can distinguish between identified objects,
and feeding identified objects into any well-formed construction yields an identification
between the outputs.

Following this idea it is relatively straightforward to inductively describe
--- at least heuristically ---
what an $(\infty,d)$-category is for any finite $d$,
where the number $d$ indicates the dimension above which all arrows are invertible
(whatever that means):
An $(\infty,0)$-category is just an anima,
with the arrows declared to be the identifications,
all of which are then automatically invertible
(in the only sensible way, namely up to higher identifications).
An $(\infty,d+1)$-category consists of an anima $C^\simeq$ of objects
and for each pair of points $x,y\in C^\simeq$,
an $(\infty,d)$-category $C(x,y)$;
together with composition maps
$C(x,y)\times C(y,z)\to C(x,z)$
that are associative and unital up to coherent identifications.
Since the notion of identification is already inbuilt into the theory
from the ground up, this description is no longer circular.
One can then describe isomorphisms inductively:
an isomorphism $f\colon x\to y$ is an arrow
for which there exists a candidate inverse $f'\colon y\to x$
and isomorphisms $\id_x\xrightarrow{\sim}f'f$ and $ff'\xrightarrow{\sim} \id_y$
in the $(\infty,d)$-categories $C(x,x)$ and $C(y,y)$, respectively --- %
a notion that is already defined by induction.
A priori, isomorphism is then a new notion of sameness between objects that is distinct
from the one inbuilt in the anima $C^\simeq$;
to address this one usually imposes the \emph{univalence} condition
that forces the two\footnote{
  For $d\geq 2$, in the definition of isomorphism
  one could talk about an isomorphism
  $\id_x\xrightarrow{\sim}f'f$, an identification $\id_x\simeq f'f$,
  or any of the other inductively defined $d-2$ notions of sameness in between;
  in total this yields $d$ new notions of sameness
  that without univalence are all distinct from the one inbuilt in $C^\simeq$.
} notions to agree.

While $(\infty,d)$-categories can be described in many different ways
--- including
globular models \cite{rezkTheta,AraMS},
enriched models \cite{br1,br2},
simplicial models \cite{VerityComplicialI,EmilyNotes,or},
multisimplicial models \cite{BarwickThesis},
and cubical models \cite{CKM}, \cite{DKM}
--- it is well understood that
these approaches are equivalent, in the sense that they give rise to equivalent
$(\infty,1)$-categories; various comparisons are provided in
\cite{BSP,br1,br2,DKM,Loubaton4}.

\subsection{\texorpdfstring{$(\infty,\infty)$}{(infinity, infinity)}-categories via cores or localizations}

In recent years, the case $d=\infty$ has gained some attention,
for example in Masuda's work on categorical spectra~\cite{Masuda-thesis}
or in the theory of Gestalten of Scholze and Stefanich
as presented in \cite{ScholzeGestalten}.

There is a very natural way for how one would hope to describe the data of an
$(\infty,\infty)$-category: it consists of
\begin{itemize}
\item
  an $(\infty,0)$-category $C_0$, i.e., an anima, and
\item
  for each $d\geq 0$,
  the data that upgrades $C_{d}$ to an $(\infty,d+1)$-category $C_{d+1}$.
\end{itemize}

Formally, this amounts to saying that the $(\infty,1)$-category $\Cat_{(\infty,\infty)}$
of $(\infty,\infty)$-categories is the limit of a sequence
\begin{equation}
  \dots \to \Cat_{(\infty,d+1)}
  \xrightarrow{T_d} \Cat_{(\infty,d)}\ \dots \to \Cat_{(\infty,1)}\xrightarrow{T_0} \Cat_{(\infty,0)}=\An.
\end{equation}
So an $(\infty,\infty)$-category ``is'' a sequence
\begin{equation}
  D_0=C_0\in \An,
  \quad
  \dots
  \quad
  D_{d+1} \in \fib_{C_d}(T_d),
  \quad
  \dots
\end{equation}
where we think of objects $D_{d+1}$ of the fibers $\fib_{C_d}(T_d)$
of the functor $T_d$ as the possible ``upgrades''
of a given $(\infty,d)$-category $C_d$ to an $(\infty,d+1)$-category
$C_{d+1}=(C_d,D_{d+1})$.

It turns out that there are not only one but two universal choices for the functors $T_d$,
namely the left and right adjoint
of the fully faithful inclusion $I_d\colon \Cat_{(\infty,d)}\hookrightarrow\Cat_{(\infty,d+1)}$:
\begin{itemize}
\item
  The left adjoint $L_d$ is the $d$-localization functor
  that acts on an $(\infty,d+1)$-category by formally inverting
  all arrows of dimension $(d+1)$.
\item
  The right adjoint $R_d$ is the $d$-core functor
  that acts on an $(\infty,d+1)$-category
  by removing all non-invertible arrows of dimension $(d+1)$.
\end{itemize}

Choosing $T_\bullet\coloneqq L_\bullet$ or $T_\bullet\coloneqq R_\bullet$ yields two\footnote{
  There are also more esoteric limit categories obtained by mixing
  and matching $T_d\in\{L_d,R_d\}$ separately for each $d\in \naturals$,
  but even our appetite for abstract nonsense has its limits.
}
possible limiting $(\infty,1)$-categories
\begin{equation}
  \Cat_{(\infty,\infty)}^L
  \quad\text{and}\quad
  \Cat_{(\infty,\infty)}^R;
\end{equation}
the main goal of this article is to clarify the relationship between them.
For completely formal reasons
(only due to the fact that $R_d$ and $L_d$ are the
two adjoints of a common full embedding $I_d$;
see \Cref{lem:biinductive-adjunction}),
there is an adjunction between them;
we prove that the right adjoint is fully faithful, yielding our main theorem:

\begin{theorem-intro} [\Cref{thm:main}]
  There is a reflective localization
  \begin{equation}
    \begin{tikzcd}
      L\colon \Cat_{(\infty,\infty)}^R\ar[r,bend left=20]
      \ar[r,hookleftarrow]
      \ar[r,phantom,bend left =10, "\bot"]
      & \Cat_{(\infty,\infty)}^L : R
    \end{tikzcd}
  \end{equation}
  that inverts precisely the weakly $\infty$-surjective maps.
\end{theorem-intro}

As has long been folklore knowledge, the left adjoint $L$
is not fully faithful:
for example the $(\infty,\infty)$-category of spans (of animae, say)
becomes trivial under $L$ (see \Cref{ex:spans})
and the $(\infty,\infty)$-category of cobordisms
is collapsed to its $(\infty,0)$-localization
(just by virtue of being fully dualizable; see \Cref{cor:full-dualizable-coind-collapse}). We note that in our description via (weak) $\infty$-surjectivity
the first fact becomes particularly tautological.
In separate work~\cite{ORW2, ORW3},
we also provide an explicit proof in the complicial model
that the nerve of any walking coinductive equivalence\footnote{
  Which we would call walking coinductive \emph{isomorphism}
  following the conventions in this paper;
  see \Cref{def:walking-coind}.
} endowed
with the Roberts-Street marking is non-trivial in $\Cat_{(\infty,\infty)}^R$,
even though it becomes trivial in $\Cat_{(\infty,\infty)}^L$ after applying $L$.

Since the right adjoint is fully faithful but the left adjoint is not,
it is natural to ask what the exact place is,
where the symmetry between the two choices $L_\bullet$ and $R_\bullet$ breaks.
We answer this question by proving our main theorem 
in an abstract axiomatic setting meant to capture only the essential features
of the core and localization functors that make the proof work.
The key ingredient is that of \emph{$n$-surjective maps}\footnote{
  This is not a novel concept and has been studied before,
  for example in \cite{Loubaton-effectivity} and \cite {LMRSW}.
} (see \Cref{sec:surjectivity}),
which generalizes the notion of $n$-connected maps from animae to $(\infty,d)$-categories
(where $n\in \naturals$ is a number independent from the categorical dimension $d$).
More precisely, we study exactly how $n$-surjectivity interacts
with cores and localizations,
as well as some other properties such as countable compositions
and cancellation laws.
We axiomatize these properties resulting in what we call a \emph{bias}\footnote{%
  The name is meant to evoke a breaking of symmetry between the left and the right adjoints.
}
on an abstract system of full embeddings with both left and right adjoints;
see \Cref{def:bias}.
We abstractly define the notion of \emph{weakly $\infty$-surjective maps}
(see \Cref{def:S} and \Cref{def:infty-surjective})
and prove that these ingredient alone suffice to exhibit
the desired reflective localization;
see \Cref{thm:bias-localization}.
While it seems unlikely that our exact axiomatic setup
can be reused elsewhere in the very same form,
we believe that it adds significant clarity to the overall argument
which might then be adapted to other related settings.

\begin{remark}
  Our main localization theorem
  was also obtained simultaneously by Gepner and Heine~\cite[Theorem~1.10.4]{GH-hyper},
  who even consider a Gray-enriched enhancement.
  To ease the comparison with their work,
  we provide a short dictionary of terminology in
  \Cref{rem:GH-hyper} below.
\end{remark}

\subsection{External localizations via internal invertibility}

While our main theorem is ultimately a statement about
$(\infty,d)$-categories (possibly with $d\to \infty$),
it mostly treats each $\Cat_{(\infty,d)}$ as an abstract $(\infty,1)$-category
with very little regard to the fact that each object
$C\in \Cat_{(\infty,d)}$
(or in the limit $\Cat_{(\infty,\infty)}$, be it left or right)
has itself a rich internal categorical structure.
While the notions of $n$-surjectivity, $d$-core $R_d$ and $d$-localization $L_d$
are ultimately defined in terms of this internal structure,
only their abstract properties and interactions are actually used.

In \Cref{sec:strong} we instead study localizations/subcategories of the category
$\Cat_{(\infty,\infty)}^R$
that arise from \emph{internal} notions of invertibility,
i.e., notions defined in terms of objects and arrows (of various dimensions)
\emph{within} individual $(\infty,\infty)$-categories $C$.
We then compare them to the prior external characterizations
that only speak of cores, localization and surjectivity abstractly.

For example, one defines an arrow $f\colon x\to y$ in $C$
(where $C\in \Cat_{(\infty,\infty)}^R$)
to be a \emph{coinductive isomorphism}
if there exists a candidate inverse\footnote{
  This is a slight oversimplification:
  to avoid having to specify an infinite amount of ``triangle-arrows''
  (like triangle identities, but not invertible),
  it is actually more convenient to ask for separate left and right
  inverses at each stage;
  see \Cref{constr:standard-E} and \Cref{prop:coinductive-equiv-char} for details.
} $f^1\colon y\to x$
that is inverse to $f$ up higher arrows
$f^1_+\colon \id_x\to f'f$ and $f^1_-\colon ff'\to \id_y$
that are themselves coinductively invertible,
in the sense that they in turn have candidate inverses
$f^2_+\colon f'f\to \id_x$ and $f^2_-\colon \id_y\to ff'$, respectively,
that are inverses up to higher arrows $f^2_{++}, f^2_{+-}$ and $f^2_{-+},f^2_{--}$,
respectively,
and that these in turn are coinductively invertible,
etc.\ all the way to infinity --- %
yielding in summary an infinite tower $(f^n_\sigma)_{n\in \naturals,\sigma}$
of ``inverses'' none of which is necessarily an actual identification;
see also \cite{ORsurvey,HLOR}
for the same notion in the context of \emph{strict} $(\infty,\infty)$-categories.
Similarly, an \emph{$\omega$-isomorphism} is one that for each
fixed finite height $n$ admits a tower of ``inverses''
$(f^i_\sigma)_{i\leq n,\sigma}$ of height $n$,
but not necessarily a single consistent tower that goes all the way to infinity;
see \Cref{def:alpha-invertible}.

Following a suggestion of Loubaton we then show
that the abstract localization of $\Cat_{(\infty,\infty)}^R$
at the $\infty$-surjective maps
yields the full subcategory of coinductively complete
$(\infty,\infty)$-categories,
i.e., those whose only coinductive isomorphisms are the actual isomorphisms
(see \Cref{cor:strong});
this settles \cite[Conjecture~1.2.14]{Loubaton-effectivity}.
Inspired by a counterexample of Henry--Loubaton we show that
the reflective subcategory
$\Cat_{(\infty,\infty)}^L\subseteq \Cat_{(\infty,\infty)}^R$
is contained in the proper full subcategory
$\Cat^{\omega}_{(\infty,\infty)}\subsetneq \Cat_{(\infty,\infty)}^\coind$
of $\omega$-complete $(\infty,\infty)$-categories
(see \Cref{prop:coind->weak-coind}),
which are those that have no non-invertible $\omega$-isomorphisms --- %
a condition that is genuinely stronger than having no non-invertible
coinductive isomorphisms.

It has been conjectured that the inclusion
$\Cat^L_{(\infty,\infty)}\subseteq\Cat_{(\infty,\infty)}^{\omega}$
is an equivalence.
We resolve this conjecture in the negative in \Cref{lem:V-W-L-not-omega}.
Finally, we give a different characterizion of the maps inverted by the localization onto
$\Cat_{(\infty,\infty)}^L$
as those that for each fixed finite $n\in\naturals$
are $\infty$-surjective up $n$-isomorphisms\footnote{
  The $n$-isomorphisms are the arrows that only admit a tower $(f^i_\sigma)_{i\leq n,\sigma}$
  of height $n$ but not necessarily any taller one;
  see \Cref{def:alpha-invertible}.
}.

The following table summarizes some of our findings:
\begin{center}
  \begin{tabular}{c||c|c|c}
    reflective subcats of $\Cat_{(\infty,\infty)}^R$
    &
      $\supsetneq$\footnotemark$\Cat_{(\infty,\infty)}^\coind$
    &
      $\supsetneq$\footnotemark$\Cat_{(\infty,\infty)}^\omega$
    &
      $\supsetneq$\footnotemark$\Cat_{(\infty,\infty)}^L$
    \\
    \hline
    \hline
    internal completeness
    &
      coinductive isos
    &
      $\omega$-isos
    &
    \\
    with respect to
    &
      (\Cref{def-coind-compl})
    &
      (\Cref{def:alpha-invertible})
    &
      ?
    \\
    \hline
    equivalences are exactly
    &
      coinductive isos
    &
    &
      $n$-isos $\forall n\in\naturals$
    \\
    the $\infty$-surjections up to
    &
      (\Cref{prop:L-coind-equivalences})
    &
      ?
    &
      (\Cref{lem:L-iso-coind-inv})
    \\
    \hline
    abstract localization at
    &
      $\infty$-surjections
    &
    &
      weak $\infty$-surjections
    \\
    &
      (\Cref{cor:strong})
    &
      ?
    &
      (\Cref{thm:main})
  \end{tabular}
\end{center}
\footnotetext{
  for ``$\neq$'',
  see \Cref{ex:spans}, \Cref{cor:E-not-coind} or \Cref{ex:cobordisms}
}
\footnotetext{
  for ``$\neq$'', see \Cref{ex:HL}
}
\footnotetext{
  for ``$\neq$'', see \Cref{lem:V-W-L-not-omega}
}

\subsection{Acknowledgements}

We thank F\'elix Loubaton for suggesting the connection between
$\infty$-surjectivity and coinductive isomorphisms 
and for his detailed explanation of the fundamental discrepancy
with $\omega$-isomorphisms.
Further thanks go to
Thomas Blom%
, %
Denis-Charles Cisinski%
{} and
Thomas Nikolaus%
, who provided very useful feedback in an early stage of this project.
T.W.\ thanks Lyne Moser for many conversations on cell-attachments,
pushouts and cores,
even though most of the resulting insights ended up
not making it into the final version of the paper.
M.R.\ thanks Clark Barwick for very valuable conversations on the subject
at the Isaac Newton Institute for Mathematical Sciences, Cambridge,
during the programme \emph{Equivariant homotopy theory in context}
(EPSRC grant EP/Z000580/1).
T.W.\ joined this collaboration with V.O.\ and M.R.\
after many exciting conversations at the 2025 summer school and workshop
``Higher Structures: Recent Developments and Applications''
in Hamburg;
we thank the organizers of that event for bringing us together.
M.R.\ is grateful for support from the National Science Foundation under Grant No.~DMS-2203915.

\subsection{AI use disclosure}
The key idea behind \Cref{lem:V-W-L-not-omega} was suggested by GPT-6,
based on the authors' construction of $V$ in an earlier version of the paper.

All text in this article (including \Cref{lem:V-W-L-not-omega}) is written by the authors.

\section{Preliminaries}

\subsection{Animae and categories}

In this paper we work fully within the realm of \((\infty,1)\)-categories,
which we just call \emph{categories}.
Everything we do is independent from the specific implementation of this theory.
Similarly,
we just write \emph{$d$-categories} rather than ``$(\infty,d)$-categories'';
see \Cref{sec:d-cats} for more extensive preliminaries on $d$-categories.
In the limit, \emph{$\infty$-category} for us means ``$(\infty,\infty)$-category'',
(and \emph{not} ``$(\infty,1)$-category '',
as would be the more common convention after Lurie);
see \Cref{def:infty-cat}.

\begin{itemize}
\item
  We denote by \(\Cat\) the category of categories
  and by \(\An\subset \Cat\) the full subcategory of \emph{animae}
  (a.k.a.\ \emph{spaces}, \emph{homotopy types} or \emph{groupoids}).
\item
  The inclusion \(\An\hookrightarrow \Cat\)
  has both a right and a left adjoint,
  namely the \emph{core} functor \(C\mapsto C^\simeq\)
  and the \emph{localization} (a.k.a.\ \emph{classifying space})
  functor \(C\mapsto |C|\), respectively.
  We also call \(C^\simeq\) the \emph{anima of objects} of \(C\).
\item
  Every anima \(A\) has \emph{objects}
  (a.k.a.\ \emph{points}, \emph{elements}, \emph{inhabitants} or \emph{terms}) \(a\in A\).\footnote{
    We use the symbol ``$\in$'' as a typing judgment,
    not as set-theoretic membership.
  }
  Between any two objects \(a,b\in A\) we speak of \emph{identifications}
  (a.k.a.\ \emph{paths} or \emph{equalities})
  \(f\colon a\simeq b\).
  For every object \(a\in A\) we have an identification with itself
  \(\id_a\colon a\simeq a\),
  called \emph{identity} (a.k.a.\ \emph{reflexivity}).
\item
  The \(0\)-truncated (a.k.a.\ \emph{discrete} or \emph{static})
  animae are called \emph{sets}\footnote{
    This means that between two objects $a,b\in A$ in a set,
    if there exists an identification $a\simeq b$,
    then this identification is unique.
  };
  the full inclusion \(\Set\hookrightarrow \An\)
  has a left adjoint \(\pi_0\colon \An\to \Set\)
  (but no right adjoint\footnote{
    An anima does \emph{not}
    have a well defined ``underlying set of objects''.
  }).
  The $(-1)$-truncated animae are called \emph{propositions}
  (a.k.a.\ \emph{subterminal});
  the left adjoint to the full inclusion $\Prop\hookrightarrow \An$
  is the \emph{propositional truncation}\footnote{
    Assuming the law of excluded middle,
    every proposition is either $\emptyset$ (false) or $*$ (true);
    then we have $||A||_{-1}\simeq \emptyset$ for $A\simeq \emptyset$,
    and $||A||_{-1}\simeq *$ otherwise.
  } 
  \(||{-}||_{-1}\).
\item
  Every existence and uniqueness statement
  is to be interpreted homotopically and constructively;
  we use language from homotopy type theory to reflect this:
  \begin{itemize}
  \item
    To ``have'' or ``construct'' a certain thing
    (e.g., an identification \(a\simeq b\)),
    means to provide an object \(t\in T\),
    where \(T\) is the anima of those things
    (e.g., the anima of identifications \(a\simeq b\)).
    In particular it needs to be clear from context which anima $T$ is meant.
  \item
    Every construction of a thing $s(t)\in S(t)$
    that depends on an object $t\in T$
    amounts to a section of the associated fibration $S\to T$
    and automatically preserves identifications, i.e., induces an assignment 
    $(f\colon t\simeq t') \mapsto (s(f)\colon s(t)\simeq s(t'))$%
    \footnote{
      One could call this behavior ``functoriality''.
      But since between animae \emph{every} well-defined assignment
      on objects is automatically ``functorial'' in this sense,
      we reserve that word for the setting of categories.
    }.
  \item
    To say that something ``is ...'' 
    (e.g., a square commutative, an assignment functorial, or a category monoidal),
    always means that a certain data is understood that exhibits it as such,
    even if it is not explicitly spelled out or recorded in the notation.
  \item
    When we say that a certain thing ``exists''%
    \footnote{
      Unlike the case of actual ``constructions'',
      if there merely exists an object in each $S(t)$
      then this does not imply that there merely exists a section of $S\to T$.
      The axiom of choice says that this is true when
      $T$ is a set,
      but even in the presence of choice this fails for general $T$.
    } (or ``merely exists'', for emphasis),
    we mean that we have an object \([t]\in ||T||_{-1}\),
    but not necessarily an object \(t \in T\).
  \item
    If a certain thing merely exists (i.e., we have $[t]\in ||T||_{-1}$)
    then we are allowed to ``choose'' a witness
    (i.e., assume that we actually have $t\in T$),
    but only for the purpose of proving a proposition $P$
    (i.e., to produce an object $p\in P$).
    Note that this is just the universal property of the propositional truncation
    and is a fully constructive operation;
    despite the word, it does not use (any version of) the axiom of choice.
  \item
    A certain thing ``exists uniquely''
    if the anima \(T\) is trivial (a.k.a.\ contractible).
    In this case we always have a \(t \in T\)
    and for any two \(t,t'\in T\)
    there exists a unique identification \(t\simeq t'\)
    (again in the same sense).
  \item
    The statement in a ``Theorem''-environment
    (or ``Lemma'', ``Proposition'', ``Corollary'', etc.)
    always amounts to the specification of an anima $T$,
    even though usually $T$ will be described in natural language rather than formulas.
    Establishing the theorem then amounts to producing an object $t\in T$;
    later references to the theorem always refer to that specific inhabitant
    (and not just the mere existence statement $[t]\in||T||_{-1}$).
    In many cases $T$ is actually a proposition,
    so that one can forget about the specific object $t\in T$
    (since any two objects are uniquely identified);
    we reserve the use of the ``Proof''-environment for this case.
    When $T$ is not (or not known a priori to be) a proposition,
    we use the ``Construction''-environment instead.\footnote{
      In classical mathematics, proofs and constructions
      are fundamentally different operations that live in different layers
      of the logical system.
      For us, the former is just a special case of the latter.
    }
  \end{itemize}
\item
  By definition,
  the \emph{objects} of a category \(C\) are just those of its core \(C^\simeq\);
  we write \(x\in C\) and \(x\in C^\simeq\) interchangeably.
\item
  For each pair of objects \(x,y\in C\), we denote by \(C(x,y)\)
  or \(\Map(x,y)\) the associated \emph{hom-anima};
  its objects are interchangeably called \emph{arrows}, \emph{maps} or \emph{morphisms}
  and written \(f \colon x\to y\).
  We denote the identity arrows by \(\id_x\colon x\to x\)
  and composition of arrows by juxtaposition.
\item
  If we want to emphasize that an assignment
  $(x\in C)\mapsto F(x)$
  on objects is actually \emph{functorial}\footnote{
  Given an assignment on objects (i.e., a functor between cores),
  functoriality is of course not a property,
  but additional structure whose presence must be justified,
  usually via some universal construction.},
  we say that we have $F(x)$ ``functorially for each $x:C$''
  (rather than ``for each $x\in C$'').
  Since this terminology can potentially be confusing in natural language,
  we almost exclusively use it in formulas such as
  ``$\lim_{x:C}F(x) \in D$''\footnote{
    Note that $\lim_{x\in C}F(x)$ would have a different meaning,
    namely the limit of the restriction $C^\simeq \to C\xrightarrow{F}D$.
  },
  and only if the implicit functor $F\colon C\to D$
  is clear from the context.
  When $C$ is an anima,
  the expressions ``$x\in C$'' and ``$x:C$'' are synonymous
  and the adverb ``functorially'' is vacuous.
\item
  An arrow \(f \colon x\to y\) in \(C\) is called \emph{an isomorphism}
  (a.k.a.\ \emph{invertible})
  if it has both a left and a right inverse,
  i.e.\ arrows \(g_+,g_-\colon y\to x\) and identifications
  \(\id_x\simeq g_+f\) and \(fg_-\simeq \id_y\).
  If such data exists, then it exists uniquely;
  thus we are allowed to omit it from the notation and just write
  \(f\colon x\xrightarrow{\sim} y\).
  One then also has a unique coherent two-sided inverse,
  which we denote by \(f^{-1}\colon y\xrightarrow{\sim} x\).
\item
  Every identification \(f\colon x\simeq y\) of objects 
  gives rise to an isomorphism \(f\colon x\xrightarrow{\sim} y\),
  which we denote with the same name.
\item
  All categories are univalent, which means
  that conversely every isomorphism \(x\xrightarrow{\sim} y\)
  arises from a unique identification \(x\simeq y\).
\item
  The animae are precisely those categories in which all arrows are invertible.
  Thus for any anima \(A\) we may use the notation \(A(x,y)\) interchangeably
  for the anima of identifications \(x\simeq y\),
  arrows \(x\to y\), or isomorphisms \(x\xrightarrow{\sim}y\).
\item
  A map \(f\colon X\to Y\) of animae is called \emph{surjective}
  (a.k.a.\ \emph{effective epimorphism}\footnote{
    Warning: Despite the name, surjective maps need not be epimorphisms at all --- %
    the latter is a much, much stronger condition.
  })
  if for every \(y\in Y\) there merely exists an \(x\in X\)
  and an identification \(fx\simeq y\),
  or equivalently, if \(\pi_0(f)\colon \pi_0(X)\to \pi_0(Y)\)
  is a surjective map of sets.
\item
  Invertible maps $F\colon C\to D$
  between animae or (higher) categories
  are interchangeably called \emph{isomorphisms}
  or \emph{equivalences}.
  We use the former when we think of $\An$ or $\Cat$
  (or later $\Cat_d$ for $d\in \naturals$, or (any version of) $\Cat_\infty$)
  as abstract categories
  and $F\colon C\to D$ as an abstract arrow within it;
  the latter is used when we think of $C$ and $D$
  as animae or (higher) categories in their own right\footnote{
    From the perspective of type theory
    this corresponds to viewing an anima $X$
    either as a term $X : U$ of the universe
    or as a type $X$ which itself can have terms $x:X$.
    The correspondence between the external notion of equivalence
    and the internal notion of isomorphism
    is (a version of) Voevodsky's axiom of univalence.
  } (with all the objects and arrows that this entails)
  and $F$ a functor between them.
\item
  A functor \(f\colon C\to D\) of categories (or animae)
  is an equivalence if and only if it is
  \begin{itemize}
  \item
    \emph{surjective on objects} (a.k.a.\ \emph{essentially surjective}),
    i.e., if \(f^\simeq\colon C^\simeq\to D^\simeq\) is surjective,
  \item
    and \emph{fully faithful},
    i.e.,
    for all \(x,y \in C\),
    the induced map \(C(x,y)\to D(fx,fy)\) is an equivalence of animae.
  \end{itemize}
\item
  There is a unique non-trivial involution \((-)^\op\colon \Cat\to \Cat\),
  which is uniquely trivial on animae;
  hence we have \((C^\op)^\simeq\simeq C^\simeq\).
  Under this identification (which we always leave unnamed),
  we have \(C^\op(x,y)\simeq C(y,x)\) with reversed composition.
\item
  The map \(C(-,-)\colon C^\simeq\times C^\simeq \to \An\)
  extends to the \emph{hom-functor}
  \begin{equation}
    C(-,-)\colon C^\op\times C\to \An.
  \end{equation}
\item
  We denote the terminal proposition/anima/category
  (a.k.a.\ \emph{singleton} or \emph{true})
  by \(*\)
  and the initial one (a.k.a.\ \emph{empty} or \emph{false})
  by $\emptyset$.
\item
  We denote by \(\omega=\{0\to 1\to \cdots\}\)
  the well-ordered set of finite ordinals,
  viewed as a (locally discrete) category.
  Its core is the set \(\naturals\) of natural numbers.
\item
  The simplex category $\Delta$ is the (locally discrete) category
  of finite non-empty totally ordered sets,
  viewed as a full subcategory $\Delta\subset\Cat$.
  Its objects are the categories
  $[n]\coloneqq \{0\to \dots \to n\}$
  for each natural number $n\in \naturals$.
\item
  The restricted Yoneda embedding
  $\Cat\hookrightarrow \Fun(\Delta^\op,\An)$,
  is fully faithful\footnote{
    For example, see \cite{Hebestreit-Steinebrunner}
    for a short model-independent proof.}.
  and allows us to identify categories with complete Segal animae.
\item
  We also write
  $|{-}|\colon\Fun(\Delta^\op,\An)\to \An$ for the colimit functor;
  when restricted to $\Cat$,
  this is compatible with the previous notation
  since both functors are right adjoint to the constant-diagram embedding.
\item
  For any simplicial object $X\colon\Delta^\op\to D$
  with values in a category $D$ with limits,
  we also denote by the same name
  its right Kan extension $X\colon \Fun(\Delta^\op,\An)^\op\to D$
  along (the opposite of) the Yoneda embedding of $\Delta$;
  in particular we can evaluate $X$ on all simplicial sets
  by the usual limit formula $X(K)\coloneqq \lim_{\Delta^i\to K}X_i$.
  For example, when $D\coloneqq\An$ and $C$ is a category viewed as a simplicial anima,
  we may write
  $C(\Delta^0)\simeq C^\simeq$ and $C(\Delta^1)\simeq \Mor C$
  for the animae of objects and arrows of $C$,
  respectively.
\end{itemize}

\begin{remark}
  This paper is written in a style that directly slots into any
  (sufficiently developed) synthetic theory of $(\infty,1)$-categories.
  Thus it yields interpretations not just in classical mathematics,
  but also in parameterized or internal category theory
  in the sense of Martini--Wolf~\cite{Martini-Wolf}.

  To obtain a well-behaved theory of $n$-surjectivity for $d$-categories
  and prove the main theorem,
  we require the ambient theory of animae to satisfy the following:
  \begin{itemize}
  \item
    Hypercompleteness:
    every $\infty$-connected map of animae is an equivalence.
  \item
    Axiom of dependent choice:
    surjective maps of animae are closed under countable composition.
  \end{itemize}

  In \Cref{sec:strong} we further need:
  \begin{itemize}
  \item
    Axiom of choice:
    for every surjective map from an anima onto a set, there merely exists a section.
  \end{itemize}
  We make use of choice in two ways:
  first to establish good lifting properties of categorical cell complexes
  and later in the key computation of the coinductive completion
  (\Cref{prop:localization-coind-approx}).
  The first use of unrestricted choice could easily be avoided
  (and replaced with dependent choice)
  by just considering countable cell complexes,
  as these are the only ones that show up in article;
  the second one is more substantial and we do not know of an easy workaround to avoid it.

  While we have attempted to make this paper as self-contained as possible
  (especially regarding the main theorem),
  some preliminaries were nonetheless necessary.
  For transparence, we list here the main non-trivial inputs that
  go beyond some general background theory of categories:
  \begin{itemize}
  \item
    When working with $d$-categories,
    we use the operadic theory of enriched categories~\cite{GH} and/or
    the related theory of absolute distributors
    and complete Segal objects therein~\cite{Haugseng-rectification}.
    While the original papers are written in classical models,
    it seems likely that the whole theory can be ported to a synthetic setting
    without changing its essence.
  \item
    In the section about surjectivity and localizations,
    we make use of a basic lemma
    about connectedness and simplicial anima (\Cref{lem:connected-realization}).
    The only reference we could find for this lemma is \cite{ERW}
    where it is proved in the topological model.
    Again, it seems likely that a purely synthetic proof
    of this lemma should be possible,
    but we are not aware of one in the literature.
  \item
    Since \Cref{sec:strong} is not the main focus of the paper,
    we allow ourselves to be a little looser regarding the use
    of background assumptions.
    Most prominently
    we now use some more advanced results from
    \cite{Loubaton-effectivity} or \cite{LMRSW},
    such as the ($n$-surjective, $(n+1)$-fully faithful) factorization system.
  \end{itemize}
\end{remark}

\begin{remark}
  \label{rem:GH-hyper}
  To ease the comparison with the parallel work of Gepner and Heine~\cite{GH-hyper},
  we give a short translation between our terminology and theirs:
  \begin{itemize}
  \item
    They call ``univalent $\infty$-category'' what we call ``right $\infty$-category''
    (see \Cref{def:infty-cat})
    or just ``$\infty$-category'' in \Cref{sec:strong}.
  \item
    What they call ``hypercomplete'' is equivalent to ``coinductively complete''
    in our sense (see \Cref{def-coind-compl})
  \item
    They call ``$n$-connected'' (or ``$(n+1)$-connective'')
    what we call ``$(n+1)$-surjective''
    (see \Cref{def:surj}).
  \item
    Their ``Postnikov-complete'' (univalent) $\infty$-categories
    are our ``left'' $\infty$-categories, i.e.,
    those in the image of the right adjoint
    $R\colon \Cat^L_{\infty}\to \Cat^R_\infty$
    of \Cref{thm:main}.
  \end{itemize}
\end{remark}

\subsection{Enriched categories}

We recall some basic features of enriched category theory,
following Gepner--Haugseng~\cite{GH}.

They construct a functor\footnote{
  Gepner--Haugseng denote this functor by
  \(L_{\mathrm{gen}}\mathbb{\Delta}_{-}^\op\), reflecting its construction.
  But since we use it more or less as a black box,
  we use a simpler name instead.
} \(O(-)\colon \An\to \Opd\)
which to each anima \(A\) associates a (colored, non-symmetric)
operad \(O(A)\) given as follows:
\begin{itemize}
\item
  The colors/objects of $O(A)$ are the pairs $(x,y)\in A\times A$.
\item
  Given input colors $(x_0,y_1), \dots, (x_{n-1},y_n)$ and an output color $(y_0,x_n)$,
  the anima of $n$-ary operations between them is
  \begin{equation}
    O(A) ((x_0,y_1), \dots, (x_{n-1},y_{n}) ; (y_0,x_n))
    \coloneqq
    \prod_{i=0}^n A(y_i,x_i);
  \end{equation}
  the operadic composition law is induced
  by composing identifications of $A$.
\end{itemize}
In particular, for each \(x,y,z\in A\),
it has distinguished binary operations
\begin{equation}
  \mu_{x,y,z}\coloneqq(\id_x,\id_y,\id_z)\in O(A)((x,y) , (y,z) ; (x,z) )
  =A(x,x)\times A(y,y)\times A(z,z)
\end{equation}
and zero-ary operations
\begin{equation}
  \epsilon_x\coloneqq (\id_x)\in O(A)(\,\,; (x,x))=A(x,x)
\end{equation}

Let \((V,\otimes)\) be a monoidal category.
Using these operads \(O(A)\) one defines the following notions:

\begin{definition}[\cite{GH}, Definition~4.3.1]
  \begin{itemize}
  \item[]
  \item
    An \emph{$A$-flagged $V$-category}
    (called a \emph{categorical $A$-algebra} by Gepner--Haugseng)
    is an algebra in
    \((V,\otimes)\) for the operad \(O(A)\).
    We denote by \(\flCat[A]{V}\coloneqq \Alg(O(A); (V,\otimes))\)
    the category of $A$-flagged $V$-categories.
  \item
    The functor
    \begin{equation}
      \An\xrightarrow{O(-)}\Opd\xrightarrow{\Alg(-; (V,\otimes))} \Cat^\op
      \quad A\mapsto \flCat[A]{V}
    \end{equation}
    that sends each anima $A$ to the category of $A$-flagged $V$-categories
    has a cartesian unstraightening that we denote by
    \begin{equation}
      (-)^\simeq \colon \flCat{V}\to \An;
    \end{equation}
    We call its total category \(\flCat{V}\)
    the category of \emph{flagged $V$-categories}.
  \end{itemize}
\end{definition}

\begin{remark}
  \begin{itemize}
  \item[]
  \item
    Concretely an $A$-flagged $V$-category assigns to each pair \((x,y)\in A\times A\)
    an object \(D(x,y)\in V\)
    equipped with composition maps
    \begin{equation}
      D(\mu_{x,y,z})\colon D(x,y)\otimes D(y,z)\to D(x,z);
    \end{equation}
    the algebra structure then encodes coherent associativity and unitality with respect to
    \(D(\epsilon_x)\)
  \item
    A flagged $V$-category is a pair \((A,D)\)
    where \(A\) is an anima and \(D\) is an $A$-flagged $V$-category.
  \item
    A morphism \((A,D)\to (B,E)\) consists of a map \(f\colon A\to B\) of animae
    and a map of $O(A)$-algebras \(D\to f^*(E)\);
    the latter ``just'' amounts to a map
    $D(x,y)\to E(fx,fy)$ in \(V\)
    for each \((x,y) \in A\times A\),
    coherently compatible with the composition and identities.
  \end{itemize}
\end{remark}

Let us now consider the case \(V=\An\)
(always equipped with its cartesian monoidal structure).
\begin{definition}
  Let \((A,D)\) be a flagged \(\An\)-category.
  \begin{itemize}
  \item
    For each \(x,y\in A\) one defines the subanima
    \(D^\inv(x,y)\subseteq D(x,y)\) of the isomorphisms
    (those which admit a left and a right inverse
    with respect to the composition maps).
  \item
    One says that \((A,D)\) is \emph{complete} (or \emph{univalent})
    if for each \(x,y\in A\), the map
    \begin{equation}
      A(x,y)\to D^\inv(x,y),
      \quad
      f\mapsto D(\id^A_x,f)(\id^D_x)
    \end{equation}
    is an equivalence of animae.
    Here $\id^A_x\in A(x,x)$ is the identity of $x\in A$ 
    and \(\id^D_x\in D(x,x)\) is the object picked out by the map
    \(D(\epsilon_x)\colon *\to D(x,x)\);
    it is not hard to show that each \(D(\id_x,f)(\id^D_x)\in D(x,y)\)
    actually lives in \(D^\inv(x,y)\).
  \end{itemize}
\end{definition}

\begin{theorem}
  \label{thm:Cat-An}
  \begin{enumerate}
  \item[]
  \item
    Let \(C\) be a category.
    Composition in \(C\) defines an extension of the map
    \(C(-,-)\colon C^\simeq\times C^\simeq\to \An\)
    to an \(O(C^\simeq)\)-algebra in \((\An,\times)\).
  \item
    This construction assembles into
    a fully faithful functor
    \begin{equation}
      \begin{tikzcd}
        \Cat\ar[rd,"{(-)^\simeq}"']\ar[r,hookrightarrow,dashed]& \flCat{\An}\ar[d]
        \\
        &\An
      \end{tikzcd}
      \quad C\mapsto (C^\simeq,C(-,-))
    \end{equation}
    over \(\An\),
    whose image are precisely the complete flagged \(\An\)-categories.
  \end{enumerate}
\end{theorem}

\begin{proof}
  The equivalence between categories and complete flagged $\An$-categories
  is \cite[Theorem~5.4.6]{GH}
  which in turn comes from \cite[Theorem~4.4.7]{GH}.
  One can explicitly unravel the construction of the latter
  to see that for any $\An$-category $(C^\simeq,C)$,
  the restriction $C\colon C^\simeq \times C^\simeq\to \An$
  of the algebra structure does
  indeed agree with the family of hom-animae of the corresponding category.
\end{proof}

Going forward, we always view \(\Cat\)
as a full subcategory of $\flCat{\An}$ via this construction;
for any category \(C\),
we interchangeably write \(C\) or \((C^\simeq, C)\),
depending on whether we want to emphasize its (tautological) flagging or not.

\begin{definition}
  \begin{itemize}
  \item[]
  \item 
    For any lax monoidal functor \(F\colon (V,\otimes)\to (W,\otimes)\),
    we denote by
    \begin{equation}
      \begin{tikzcd}
        \flCat{V}\ar[r,"F_*",dashed]
        \ar[dr]
        &
        \flCat{W}
        \ar[d]
        \\
        &
        \An
      \end{tikzcd}
      \quad
      (A, D) \mapsto (A, F_*D)
    \end{equation}
    the induced functor given by postcomposition of algebras,
    i.e., the unstraightening of the transformation
    \begin{equation}
      \Alg(O(-),(V,\otimes))\xrightarrow{F\circ -}\Alg(O(-),(W,\otimes))\colon
      \An^\op\to \Cat.
    \end{equation}
    The functor \(F_*\) is called \emph{change of enrichment along $F$}.
  \item
    Change of enrichment along the lax monoidal functor
    \(V(\monunit,-)\colon V\to \An\)
    is said to take any flagged $V$-category to its
    \emph{underlying} flagged $\An$-category.
  \item
    A flagged $V$-category is called \emph{complete}
    (or \emph{univalent}),
    if the same is true for its underlying flagged $\An$-category.
  \item
    A complete flagged $V$-category is simply called
    a \emph{$V$-category};
    these span a full subcategory that we denote
    \(\Cat_V\subset \flCat{V}\).
  \end{itemize}
\end{definition}

\begin{remark}
  With this notation and the identification of \Cref{thm:Cat-An}
  we have \(\Cat\simeq \Cat_\An\subseteq \flCat{\An}\).
\end{remark}

\begin{definition}
  A map \(f\colon (A,D)\to (B,E)\) of flagged $V$-categories is called
  \begin{itemize}
  \item
    \emph{fully faithful} if for all \(x,y \in A\)
    the induced map \(D(x,y)\to E(fx,fy)\) of animae is an equivalence;
  \item
    \emph{surjective on objects} if \(A\to B\) is surjective;
  \item
    \emph{essentially surjective} if for every $z \in B$
    there merely exists an object $x \in A$ and an isomorphism
    $g \in V(\monunit,E)^\inv(fx,z)$
    between $fx$ and $z$ in the underlying $\An$-flagged category of $(B,E)$;
  \item
    a \emph{complete equivalence}
    if it is fully faithful and essentially surjective.
  \end{itemize}
\end{definition}

\begin{theorem}[\cite{GH}, Theorem~5.6.6]
  \label{thm:complete-localization}
  \begin{enumerate}
  \item[]
  \item
    For every flagged $V$-category \(D\)
    there exists a unique complete equivalence
    \(D\to \widehat{D}\) to a complete flagged $V$-category.
  \item
    These maps assemble as the counit of a reflector
    \begin{equation}
      \widehat{(-)}\colon \flCat{V}\rightleftarrows \Cat_{V}
    \end{equation}
    onto the full subcategory of $V$-categories,
    which exhibits \(\Cat_V\) as a localization of \(\flCat{V}\)
    at the complete equivalences.
  \end{enumerate}
\end{theorem}

\begin{remark}
  \label{rem:isos-in-cat-V}
  For a map $f\colon D\to E$ of flagged $V$-categories,
  if the target $E$ is already complete then there is no difference between
  essential surjectivity and surjectivity on objects.
  It follows that in \Cref{thm:complete-localization}
  it would suffice to localize at the maps
  that are fully faithful and surjective on objects.
  
  Moreover, it follows directly from \Cref{thm:complete-localization}
  that the isomorphisms in $\Cat_V$ are the complete equivalences.
  Extending our convention from $\Cat$,
  we just call these isomorphisms \emph{equivalences} (of $V$-categories).
\end{remark}

For the convenience of the reader, we summarize some of the key basic facts
of enriched category theory that can be extracted from \cite{GH}.

\begin{theorem}
  \label{prop:flCat-inherits-monoidal}
  Let \((V,\otimes)\) be a symmetric monoidal category.
  \begin{enumerate}
  \item
    For each anima \(A\),
    the category \(\flCat[A]{V}\) inherits
    a  symmetric monoidal structure
    which is computed pointwise, i.e., 
    \((D\otimes E)(x,y)\simeq D(x,y)\otimes E(x,y)\).
  \item
    These symmetric monoidal structures assemble
    to a symmetric monoidal structure on \(\flCat{V}\),
    which pairwise is of the form
    \begin{equation}
      (A,D)\otimes (B,E) \simeq (A\times B, D\boxtimes E),
    \end{equation}
    with
    \begin{equation}
      (D\boxtimes E) \simeq \pr_A^*(D)\otimes \pr_B^*(E)\colon ((a,b), (a',b'))
      \mapsto D(a,a')\otimes E(b,b').
    \end{equation}
  \item
    \label{inprop:localize-symmon}
    There is a unique symmetric monoidal structure
    \(\widehat{\otimes}\) on  \(\Cat_V\)
    such that the localization functor
    \((\flCat{V},\otimes )\to (\Cat_V,\widehat{\otimes})\)
    is symmetric monoidal;
    in particular, we have
    \(C\widehat{\otimes} D \simeq \widehat{C\otimes D}\)
    naturally for all $V$-categories \(C,D\).
  \item 
    \label{inprop:localized-symmon-is-cartesian}
    If \((V,\times)\)
    carries the cartesian symmetric monoidal structure,
    then the induced symmetric monoidal structures
    on \(\flCat{V}\) and on \(\Cat_V\) are also cartesian.
  \item
    \label{inprop:presentable-Cat-V}
    If \((V,\otimes)\) is presentably symmetric monoidal,
    then the same is true for \(\flCat{V}\) and \(\Cat_V\).
  \item
    \label{inprop:coe-symmonoidal}
    Change of enrichment along any symmetric monoidal adjunction\footnote{
      This means that $F$ is symmetric monoidal and that
      the induced lax symmetric monoidal structure on $G$ is strong.
    }
    \begin{equation}
      F\colon (V,\otimes) \rightleftarrows (W,\otimes) : G
    \end{equation}
    yields a symmetric monoidal adjunction
    \begin{equation}
      F_*\colon (\flCat{V},\otimes) \rightleftarrows (\flCat{W},\otimes) : G_*.
    \end{equation}
  \end{enumerate}
\end{theorem}

\begin{lemma}
  \label{lem:enriched-adjunction-Cat}
  Let \(V,W\) be two categories with finite products,
  which we consider with the cartesian symmetric monoidal structure.
  Let \(F\colon V \rightleftarrows W : G\) be an adjunction between them.
  Assume that the left adjoint \(F\) preserves finite products
  so that the adjunction is symmetric monoidal\footnote{
    The right adjoint $G$ always preserves products.
  }.
  Then completion and change of enrichment yield an adjunction
  \begin{equation}
    \widehat{F_*}\colon \Cat_V \rightleftarrows \Cat_W : G_*,
  \end{equation}
  where the left adjoint again preserves finite products.
\end{lemma}

\begin{proof}
  The assumption that \(F\) preserves finite products means that
  the adjunction \(F\dashv G\) is uniquely symmetric monoidal
  with respect to the cartesian structures.
  Thus by \Cref{prop:flCat-inherits-monoidal}
  we have an induced adjunction \(F_*\dashv G_*\),
  where \(F_*\) preserves products.
  We consider the composed adjunction
  \begin{equation}
    \label{eq:change-of-enrichment-adjunction}
    \begin{tikzcd}
      \widehat{F_*}\colon
      &
      \flCat{V}
      \ar[r,bend left, "F_*"]
      \ar[r,phantom,"\bot"]
      &
      \flCat{W}
      \ar[r,bend left, "\widehat{(-)}"]
      \ar[r,phantom,"\bot"]
      \ar[l,bend left,"G_*"]
      &
      \Cat_{W}
      \ar[l,bend left,hookrightarrow]
      &
      : G_*
    \end{tikzcd}
  \end{equation}
  First, we will show that the composite right adjoint actually takes values in $\Cat_V$.

  Since \(F\) preserves the terminal object (which is the monoidal unit),
  we see that \(G\) induces an equivalence
  \begin{equation}
    G\colon W(\monunit, -)\simeq V(\monunit,G(-)),
  \end{equation}
  hence
  \begin{equation}
    V(\monunit,-)\circ G_* \simeq W(\monunit,-),
  \end{equation}
  which means that, for each \(D\in \flCat{W}\),
  its underlying flagged \(\An\)-category agrees with that of \(G_*(D)\).
  Since completeness is defined purely with respect
  to the underlying flagged $\An$-category,
  this means that \(G_*\) preserves and detects completeness.
  In particular, the functor
  \(G_*\colon \Cat_W\to \flCat{V}\) takes values in \(\Cat_V\subset\flCat{V}\),
  which means that the composite adjunction
  \eqref{eq:change-of-enrichment-adjunction}
  restricts to the desired adjunction
  \begin{equation}
    \widehat{F_*}\colon \Cat_V\rightleftarrows \Cat_W : G_*.
  \end{equation}

  Since both \(F_*\colon \flCat{V}\to\flCat{W}\)
  and \(\widehat{(-)}\colon \flCat{W}\to \Cat_W\)
  are symmetric monoidal
  (see \Cref{prop:flCat-inherits-monoidal},
  parts \ref{inprop:coe-symmonoidal} and \ref{inprop:localize-symmon}, respectively),
  their composite \(\widehat{F_*}\) preserves products as claimed
  because all appearing symmetric monoidal structures are the cartesian ones
  (see \Cref{prop:flCat-inherits-monoidal},
  part \ref{inprop:localized-symmon-is-cartesian}).
\end{proof}

Under suitable conditions on the enriching category $V$
(that will always be satisfied in the examples considered in this paper),
one can also describe $V$-categories as simplicial objects.
For this we recall that an \emph{absolute distributor}
is a presentable category $V$
such that the unique colimit-preserving functor $\An\to V$ sending $*$
to the terminal object
is fully faithful and satisfies certain descent properties
which we will not recall here precisely;
see \cite[Definition 7.2]{Haugseng-rectification} for details.
We always consider such a $V$ with its cartesian monoidal structure.

\begin{theorem}[\cite{Haugseng-rectification}, Theorem~7.18]
  \label{thm:V-cat-simplicial}
  If $V$ is an absolute distributor,
  then we have a fully faithful embedding
  \begin{equation}
    \Cat_V\hookrightarrow \Fun(\Delta^\op,V)
  \end{equation}
  whose essential image consists of those simplicial objects
  $X\colon \Delta^\op\to V$ satisfying the following three conditions
  \begin{itemize}
  \item
    $X_0$ is an anima, i.e., lies in the image of the full embedding $\An\hookrightarrow V$.
  \item
    $X$ is Segal, i.e.,
    $X_n\xrightarrow{\sim} X_1\times_{X_0}\dots \times_{X_0}X_1$
    (induced by the usual spine inclusion in $\Delta^n$).
  \item
    $X$ is complete, i.e.,
    $X_0\xrightarrow{\sim} {X_0}\times_{X_1}X_2\times_{X_1} X_2\times_{X_1} X_0$
    (induced by the usual simplicial presentation of the walking isomorphism).
  \end{itemize}
  Under this equivalence we can describe functorially for each $V$-category $C$
  \begin{itemize}
  \item
    the anima of objects as $C^\simeq\simeq C_0$, and
  \item
    for each $x,y\in C$ the corresponding hom-object as the fiber
    \begin{equation}
      C(x,y)\simeq \fib_{(x,y)}(C_1\to C_0\times C_0)
    \end{equation}
    (induced by the usual boundary inclusion in $\Delta^1$).
  \end{itemize}
\end{theorem}

\begin{example}
  Examples of absolute distributors $V$ are $V=\An$
  (so that \Cref{thm:V-cat-simplicial} subsumes \Cref{thm:Cat-An})
  and $V=\Cat_d$ which we will introduce momentarily
  (see also \cite[Definition~7.19]{Haugseng-rectification}).
\end{example}

\subsection{\texorpdfstring{$d$}{d}-categories}
\label{sec:d-cats}

We just say ``$d$-category'' for what is usually called an $(\infty,d)$-category.

\begin{definition}
  \label{def:cat-d}
  By induction on $d\in \naturals$ 
  we define the category of \(d\)-categories as
  \begin{equation}
    \Cat_0\coloneqq \An
    \quad\text{and}\quad
    \Cat_d\coloneqq \Cat_{\Cat_{d-1}} \text{(for $d\geq 1$)},
  \end{equation}
  where we always equip $\Cat_d$ with the cartesian symmetric monoidal structure.
\end{definition}

\begin{remark}
  \label{rem:Cat-d-presentable}
  Since $(\An,\times)$ is presentably symmetric monoidal,
  the same is true for each $(\Cat_d,\times)$;
  in other words, $\Cat_d$ is presentable and cartesian closed.
  Moreover, it follows from \Cref{thm:Cat-An} that we have \(\Cat_1\simeq \Cat\).
\end{remark}

Denote by \(\pred{d}\coloneqq \max(0,d-1)\)
the predecessor function on natural numbers,
and \(\pred[n]{d}= \max(0,d-n)\) its $n$-fold iteration.

\begin{definition}
  \label{def:n-cat-arrows}
  Let \(d,n\in \naturals\) and let \(C\) be a $d$-category.
  \begin{itemize}
  \item
    The \emph{objects} of \(C\) are the objects of \(C^\simeq\).
    We write \(x\in C^\simeq\) or \(x\in C\) interchangeably.
    Note that when $d\leq 1$ this agrees with the old definition.
  \item
    A \emph{parallel pair of $0$-arrows} $(x,y)$ in \(C\)
    is just a pair of objects \(x,y\in C\).
    In this case we have the \(\pred{d}\)-category
    $C(x,y)$
    whose objects are called the \emph{$1$-arrows} of \(C\).
  \item
    For \(n\geq 1\), a \emph{parallel pair of $n$-arrows} \((x,y)\)
    consists of two objects $s,t\in C$ --- %
    which we drop from the notation --- %
    and a parallel pair \((x,y)\) of $(n-1)$-arrows in the
    $\pred{d}$-category $C(s,t)$.

    In this case we inductively define the \(\pred[(n+1)]{d}\)-category
    \begin{equation}
      C\arlev{n+1}(x,y)\coloneqq C(s,t)\arlev{n}(x,y)
    \end{equation}
    whose objects are called \emph{$(n+1)$-arrows}
    and depicted as $x\to y$.

  \item
    It is notationally convenient to declare that a parallel pair of $(-1)$-arrows
    \(()\) amounts to no data.
    Then we set \(C\arlev0()\coloneqq C\) and call its objects
    (which are just those of $C$) \emph{$0$-arrows}.
  \item
    We write $\Par_n C$ for the anima of parallel pairs of $n$-arrows of $C$,
    and $\Mor_n C$ for the anima of $n$-arrows.
  \item
    For each $x\in \Mor_n C$ and $k\geq n$ we denote
    its $k$-dimensional identity arrow by
    $\id^k_x\coloneqq \id_{\id^{k-1}_x}\colon \id^{k-1}_x\to \id^{k-1}_x$
    starting with $\id^n_x=x$ and $\id^{n+1}_x=\id_x$.
  \item
    For each $x\in \Mor_n C$ and $k>n$ we write
    $\End^k_C(x)\coloneqq C\arlev{n}(\id^{k-1}_x,\id^{k-1}_x)^\simeq$
    for the anima of $k$-dimensional endomorphisms of $x$.
  \end{itemize}
\end{definition}

\begin{remark}
  \begin{itemize}
  \item[]
  \item
    The reason we have to use the slightly clunky predecessor operation
    is that for $d\geq 1$, parallel arrows between two given objects in a $d$-category form a $(d-1)$-category,
    but for $d=0$ they again form a $0$-category.
    One could also avoid this annoyance by defining \(\Cat_d\coloneqq \An\)
    for all \(d\leq 0\).
  \item
    Explicitly unpacking the tail-recursion in \Cref{def:n-cat-arrows}
    we see that for $n\geq 0$ an $(n+1)$-arrow $f\colon x\to y$
    amounts to a sequence
    \begin{align}
      s_0,t_0&\in C,
      \\
      s_1,t_1 &\in C(s_0,t_0),
      \\
      \dots
      \\
      x,y &\in C\arlev{n}(s_{n-1},t_{n-1})\coloneqq C(s_0,t_0)\cdots (s_{n-1},t_{n-1}),
      \\
      f&\in C\arlev{n+1} (x,y)\coloneqq C(s_0,t_0)\cdots (s_{n-1},t_{n-1})(x,y),
    \end{align}
    where we usually omit all the \((s_k,t_k)\) from the notation.
  \item
    For all \(0\leq k < n\), we call \(s_k\) the \emph{$k$-source}
    and \(t_k\) the \emph{$k$-target} of \(x, y, f\);
    we call \(x\) and \(y\)
    the \emph{source} (short for $n$-source)
    and \emph{target} (short for $n$-target) of \(f\),
    respectively.
  \item
    Two $(n+1)$-arrows $f\colon x\to y$ and $g\colon y\to z$
    can be composed in the \(\pred[n]{d}\)-category
    \(C\arlev{n}(s_{n-1},t_{n-1})\)
    (where $s_{n-1}$ and $t_{n-1}$ are the $(n-1)$-dimensional source/target of
    $x,y,z$),
    which by convention is just \(C\arlev0()=C\) itself in the case $n=0$.
  \item
    The completeness condition which distinguishes
    an enriched category from a flagged one
    precisely says that every invertible $(n+1)$-arrow
    \(f\colon x\xrightarrow{\sim} y\) arises uniquely from an identification
    \(x\simeq y\), i.e., we have
    \begin{equation}
      C\arlev n(s_{n-1},t_{n-1})^\simeq(x,y)
      \xrightarrow{\simeq}
      C\arlev{n+1}^\inv(x,y)^\simeq.
    \end{equation}
    where
    \begin{equation}
      C\arlev{n+1}^\inv(x,y)^\simeq
      \subseteq
      C\arlev{n+1}(x,y)^\simeq
    \end{equation}
    denotes the subanima of those $(n+1)$-arrows \(f\colon x\to y\) which are invertible
  \item
    Let \(n\geq d\) and \(C\) a $d$-category.
    Then \(\pred[n]{d}=0\).
    Hence for every parallel pair \((x,y)\) of ${(n-1)}$-arrows,
    \(C\arlev{n}(x,y)\) is an anima,
    which means that all \emph{its} arrows --- %
    that are precisely the \((n+1)\)-arrows of \(C\) --- are invertible.
  \end{itemize}
\end{remark}

\subsection{Localizations and cores}

\begin{theorem}
  \label{thm:L-I-R}
  For each \(d\in \naturals\) we have adjunctions
  \begin{equation}
    \label{eq:L-I-R-adjunction}
    \begin{tikzcd}
      \Cat_{d}
      \ar[r,"I^{d+1}"]
      \ar[r,bend right=15, phantom,"\bot"]
      \ar[r,bend left=35, phantom,"\bot"]
      &
      \Cat_{d+1}
      \ar[l,bend right=60, "L_{d}"']
      \ar[l,bend left=40, "R_{d}"]
    \end{tikzcd}
  \end{equation}
  such that
  \begin{itemize}
  \item
    $I^{d+1}$ is fully faithful with essential image consisting
    of those $(d+1)$-categories $C$
    in which all $(d+1)$-arrows are invertible,
  \item
    the left adjoint functor $L_n$ preserves finite products.
  \end{itemize}
\end{theorem}

\begin{proof}[Construction]
  We proceed by induction on $d$.
  \begin{itemize}
  \item
    For $d=0$,
    we take functor \(I^1\colon \An\hookrightarrow \Cat\) to just be the usual inclusion;
    $L_0=|{-}|$ and $R_0=(-)^\simeq$ are its left and right adjoint, respectively.
    It is a basic fact of category theory
    that the classifying space functor $|{-}|$ preserves finite products.\footnote{
      The usual way to see this is to compute the localization  $|C|$ of $C$
      as the geometric realization (=colimit)
      of its associated Segal simplicial anima
      and then note that $\Delta$ is sifted,
      so that geometric realizations preserve products.
    }

  \item
    For \(d\geq 1\) we apply \Cref{lem:enriched-adjunction-Cat}
    to the two adjunctions
    \begin{equation}
      \begin{tikzcd}
        \Cat_{d-1}
        \ar[r,bend right,"I^{d}"']
        \ar[r,phantom,"\bot"]
        &
        \ar[l,bend right,"L_{d-1}"']
        \Cat_{d}
      \end{tikzcd}
      \quad
      \text{and}
      \quad
      \begin{tikzcd}
        \Cat_{d-1}
        \ar[from=r,bend left,"R_{d-1}"]
        \ar[r,phantom,"\bot"]
        &
        \ar[from=l,bend left,"I^{d}"]
        \Cat_{d}
      \end{tikzcd}
    \end{equation}
    where the left adjoints $L_{d-1}$ and $I^d$ preserve products:
    the former by the induction hypothesis, the latter because it is also a right adjoint.
    This yields the desired adjunction \eqref{eq:L-I-R-adjunction}
    by setting
    \begin{equation}
      L_{d}\coloneqq \widehat{(L_{d-1})_*}
      \quad\text{,}\quad
      R_{d}\coloneqq (R_{d-1})_*
      \quad\text{and}\quad
      I^{d+1}\coloneqq \widehat{(I^{d})_* }\simeq (I^{d})_*
    \end{equation}
    and guarantees that $L_d$ preserves finite products.

    Finally, since $I^d$ is fully faithful by the inductive hypothesis,
    the same is true for $I^{d+1}=(I^d)_*$.
    The $(d+1)$-categories $C$ in its image
    are precisely those such that each pair of objects $x,y\in C$,
    the $d$-category $C(x,y)$ lies in the image of $I^d$;
    by the inductive hypothesis, this happens precisely
    if all $d$-arrows of all $C(x,y)$ --- %
    which are precisely all the $(d+1)$-arrows of $C$ --- %
    are invertible.
    \qedhere
  \end{itemize}
\end{proof}

\begin{notation}
  By abuse of notation we denote by \(I^n\) any composite of the form
  \begin{equation}
    I^n\colon \Cat_{d}\xrightarrow{I^{d+1}} \cdots \xrightarrow{I^n} \Cat_n
  \end{equation}
  (with \(d\leq n\)).
  Similarly, we write \(L_n\) or \(R_n\)
  for any composite of the form
  \begin{equation}
    L_n\colon \Cat_{d}\xrightarrow{L_{d-1}} \cdots \xrightarrow{L_n} \Cat_n
    \quad\text{or}\quad
    R_n\colon \Cat_{d}\xrightarrow{R_{d-1}} \cdots \xrightarrow{R_n} \Cat_n,
  \end{equation}
  respectively (with \(d\geq n\)).
  In other words, \(R_n, L_n, I^n\) always take values in $n$-categories
  irrespective of their source.
  If we ever want to keep track of the source of the functor,
  we write \(R_n^d, L_n^d, I^n_d\) instead\footnote{
    The convention is such that the bigger number is always on top.
  }.

  We call \(R_n\) and \(L_n\) the \emph{$n$-core} and \emph{$n$-localization}
  functor, respectively.
  Moreover, we view the categories
  \begin{equation}
    \An=\Cat_0\subset \Cat_1\subset\cdots \subset \Cat_{d-1}\subset \Cat_{d}\subset\cdots
  \end{equation}
  as contained in each other via the fully faithful functors \(I^d\),
  which we drop from the notation when convenient.
\end{notation}

\begin{remark}
  \label{rem:identify-R_0}
  The right adjoint to the canonical inclusion
  $I^d\colon \An\hookrightarrow \Cat_d$
  can be described in three different ways,
  which therefore are all identified
  (compatibly with the respective units).
  \begin{itemize}
  \item
    $\Cat_d(*,-)$,
    because $I^d$ is the unique colimit preserving functor that maps $*\mapsto *$,
    hence $I^d\simeq -\otimes *$
    (using the unique $\An$-tensoring on the presentable category $\Cat_d$).
  \item
    $R_0^d$, by the construction of \Cref{thm:L-I-R}.
  \item
    The projection map
    \((-)^\simeq\colon \Cat_d\subset \flCat{\Cat_{d-1}}\to \An\)
    (where on the right we have the defining cartesian fibration
    \((A,D)\mapsto A\)).
    This adjunction is exhibited pointwise for each $(C^\simeq,C)\in \Cat_d$
    by the unit map
    $C^\simeq \simeq (C^\simeq,C^\inv(-,-)^\simeq)\to (C^\simeq,C)$,
    where
    \begin{equation}
      C^\inv(-,-)^\simeq\colon O(C^\simeq)\to \An\subseteq \Cat_{d-1}
    \end{equation}
    is the subalgebra of $C\colon O(C^\simeq)\to \Cat_{d-1}$
    defined for each color $(x,y)\in C^\simeq\times C^\simeq$
    by the subanima
    \begin{equation}
      C^\inv(x,y)^\simeq \subseteq C(x,y)^\simeq\subseteq C(x,y)
    \end{equation}
    of invertible $1$-arrows.
  \end{itemize}
\end{remark}

From the construction of \Cref{thm:L-I-R}
we can easily deduce the following recursive description
of the $n$-core and $n$-localization functor.

\begin{lemma}[Objects and arrows of the $n$-core]
  \label{lem:arrows-of-n-core}
  Let \(0\leq n\leq d\) and \(C\) a \(d\)-category.
  \begin{enumerate}
  \item
    The counit \(R_nC\to C\)
    induces an equivalence \((R_nC)^\simeq\xrightarrow{\simeq} C^\simeq\) on objects.
  \item
    For each $k\geq 0 $ and each parallel pair \((x,y)\) of $(k-1)$-arrows in \(R_nC\),
    it induces equivalences 
    \begin{equation}
      (R_nC)\arlev{k}(x,y)
      \xrightarrow{\sim} 
      \begin{cases}
        R_{n-k}(C\arlev{k}(x,y)),\quad \text{for } k\leq n
        \\
        C^\inv\arlev{k}(x,y)^\simeq, \quad\text{for } k > n
      \end{cases}
    \end{equation}
    of $\pred[k]{n}$-categories.
  \end{enumerate}
\end{lemma}

\begin{remark}
  The lemma can be summarized as saying that
  \(R_nC\) is the sub-$n$-category of \(C\)
  with the same $k$-arrows for $k\leq n$
  and only the invertible $k$-arrows for $k>n$.
  
  The identifications of \Cref{lem:arrows-of-n-core}
  are well-formed, because inductively we can
  interchangeably treat \((x,y)\)
  either as parallel $(k-1)$-arrows in $R_nC$
  or as parallel $(k-1)$-arrows in $C$ (which are invertible for $k-1>n$).
\end{remark}

\begin{proof}
  \begin{enumerate}
  \item
    By decomposing the counit as $R_nC\to R_{n+1}C\to \dots \to C$,
    it suffices to show the statement for $d=n+1$.
    For $n=0$ and $d=1$, we have $R_0C=C^\simeq$ by definition.
    For $d\geq 2$,
    If we view $C=(C^\simeq,C)$ as a $C^\simeq$-flagged $\Cat_{d-1}$-category,
    then by construction, the counit $R_nC\to C$ is just the map
    \begin{equation}
      (R_{n-1})_*(C^\simeq,C)\coloneqq (C^\simeq,R_{n-1}\circ C)\to (C^\simeq,C)
    \end{equation}
    induced by the counit transformation
    $I^{d-1}R_{n-1}\to \id\colon \Cat_{d-1}\to \Cat_{d-1}$;
    here $R_{n-1}\circ C$ denotes the composite of the operad map
    (a.k.a.\ $O(C^\simeq)$-algebra)
    $C\colon O(C^\simeq)\to (\Cat_{d-1},\times)$
    with the product-preserving functor $R_{n-1}$.
    From this description, it is manifest that it is the identity on objects.
  \item
    The case $n=d$ is vacuous, as is the case $k=0$.

    The case $0=n<d$ and $k=1$ is
    \begin{equation}
      C^\simeq(x,y)\simeq (R_0C)(x,y)
      \xrightarrow{\sim}
      C^\inv(x,y)^\simeq,
    \end{equation}
    (using the identification of \Cref{rem:identify-R_0})
    which is precisely the univalence condition that distinguishes
    $d$-categories from flagged $\Cat_{d-1}$-categories.

    In the case $1\leq n$ and $k=1$,
    the counit is the map $R_nC\simeq (R_{n-1})_*(C)\to C$
    which for every pair of objects $(x,y)$ of $C$
    induces the equivalence
    \begin{equation}
      (R_nC)(x,y)\xrightarrow{\sim}R_{n-1}(C(x,y)).
    \end{equation}

    We prove the remaining cases by double induction on $d$ and $k$.
    Let $k\geq 2$ and denote by $(s,t)$ the $(k-2)$-dimensional source
    and target of $(x,y)$.
    Then by inductive hypothesis applied to $(s,t)$
    (instead of $(x,y)$, by induction on $k$), we have the equivalence
    \begin{equation}
      (R_nC)(s,t)\xrightarrow{\sim}
      \begin{cases}
        R_{n-k+1}(C(s,t)),
        \quad \text{for } k-1\leq n
        \\
        C^\inv(s,t)^\simeq
        \quad \text{for } k-1> n
      \end{cases}
    \end{equation}
    We pass to hom-$\pred[k]{n}$-categories and treat the two cases separately:
    In the first case ($k\leq n+1$) we have
    \begin{equation}
      (R_nC)(s,t)(x,y)
      \xrightarrow{\sim}
      R_{n-k+1}(C(s,t))(x,y)
      \xrightarrow{\sim}
      \begin{cases}
        R_{n-k}(C(s,t)(x,y))
        \quad\text{for } k\leq n
        \\
        C(s,t)^\inv(x,y)^\simeq
        \quad\text{for } k=n+1
      \end{cases}
    \end{equation}
    where the second equivalence is the inductive hypothesis
    (for $C(s,t)$ instead of $C$, by induction on $d$).
    In the second case ($k>n+1$), $x,y\colon s\to y$ are then invertible
    and we have
    \begin{equation}
      (R_nC)(s,t)(x,y)\xrightarrow{\sim}C^\inv(s,t)^\simeq(x,y)
      \xrightarrow{\sim}
      C(s,t)^\simeq(x,y)
      \xrightarrow{\sim}
      C(s,t)^\inv(x,y)
    \end{equation}
    justified as follows:
    the middle equivalence follows from the fact that
    $C^\inv(s,t)^\simeq\hookrightarrow C(s,t)^\simeq$
    is an embedding of anima,
    the rightmost equivalence is univalence for the $\pred[(k-1)]{n}$-category $C(s,t)$.
    \qedhere
  \end{enumerate}
\end{proof}

\begin{lemma}[Objects and arrows of the $n$-localization]
  \label{lem:arrows-of-n-localization}
  Let \(0\leq n\leq d\) and \(C\) a \(d\)-category.
  \begin{enumerate}
  \item
    The unit \([-]\colon C\to L_n C\) is surjective on objects.
  \item
    For each \(n \geq k\geq 0\) and each parallel pair
    \((x,y)\) of \((k-1)\)-arrows in \(C\), it induces an equivalence
    \begin{equation}
      L_{n-k}(C\arlev{k}(x,y))\xrightarrow{\sim} (L_n C)\arlev{k}([x],[y])
    \end{equation}
    of \((n-k)\)-categories.
  \end{enumerate}
\end{lemma}

\begin{proof}
  Note that the case $n=d$ is vacuous because in this case $L^d_n\simeq \id_{\Cat_d}$.
  So let us assume $n<d$.

  \begin{enumerate}
  \item
    Consider the full sub-$n$-category $D\subseteq L_nC$
    spanned by the image of the map $[-]\colon C^\simeq \to (L_nC)^\simeq$.
    Then by universality, the functor $[-]\colon L\to D$
    factors through the unit $[-]\colon C\to L_nC$
    so that we have a commutative diagram
    \begin{equation}
      \begin{tikzcd}
        &C
        \ar[ld,"{[-]}"']
        \ar[rd,"{[-]}"]
        \ar[d]
        \\
        L_nC\ar[r]&D\ar[r,hookrightarrow]&L_nC,
      \end{tikzcd}
    \end{equation}
    where the lower composite is (uniquely identified with) the identity
    $L_nC\to L_n C$ (again by universality).
    It follows that $D=L_nC$, which is what we needed to show.
  \item
    The case $d\leq 1$ is vacuous.
    We prove the case $d\geq 2$ by induction.
    We recall that $L_nC$ comes equipped with
    the defining complete equivalence
    \begin{equation}
      \label{eq:completion-L-star}
      (L_{n-1})_*C\to  L_nC
    \end{equation}
    (of flagged $\Cat_{d-1}$-categories)
    from the construction of \Cref{thm:L-I-R}
    (which is an instance of \Cref{thm:complete-localization}).
    For any parallel pair $(x,y)$ of objects of $C$
    (or, equivalently, of $(L_{n-1})_*C$)
    the complete equivalence \eqref{eq:completion-L-star}
    induces an equivalence
    \begin{equation}
      ((L_{n-1})_*C)(x,y) \coloneqq L_{n-1}(C(x,y))\xrightarrow{\sim} (L_nC)([x],[y])
    \end{equation}
    of $(n-1)$-categories which yields the case $k=1$.
    The case $k=0$ is vacuous.
    Finally, we prove the case $k\geq 2$ by induction.
    Denote by $(s,t)$ the $(k-2)$-dimensional source and target of $(x,y)$;
    then the desired equivalence is the composite
    \begin{equation}
      L_{n-k}(C(s,t)(x,y))
      \xrightarrow{\sim}
      (L_{n-k+1}(C(s,t)))([x],[y])
      \xrightarrow{\sim}
      (L_nC)([s],[t])([x],[y]),
    \end{equation}
    where the first equivalence is the inductive hypothesis applied to the
    $(d-k+1)$-category $C(s,t)$
    (instead of $C$, by induction on $d$)
    and the second equivalence 
    is the inductive hypothesis applied to $(s,t)$
    (instead of $(x,y)$, by induction on $k$)
    and then passing to homs.
    \qedhere
  \end{enumerate}
\end{proof}

\begin{lemma}[Objects and arrows of sequential (co)limits]
  \label{lem:sequential-homs}
  Fix $d\in \naturals$.
  \begin{enumerate}
  \item
    Consider a sequential limit
    $C\coloneqq \lim_{k:\omega^\op}C_k$
    of $d$-categories.
    We have $C^\simeq\xrightarrow{\sim} \lim_k (C_k)^\simeq$,
    and, for each $x,y\in C$, 
    the induced equivalence
    \begin{equation}
      C(x,y)\xrightarrow{\simeq}\lim_{k:\omega^\op}C_k(\pr_kx,\pr_ky)
    \end{equation}
    of $\pred{d}$-categories.
  \item
    Consider a sequential limit
    $C\coloneqq\colim_{k:\omega}C_k$
    of $d$-categories.
    We have $\colim_k (C_k)^\simeq\xrightarrow{\sim} C^\simeq$,
    and, for each $x,y\in C_0$, the induced equivalence
    \begin{equation}
      \colim_{k:\omega} C_k(x_k,y_k)\xrightarrow{\simeq}C(x_\infty,y_\infty)
    \end{equation}
    of $\pred{d}$-categories,
    where we denote by $x_k$ and $x_\infty$ the image of $x\in C_0$
    in $C_k$ and $C$, respectively.
  \end{enumerate}
\end{lemma}

\begin{proof}
  Since $\Cat_{d-1}$ is an absolute distributor
  (see \cite[Definition~7.19]{Haugseng-rectification}),
  we have the full embedding
  \begin{equation}
    \Cat_d\coloneqq \Cat_{\Cat_{d-1}}\subseteq\Fun(\Delta^\op, \Cat_{d-1})
  \end{equation}
  of \Cref{thm:V-cat-simplicial}.
  All three conditions describing the essential image
  are preserved by limits and filtered colimits
  (the first two are defined in terms of finite limits,
  and $\An\hookrightarrow \Cat_{d-1}$ has both adjoints).
  Thus we may compute limits and filtered colimits of $d$-categories
  in this ambient category $\Fun(\Delta^\op,\Cat_{d-1})$.
  In this ambient, the core functor is just $C\mapsto C(\Delta^0)$
  and the hom-$\pred{d}$-categories of $C$ are obtained as fibers of the map
  $C(\Delta^1)\to C(\boundary\Delta^1)\simeq C(\Delta^0)\times C(\Delta^0)$.
  \Cref{lem:sequential-homs} thus follows from the fact
  that both of these operations (evaluation and fibers)
  commute with limits and filtered colimits.
\end{proof}

\section{Surjectivity}
\label{sec:surjectivity}

The key notion that we are going to use to study $d$-categories
in the limit \(d\to \infty\) is that of $n$-surjective maps for various
$n\in \naturals$ (which is a number that in general is unrelated to \(d\)).

\begin{remark}
  The notion of $n$-surjectivity has also been studied extensively
  in \cite[Section~5.3]{LMRSW} and in \cite{Loubaton-effectivity}.
  Note that while Loubaton uses the same numbering convention as us,
  the one in \cite{LMRSW} differs by $1$:
  they call $(n-1)$-surjective what we would call $n$-surjective.

  Many of the results in this section can also be found in or deduced from those works;
  but since most proofs are rather elementary
  we prefer to record them here for completeness.
\end{remark}

\begin{definition}
  \label{def:surj}
  Let \(n,d\in \naturals\).
  A map \(F\colon C\to D\) of $d$-categories is called $n$-surjective if
  \begin{itemize}
  \item
    \(F\) is surjective on objects, and
  \item
    if $n\geq 1$, for each $x,y \in C$,
    the induced map \(C(x,y)\to D(Fx,Fy)\) is an $(n-1)$-surjective map of
    $\pred{d}$-categories.
  \end{itemize}
\end{definition}

\begin{remark}
  \label{rem:surjectivity-explicit}
  Unraveling the recursion, we see that \(F\colon C\to D\)
  is $n$-surjective if and only if,
  for each $0\leq k\leq n$
  and each parallel pair $(x,y)$ of $(k-1)$-arrows,
  the induced map
  \(C\arlev{k}(x,y)^\simeq\to D\arlev{k}(x,y)^\simeq\)
  of animae is surjective.
  Recall that, for $k=0$, this is just the map
  \(C^\simeq\to D^\simeq\) on objects by convention
  (a parallel pair of $(-1)$-arrows is no data).

  Explicitly, \(F\colon C\to D\) is $n$-surjective if and only if
  for every object $z\in D$ there merely exists \(x\in C\) with
  \(Fx\simeq z\),
  and for every $k$-arrow (\(1\leq k\leq n\))
  of the form \(f\colon Fx\to Fy\) of $D$
  there merely exists a $k$-arrow \(f'\colon x\to y\) in \(C\)
  with \(Ff'\simeq f \colon x\to y\).

  In particular, $n$-surjectivity gets progressively stronger as $n$ increases.
\end{remark}

\begin{lemma}
  \label{lem:n-surj-anima}
  A map \(F\colon A\to B\) of animae is \(n\)-surjective
  if and only if for every \(0\leq k\leq n\),
  and every lifting problem
  \begin{equation}
    \label{eq:sphere-lifting-for-k-surjectivity}
    \begin{tikzcd}
      {\sphere{k-1}}
      \ar[r,"u"]
      \ar[d]
      &
      {A}
      \ar[d,"F"]
      \\
      {*}
      \ar[ur,dashed,"f'"]
      \ar[r,"f"]
      &
      {B}
    \end{tikzcd}
  \end{equation}
  there merely exists a dashed solution as indicated.
\end{lemma}

\begin{proof}
  Let \(F\colon A\to B\) be a map of animae and \(k\in\naturals\).
  We consider the standard cell structure on the $(k-1)$-sphere \(\sphere{k-1}\)
  with two cells in each dimension \(0,\dots, k-1\);
  its extension  \(\sphere{k-1}\to \disk{k}\simeq *\)
  to a disk is obtained by attaching a single \(k\)-cell.
  Using these cell structures one sees inductively that a map \(u\colon \sphere{k-1}\to A\)
  amounts to a parallel pair
  \((x,y)\) of (necessarily invertible) $(k-1)$-arrows of \(A\).
  Moreover, a solid square \eqref{eq:sphere-lifting-for-k-surjectivity}
  then amounts to an (invertible) $k$-arrow \(f\colon Fx\xrightarrow{\sim} Fy\)
  in \(B\),
  and a dashed lift as indicated amounts to an (invertible)
  $k$-arrow \(f'\colon x\xrightarrow{\sim} y\)
  with \(Ff'\simeq f\colon Fx\to Fy\).

  Thus we see that \(F\colon A\to B\) is $n$-surjective if and only if
  for each \(0\leq k\leq n\) and each solid square as in
  \eqref{eq:sphere-lifting-for-k-surjectivity},
  there merely exists a dashed lift as indicated.
\end{proof}

\begin{remark}
  \label{rem:surjective-anima}
  \Cref{lem:n-surj-anima} says that a map $F$ of animae is $n$-surjective
  if and only if
  it is $n$-connected\footnote{
    Nowadays usually called ``$(n-1)$-connected'' in the context of topos theory%
  } in the classical sense of algebraic topology.
  Equivalently, the induced map
  \(\pi_k(F)\colon \pi_k(A)\to\pi_k(B)\) on homotopy groups is
  \begin{itemize}
  \item
    a bijection for \(0\leq k<n\)
  \item
    a surjection for \(k=n\)
  \end{itemize}
  (for all choices of base point);
  see any book on algebraic topology such as \cite[Section~6.7]{tomDieck}.
\end{remark}

\begin{corollary}
  \label{cor:finit-infty-surj-equiv}
  A map \(F\colon A\to B\) of animae is an equivalence
  if and only if it is \(n\)-surjective for all \(n\in\naturals\).
\end{corollary}

Bootstrapping from the case of animae, it is easy to deduce the same
for \(d\)-categories:

\begin{lemma}
  \label{lem:d-cat-equiv-n-surj}
  Let \(d\in \naturals\).
  A map \(F\colon C\to D\) of \(d\)-categories
  is an equivalence if and only if \(F\) is $n$-surjective for all \(n\in\naturals\).
\end{lemma}

\begin{proof}
  ``Only if'' is immediate from the fact that equivalences
  are surjective on objects and induce equivalences on hom-$\pred{d}$-categories.
  So we only prove ``if''.
  
  We prove the claim by induction on \(d\in\naturals\).
  For \(d=0\) we are just talking about animae;
  hence this case is covered by \Cref{cor:finit-infty-surj-equiv}.
  Now assume that \(d\geq 1\) and that \(F\) is $n$-surjective for all $n\in \naturals$.
  Then for each pair of objects \(x,y\in C\)
  the induced functor \(F\colon C(x,y)\to D(Fx,Fy)\)
  of \((d-1)\)-categories is (by definition) again $n$-surjective for all $n$,
  hence an equivalence by induction.
  Since \(F\) induces equivalences on all hom-$(d-1)$-categories
  it is fully faithful (by definition).
  Since \(F\) is by assumption surjective on objects ($0$-surjective),
  it is thus a complete equivalence of complete flagged \(\Cat_{d-1}\)-categories,
  hence an equivalence of $d$-categories (see \Cref{rem:isos-in-cat-V}).
\end{proof}

\begin{lemma}
  \label{lem:d+1-surj-conservative}
  Every $(d+1)$-surjective functor of $F\colon C\to D$ of $d$-categories
  is conservative,
  i.e., only sends invertible arrows to invertible arrows (in all dimensions).
\end{lemma}

\begin{proof}
  Let $f'\colon x\to y$ be a $k$-arrow such that $f=F(f')$ is an isomorphism.
  If $k>d$ then $f'$ is trivially an isomorphism.
  We prove the case $k\leq d$ by downward induction on $k$.
  Take a left and a right inverse $g_-,g_+\colon Fy\to Fx$
  with isomorphisms $h_+\colon \id_{Fx}\xrightarrow{\sim} g_+f$
  and $h_-\colon fg_-\xrightarrow{\sim} \id_{Fy}$.
  Since $F$ is $(d+1)$-surjective,
  there merely exist lifts $g'_-,g'_+\colon x\to y$ and
  $h'_-\colon f'g'_-\to \id_y$ and $h'_+\colon \id_x\to g'_+f'$ to $C$
  (under the implicit identifications $F(f'g'_-)\simeq fg_-$, $F(\id_y)\simeq \id_{Fy}$, etc.),
  where $h'_-$ and $h'_+$ are isomorphisms by induction
  (since they are sent to isomorphisms $h_-$ and $h_+$, respectively).
  It follows that $f'$ is an isomorphism, as claimed.
\end{proof}

\begin{proposition}[Surjectivity and cores]
  \label{prop:surjective-core}
  Let \(n\leq d\) and let \(F\colon C\to D\) be a functor of \(d\)-categories.
  \begin{enumerate}
  \item
    The counit map \(R_nC\to C\) is \(n\)-surjective.
  \item
    For all \(k\leq n\):
    \begin{equation}
      F\colon C\to D
      \text{ is $k$-surjective }
      \iff
      R_nF\colon R_nC\to R_nD
      \text{ is $k$-surjective}
    \end{equation}
  \item
    \label{lem:surj-cores-k>d}
    For all \(k > d\):
    \begin{equation}
      F\colon C\to D
      \text{ is $k$-surjective }
      \implies
      R_nF\colon R_nC\to R_nD
      \text{ is $k$-surjective}.
    \end{equation}
  \end{enumerate}
\end{proposition}
\begin{proof}
  The first two statements follow from the fact that $k$-surjectivity
  of a functor $F$
  is defined just in terms of arrows of dimensions at most $k$
  (and identifications between them),
  so only depends on $R_kF$.

  For the last statement let $F\colon C\to D$ be $k$-surjective
  and $l\leq k$.
  Then for every parallel pair $(x,y)$ of $(l-1)$ arrows in $R_nC$
  and every $l$-arrow $f\colon Fx\to Fy$ in $R_nD$,
  there merely exists a lift 
  $f'\colon x\to y$ in $C$ with $Ff'\simeq f$
  (as $l$-arrows $Fx\to Fy$ of $C$).
  It suffices to show that $f'$ again lies in $R_nC$
  (because then the identification $Ff'\simeq f$ also lies in $R_nC$);
  in other words (by \Cref{lem:arrows-of-n-core}),
  we have to show that if $l>n$ then $f'$ is invertible
  (for $l\leq n$ there is nothing more to show because then $f'$
  lies in $R_nC$ regardless of invertibility).
  This follows from the fact that $F$ is conservative
  by \Cref{lem:d+1-surj-conservative}:
  if $l>n$, then $f\colon Fx\to Fy$ is invertible
  (because it is an $l$-arrow of $R_nD$),
  hence so is $f'\colon x\to y$ that maps to it.
\end{proof}

\begin{remark}
  Note that the implication
  \begin{equation}
    F\colon C\to D
    \text{ is $k$-surjective }
    \implies
    R_nF\colon R_nC\to R_nD
    \text{ is $k$-surjective}.
  \end{equation}
  holds if either \(k\leq n\) or \(d<k\),
  but typically fails in the range \(n<k\leq d\).

  For example, for \(n=0, k=d=1\), one may consider the terminal map
  \begin{equation}
    t\colon W\coloneqq \{u\colon 0\rightleftarrows 1 : v\}\to *
  \end{equation}
  from the walking anti-parallel pair
  (freely generated by two $1$-arrows \(u\colon 0\to 1\) and \(v\colon 1\to 0\)).
  It is clearly surjective on objects
  and it is easily seen to be surjective on $1$-arrows;
  indeed, \(u,v,\id_0,\id_1\) are lifts of the (unique) arrows
  \begin{equation}
    t(0)\to t(1),\quad t(1)\to t(0),\quad t(0)\to t(0),\quad t(1)\to t(1)
  \end{equation}
  in \(*\), respectively.
  Since the only invertible arrows in \(W\)
  are the identities,
  we see that \(R_0t \colon R_0 W \to R_0(*) \simeq *\)
  is just the terminal map
  \begin{equation}
    t\colon \{0,1\} \to *,
  \end{equation}
  which is not $1$-surjective
  because there does not exist any lift \(0\to 1\) of the arrow
  \(\id_*\colon t(0)\xrightarrow{\sim} t(1)\).

  This example illustrates the general problem:
  while $k$-surjectivity guarantees that any \(k\)-arrow
  $f\colon Fx\to Fy$ of $D$
  can be lifted to a $k$-arrow \(f'\colon x\to y\) of $C$,
  there is in general no guarantee that for invertible $f$
  (so that $f$ is even a $k$-arrow of $R_{k-1}D$)
  the lift \(f'\) can be chosen to be invertible as well
  (so that $f$ is even a $k$-arrow of $R_{k-1}C$).
\end{remark}

\begin{remark}
  The converse implication 
  \begin{equation}
    F\colon C\to D
    \text{ is $k$-surjective }
    \impliedby
    R_nF\colon R_nC\to R_nD
    \text{ is $k$-surjective}
  \end{equation}
  typically fails when \(n<k\).
  Indeed, knowing that for \(n<l\leq k\)
  there exist (invertible) lifts for
  \emph{invertible} $l$-arrows
  \(Fx\xrightarrow{\sim} Fy\)
  (which are the $l$-arrows of \(R_n D\))
  does not let us conclude anything about lifts of non-invertible $l$-arrows.

  As a trivial example (with $d=1,n=0,k=1$),
  consider the inclusion $\{0,1\}\hookrightarrow \{0\to 1\}$
  which is not $1$-surjective but even becomes an iso after applying $R_0$.
\end{remark}

\begin{lemma}
  \label{lem:surjectivity-lifts-parallel-pairs}
  Let $n \in \naturals$ and let \(F\colon C\to D\) be an $n$-surjective functor of \(d\)-categories.
  Then for each $k\leq n$ the two induced maps
  \begin{equation}
    F\colon \Mor_kC\to \Mor_k D
    \quad\text{and}\quad
    F\colon \Par_kC\to \Par_k D
  \end{equation}
  of animae are surjective.
\end{lemma}
\begin{proof}
  Without loss of generality we may assume $k=n$
  (because if $F$ is $n$-surjective, then also $k$-surjective;
  see \Cref{rem:surjectivity-explicit}).
  We just prove the second statement, the first one is analogous.
  We have to show that for every parallel pair \((x,y)\) of $n$-arrows in \(D\),
  there merely exists a parallel pair \((x',y')\) of $n$-arrows in \(C\)
  with \(F(x',y')\simeq (x,y)\).
  If \(n=0\), this is immediate by applying the definition of $0$-surjectivity twice.

  For \(n\geq 1\),
  we can view \((x,y)\) as a parallel pair of \((n-1)\)-arrows in \(D(s_0,t_0)\),
  where $s_0,t_0\in D$ are the $0$-dimensional source and target of $(x,y)$, respectively.
  By the $0$-surjectivity of \(F\), there merely exist two objects
  \(s'_0,t'_0\in C\) with \(Fs'_0\simeq s_0\) and \(F t'_0\simeq t_0\).
  Since the induced functor \(F\colon C(s'_0,t'_0)\to D(s_0,t_0)\)
  is \((n-1)\)-surjective,
  there exists by induction a parallel pair \((x',y')\)
  of \((n-1)\)-arrows in \(C(s'_0,t'_0)\)
  with \(F(x',y')\simeq (x,y)\).
  Viewing $(x',y')$ as a parallel pair of $n$-arrows in $C$ yields the claim.
\end{proof}

The interaction of surjectivity with localization is more straightforward
than the one with cores but the proof requires more sophisticated techniques.
We give a relatively elementary proof in \Cref{subsec:surj-loca} below.
It can also be deduced from the existence of the
unique factorization system
($n$-surjective, $(n+1)$-fully faithful)
established in \cite{Loubaton-effectivity} or \cite{LMRSW};
see also \Cref{thm:factorization-surj-ff}.

\begin{proposition}[Surjectivity and localization]
  \label{prop:surjective-localization}
  Let \(n\leq d\) and
  let \(F\colon C\to D\) be a functor of $d$-categories.
  \begin{enumerate}
  \item
    The unit map \(C\to L_n C\) is \(n\)-surjective.
  \item
    For all $k\in \naturals$:
    \begin{equation}
      F\colon C\to D \text{ is $k$-surjective}
      \implies
      L_nF\colon L_nC\to L_nD \text{ is $k$-surjective}.
    \end{equation}
  \end{enumerate}
\end{proposition}

\begin{remark}
  Note that, unlike the case of cores,
  the converse implication
  can fail even for \(k\leq n\).
  For example, for \(k=n=0, d=1\)
  one may consider the inclusion
  \begin{equation}
    \{0\}\to [1]\coloneqq\{0\to 1\}
  \end{equation}
  of one end-point into the walking arrow.
  It is clearly not $0$-surjective,
  but even becomes an equivalence
  \begin{equation}
    \{0\}\simeq L_0\{0\}\to L_0[1]\simeq *
  \end{equation}
  upon passing to \(0\)-localization.
\end{remark}

It is clear by definition that surjective maps are closed under composition.
Additionally, they satisfy certain cancellation properties
which are not quite as straightforward as one might expect:

\begin{proposition}[Cancellation properties of $n$-surjective maps]
  \label{prop:cancellation-surj}
  Let \(d,n\in \naturals\).
  Consider two composable functors
  \(B\xrightarrow{F}C\xrightarrow{G} D\) 
  of \(d\)-categories
  and assume that \(GF\) is \(n\)-surjective.
  \begin{enumerate}
  \item
    (easy cancellation)
    If \(F\) is $(n-1)$-surjective, then \(G\) is \(n\)-surjective.
  \item
    (hard cancellation)
    If \(G\) is $(n+1)$-surjective, then \(L_nF\)  is $n$-surjective.
  \end{enumerate}
\end{proposition}

Before we start with the proof, let us point out that we will
use \Cref{prop:surjective-localization}
for the hard cancellation,
but \emph{not} for the easy cancellation.
Thus we will be justified in using the easy cancellation
in the proof of 
\Cref{prop:surjective-localization}
below (but of course not the hard cancellation).

\begin{proof}
  \begin{enumerate}
  \item 
    Let \(k\leq n\), and let \((x,y)\)
    be a parallel pair of \((k-1)\)-arrows in \(C\).
    Then,  by the $(n-1)$-surjectivity of \(F\colon B\to C\),
    \Cref{lem:surjectivity-lifts-parallel-pairs}
    guarantees the existence of a parallel pair \((x',y')\)
    of \((k-1)\)-arrows in \(B\) with \(F(x',y')\simeq (x,y)\).

    Now let \(f\colon Gx\to Gy\) be a $k$-arrow,
    which we can view as a $k$-arrow \(f\colon GF x'\to GF x'\).
    By the \(n\)-surjectivity of \(GF\) there merely exists
    a $k$-arrow \(f'\colon x'\to y'\) with \(GFf'\simeq f\);
    hence there exists a desired lift \(Ff'\colon x\to y\)
    of \(f\colon Gx\to Gy\).
  \item
    By \Cref{prop:surjective-localization},
    the maps \(L_n(GF)\) and \(L_n(G)\) are still $n$-surjective and \((n+1)\)-surjective,
    respectively. 
    By replacing \(F\) and \(G\) (hence $GF$) by their $n$-localizations,
    we may thus restrict to the case \(d=n\) where we have to prove that \(F=L_nF\)
    is $n$-surjective.

    To this end, let \(f\colon Fx\to Fy\) be a $k$-arrow \((k\leq n)\) in $C$,
    where \((x,y)\) is a parallel pair of \((k-1)\)-arrows in \(B\).
    By the $n$-surjectivity of \(GF\) there exists a $k$-arrow
    \(f'\colon x\to y\) in $B$ with \(GF f'\simeq G f\) in $D$.

    Since \(G\) is $(n+1)$-surjective and \(n+1>n=d\), it follows from 
    \Cref{prop:surjective-core}
    that \(R_{k}G\) is still $(n+1)$-surjective.
    In particular, there exists a lift of the $(k+1)$-isomorphism
    \(GF f'\xrightarrow{\sim} Gf\)
    to a $(k+1)$-isomorphism \(Ff'\xrightarrow{\sim} f\) in $C$,
    which means that there exists the desired \(f'\colon x\to y\) in $B$
    with \(Ff'\simeq f\) in $C$.
    \qedhere
  \end{enumerate}
\end{proof}

\begin{remark}
  In general, the hard cancellation really only yields that \(L_n F\)
  is $n$-surjective, and not $F$ itself.
  For example, for \(n=0,d=1\), one may consider the composite
  \begin{equation}
    \{0\}\xhookrightarrow{F} W\coloneqq \{0\rightleftarrows 1\} \xrightarrow{G} *.
  \end{equation}
  The terminal map \(G\) from the walking anti-parallel pair is $1$-surjective,
  and the composite \(GF\) is even an equivalence.
  But even though the induced map
  \begin{equation}
    L_0F\colon \{0\}\to L_0(W)\simeq S^1
  \end{equation}
  of animae is $0$-surjective,
  the inclusion \(F\colon \{0\}\hookrightarrow W\) itself is clearly not.
\end{remark}

\subsection{Surjectivity and localization}
\label{subsec:surj-loca}

Before we can prove 
\Cref{prop:surjective-localization}
we need a few preliminary lemmas about categories and simplicial animae.
As is customary, we denote by
$\boundary\Delta^n\subset \Delta^n\colon \Delta^\op \to \Set\subset \An$
the simplicial sets
describing the $n$-simplex and its boundary.

\begin{lemma}[\cite{ERW}, Lemma~2.4] 
  \label{lem:connected-realization}
  Let $k\in \naturals$ and let $f\colon X\to Y\colon \Delta^\op\to\An$
  be a map of simplicial animae.
  Assume that for each $i\leq k$
  the map $f(\Delta^i)\colon X(\Delta^i)\to Y(\Delta^i)$
  is $(k-i)$-surjective.
  Then the induced map $|X|\to |Y|$ of animae is $k$-surjective.
\end{lemma}

\begin{lemma}
  \label{lem:fiber-boundary-simplex-hom}
  Let $C$ be a category viewed as a simplicial anima and $n\in\naturals$.
  Then each fiber of the map $C(\Delta^n)\to C(\boundary\Delta^n)$
  is identified with $C(x,y)$
  for some parallel pair $(x,y)$ of $(n-1)$-arrows in $C$.
\end{lemma}

\begin{proof}[Construction]
  We construct the desired identification by recursion on $n$.
  For $n=0,1$ this is just saying that $C(\Delta^0)\simeq C^\simeq$
  are the objects of $C$ and the fibers of
  $C(\Delta^1)\to C(\boundary\Delta^1)\simeq C^\simeq \times C^\simeq$
  are its hom-animae.
  In the recursion step $n\geq 2$, select any $0<j<n$
  (say $j\coloneqq 1$ for concreteness)
  and consider the inner horn $\Lambda_j^n\subset\boundary\Delta^n\subset \Delta^n$
  as a simplicial set.
  For each $\sigma \in C(\boundary\Delta^n)$ denote by $\tau$
  its restriction in $C(\Lambda_j^n)$ and consider the following diagram
  obtained by forming pullbacks:
  \begin{equation}
    \label{eq:extract-top-simplex-as-arrow}
    \begin{tikzcd}
      Q\ar[r]\ar[d]\isCartesian&*\ar[d,"\sigma_1"]\ar[r]\isCartesian&C(\Delta^n)\ar[d]
      \\
      *\ar[rr,bend left=10,"\sigma" near start]\ar[r,"\sigma_0"']&P\ar[d]\ar[r]\isCartesian
      &C(\boundary\Delta^n)\ar[d]
      \isCartesian \ar[r,"d_j"]
      &C(\Delta^{n-1})\ar[d]
      \\
      &*\ar[r,"\tau"]&C(\Lambda_j^n)\ar[r,"d_j"]&C(\boundary\Delta^{n-1})
    \end{tikzcd}
  \end{equation}
  Note that as indicated
  \begin{itemize}
  \item
    the lower right square is a pullback square because
    the simplicial set $\boundary\Delta^n$ arises from $\Lambda^n_j$
    by attaching a single $(n-1)$-simplex (the missing $j$-th face);
  \item
    the top middle anima is trivial, because it is a fiber
    of the composite $C(\Delta^n)\to C(\Lambda^n_j)$ 
    which is an equivalence as a consequence of the Segal condition.
  \end{itemize}
  By recursion, we have an identification of $P$ ---
  which is the fiber of $C(\Delta^{n-1})\to C(\boundary\Delta^{n-1})$
  at $d_j\tau$ ---
  with $C(\tau_0,\tau_1)$, for some parallel pair $(\tau_0,\tau_1)$
  of $(n-2)$-arrows in $C$.
  Under this identification,
  the objects $\sigma_0,\sigma_1\in P$
  correspond to parallel $(n-1)$-arrows $\sigma_0,\sigma_1\in C(\tau_0,\tau_1)$;
  it follows from the upper left pullback that $Q$ ---
  which is the fiber of $C(\Delta^n)\to C(\boundary\Delta^n)$ at $\sigma$ ---
  is identified with $C(\tau_0,\tau_1)(\sigma_0,\sigma_1)$ as desired.
\end{proof}

\begin{lemma}
  \label{lem:surjectivity-simplicial-complex}
  Let $k\geq 2$ and $i\leq k$,
  and let $F\colon C\to D$ be a $k$-surjective functor of categories,
  viewed as a map of simplicial animae.
  Then for any finite simplicial set $K$ of dimension $\leq i$,
  the induced map
  $F(K)\colon C(K)\to D(K)$ of animae is $(k-i)$-surjective.
\end{lemma}

\begin{proof}
  We prove the lemma by induction on $i$ starting with $i=0$.
  If $K\simeq *$ is a point,
  then $F(*)\colon C(*)\to D(*)$ is just 
  the map $R_0 F\colon R_0C \to R_0D$.
  Since $k\geq 2>1$ we may apply
  \Cref{prop:surjective-core}~\ref{lem:surj-cores-k>d}
  with $d=1$ and $n=0$
  to conclude that this map is indeed $k$-surjective.
  The case of more general finite $0$-dimesional complexes
  follows immediately because finite products
  of $k$-surjective maps are $k$-surjective.

  For the induction step, we may assume $i\geq 1$,
  and that $F(K')$ is $(k-i+1)$-surjective,
  where $\iota\colon K'\subseteq K$ is the $(i-1)$-skeleton.
  Then for any $\sigma'\in C(K')$ we get a map of (horizontal) fiber sequences
  \begin{equation}
    \begin{tikzcd}
      P\coloneqq \fib_{\sigma'}(\iota^*)\ar[d]\ar[r] &C(K)\ar[r,"\iota^*"]\ar[d,"F(K)"]&C(K')\ar[d,"F(K')"]
      \\
      Q\coloneqq \fib_{F(\sigma')}(\iota^*)\ar[r]&D(K)\ar[r,"\iota^*"]&D(K')
    \end{tikzcd}
  \end{equation}
  We claim that the left vertical map is $(k-i)$-surjective for every choice of $\sigma'$.
  Granting this, it is a routine computation with the long exact sequence that
  the middle map $F(K)$ is also $(k-i)$-surjective;
  for the convenience of the reader
  we provide a sketch of this computation separately in \Cref{rem:les-computation} below.

  To prove the claim, we analyze the map $P\to Q$:
  Let $S$ be the (finite) set of non-degenerate $i$-simplices of $K$.
  Then we have a pullback square
  \begin{equation}
    \cdsquareNA[pb]
    {C(K)}
    {\prod_{s\in S}C(\Delta^i)}
    {C(K')}
    {\prod_{s\in S}C(\boundary\Delta^i)}
  \end{equation}
  which means that any $P$---%
  which is a fiber of the left vertical map---%
  is identified with some fiber of the right vertical map.
  Thus by \Cref{lem:fiber-boundary-simplex-hom}
  we have an identification $P\simeq \prod_{s\in S}C(x_s,y_s)$,
  where each $(x_s,y_s)$ is some parallel pair of $(i-1)$-arrows of $C$
  (which of course also depends on $\sigma'$).
  Doing the same for $Q$ and observing that the construction of
  \Cref{lem:fiber-boundary-simplex-hom}
  is functorial in $C$
  (because the diagram \eqref{eq:extract-top-simplex-as-arrow} is),
  we see that the mystery map $P\to Q$
  is identified with the map
  \begin{equation}
    F\colon \prod_{s\in S}C(x_s,y_s)\to \prod_{s\in S} D(Fx_s,F y_s);
  \end{equation}
  it is indeed $(k-i)$-surjective because $F$ is $k$-surjective
  and the $x_s,y_s$ are arrows of dimension $(i-1)$.
\end{proof}

\begin{remark}
  \label{rem:les-computation}
  Here is a sketch of the missing computation in the proof of 
  \Cref{lem:surjectivity-simplicial-complex}.

  For every $j\in\naturals$ we denote by $q_j$, $f_j$ and $f'_j$
  the maps induced on homotopy groups $\pi_j$
  by the maps $P\to Q$, $F(K)$ and $F(K')$,
  respectively
  (this of course depends on the choice of base-points).
  Then we consider the following segment of the long exact sequences:
  \begin{equation}
    \begin{tikzcd}
      \pi_{j+1}C(K')\ar[d,"{f'_{j+1}}"]\ar[r]& \pi_jP\ar[d,"q_j"]\ar[r] &\pi_jC(K)\ar[d,"f_j"]\ar[r]&\pi_jC(K')\ar[r]\ar[d,"f'_j"]&\pi_{j-1}P\ar[d,"q_{j-1}"]
      \\
      \pi_{j+1}C(K')\ar[r]&\pi_jQ\ar[r]&\pi_j D(K)\ar[r]&\pi_jD(K')\ar[r]&\pi_{j-1}Q
    \end{tikzcd}
  \end{equation}

  Recall from \Cref{rem:surjective-anima}
  that a map $g$ of animae is $n$-surjective if
  $\pi_n(g)$ is surjective and $\pi_j(g)$ is bijective for all $0\leq j<n$
  (for all choices of base-points).

  We consider cases:
  \begin{itemize}
  \item
    $j=k-i$.
    In this case $f'_j$ and $q_{j-1}$ are isos and $q_j$ is surjective.
    It follows that $f_j$ is surjective.
    ($f'_{j+1}$ is surjective too, but we don't need that.) 
  \item
    $j<k-i$.
    In this case $f'_{j+1}$, $q_j$, $f'_j$ and $q_{j-1}$
    are all isos.
    It follows that $f_j$ is an iso.
  \end{itemize}
  For very low $j$ one has to be a little bit careful (but this is standard):
  \begin{itemize}
  \item
    To prove surjectivity for $j=0$ (where $q_{j-1}$ is not defined)
    just observe that each object of $[d]\in \pi_0D(K)$ lives over \emph{some}
    $[d']\in \pi_0D(K')$ which then lifts to $[c']\in\pi_0C(K')$ by surjectivity of $f'_0$.
    Choosing $\sigma$ to be a representative of this class,
    we see that the surjectivity of $q_j$ implies that $[d]$ is hit by $f_j$.
  \item
    To prove injectivity for $j=0$ and surjectivity for $j=1$
    (both of which involve $q_{0}$ or $f'_0$ which are not group homomorphisms)
    one has to use the group action of $\pi_1$ on the fiber
    and the corresponding enhanced notion of exactness.
  \end{itemize}
  In all cases one has to remember that $P\to Q$ of course depends
  on the choice of base-point $\sigma'$
  and the long exact sequence exists for \emph{every} choice.
  For some arguments one needs to let the base-point vary;
  it is not enough to consider them one at a time.
\end{remark}

\begin{proof}[Proof of \Cref{prop:surjective-localization}]
  By considering one functor at a time in the defining composition
  $L^d_n=L^{n+1}_n\circ \dots \circ L^d_{d-1}$,
  we may assume without loss of generality that $n=d-1$.
 
  By \Cref{lem:arrows-of-n-localization},
  the unit $C\to L_nC$ is $0$-surjective.
  If $n\geq 1$, then it is also $n$-surjective
  because for each $x,y\in C$ the map
  $C(x,y)\to (L_nC)([x],[y])$ is identified with the unit
  $C(x,y)\to L_{n-1}(C(x,y))$, which is $(n-1)$-surjective by induction.
  This shows the first claim.
  
  To prove the second claim we first consider the commutative square
  \begin{equation}
    \cdsquare
    {C}{D}
    {L_nC}{L_nD}
    {F}{}{}{L_nF}
  \end{equation}
  where the two vertical maps are definitely $0$-surjective (even $n$-surjective)
  and $F$ is (at least) $0$-surjective by assumption;
  thus $L_nF$ is also $0$-surjective
  by composition and easy cancellation%
  \footnote{There is no circularity here,
  because in the proof of
  \Cref{prop:cancellation-surj}
  we used
  \Cref{prop:surjective-localization}
  only for the hard cancellation.}.

  This means that we can reduce the second statement by induction on $n$
  (or $d=n+1$)
  to the case $d=1,n=0$,
  by using for each $x,y\in C$ the identification
  \begin{equation}
      {L_{n-1}(F_{x,y})}\simeq
      {(L_nF)_{[x],[y]}}
      \colon
      {(L_nC)([x],[y])}
      \to
      {(L_nD)([Fx],[Fy])}.
  \end{equation}
  of \Cref{lem:arrows-of-n-localization}
  (because in the presence of $0$-surjectivity,
  $k$-surjectivity of $F$ is equivalent to $(k-1)$-surjectivity
  of each $F_{x,y}$, and similarly for $L_nF$).
 
  Finally we prove the case $d=1, n=0$, where we distinguish three cases for $k$:
  \begin{itemize}
  \item
    We already saw the case $k=0$.
  \item
    Let $k=1$ and assume that $F$ is $1$-surjective.
    Let $\overline{x},\overline{y}\in L_0 C$ be objects
    and $f\colon (L_0 F)\overline{x}\xrightarrow{\sim} (L_0F)\overline{y}$
    a (necessarily invertible) arrow in $L_0D$.
    Then there merely exist $x,y\in C$
    representing $\overline{x}$ and $\overline{y}$
    and a zig-zag
    \begin{equation}
      \label{eq:zig-zag-to-lift}
      Fx \leftarrow d_1 \rightarrow \dots  \leftarrow d_l \rightarrow Fy
    \end{equation}
    in $D$ representing the arrow $f$.
    By $1$-surjectivity of $F$
    there merely exist $c_i \in C$ with $Fc_i\simeq d_i$,
    and arrows
    \begin{equation}
      x \leftarrow c_1 \rightarrow \dots \leftarrow c_l\rightarrow y
    \end{equation}
    individually lifting the arrows in the original zig-zag \eqref{eq:zig-zag-to-lift}.
    The resulting zig-zag represents the desired arrow
    $\overline{x}\xrightarrow{\sim} \overline{y}$
    lifting $f$.
  \item
    Finally, we consider the case $k\geq 2$.
    We view $F\colon C\to D$ as a map of simplicial animae.
    Applying \Cref{lem:surjectivity-simplicial-complex}
    to the simplicial sets $\Delta^i$ for $i=0,\dots, k$,
    we see that $F$ satisfies the assumption of \Cref{lem:connected-realization},
    which then yields the desired $k$-surjectivity
    of the map $L_0C\to L_0D$.
    \qedhere
  \end{itemize}
\end{proof}

\section{Biinductive systems}

Going forward, we are only going to use a few select facts
about $d$-categories, cores, localizations and $n$-surjectivity.
To make the argument more transparent,
we now encapsulate these aspects into an abstract setup.

\begin{definition}
  A \emph{biinductive system} on a category \(C\) consists of a sequence
  \begin{equation}
    C_0\subseteq C_1\subseteq C_2 \subseteq
    \cdots
    \subseteq C_d \subseteq
    \cdots
    \subseteq C
  \end{equation}
  of full subcategories satisfying the following conditions:
  \begin{enumerate}[label=(C\arabic*),ref=(C\arabic*)]
  \item
    The sequence \(C_\bullet\) is exhaustive, i.e.,
    \begin{equation}
      \bigcup_{d\in \naturals}C_d \eqqcolon C_{<\infty}=C.
    \end{equation}
    (the union is considered as full subcategories of $C$).
  \item
    Each inclusion \(I_d^{<\infty}\colon C_d\hookrightarrow C\) has adjoints
    \begin{equation}
      L_d\colon C\to C_d
      \quad
      \text{and}
      \quad
      R_d\colon C\to C_d
    \end{equation}
    on the left and on the right, respectively.
  \item
    Each category \(C_d\) has sequential limits and colimits.
  \end{enumerate}
\end{definition}

\begin{construction}
  Let \(C_\bullet\) be a biinductive system.
  We construct the following categories:
  \begin{align}
    C^R&\coloneqq
    \lim\left(
    \cdots
    \xrightarrow{R_d} C_d \to
    \cdots
    \xrightarrow{R_1} C_1
    \xrightarrow{R_0} C_0
    \right)
    \\
    C^L&\coloneqq
    \lim\left(
    \cdots
    \xrightarrow{L_d} C_d \to
    \cdots
    \xrightarrow{L_1} C_1
    \xrightarrow{L_0} C_0
    \right)
  \end{align}
\end{construction}

\begin{lemma}
  \label{lem:C-R-L-emb-Fun}
  We have full embeddings
  \begin{equation}
    C^R\subseteq \Fun(\omega,C)
    \quad\text{and}\quad
    C^L\subseteq \Fun(\omega^\op,C)
  \end{equation}
  identifying \(C^R\) and \(C^L\)
  with the full subcategories of those sequences
  \begin{equation}
    X_\bullet=(X_0\to X_1\to \cdots\to X_d\to \cdots)
    \quad\text{or}\quad
    Y_\bullet=(\cdots \to Y_d \to\cdots \to Y_1\to Y_0)
  \end{equation}
  in \(C\) such that for each \(d\in \naturals\)
  \begin{itemize}
  \item
    the object \(X_d\) or \(Y_d\) lies in \(C_d\subseteq C\), and
  \item
    the structure maps induce isomorphisms
    \begin{equation}
      X_d\xrightarrow{\sim} R_dX_{d+1}
      \quad\text{or}\quad
      L_dY_{d+1}\xrightarrow{\sim} Y_d.
    \end{equation}
  \end{itemize}
\end{lemma}

\begin{proof}[Construction]
  We construct the identification for \(C^R\);
  the one for $C^L$ is analogous.
  
  Let $p\colon P \to\omega$
  be the cocartesian unstraightening of the sequence
  \(C_\bullet\colon \omega \to \Cat\) 
  (whose transition maps are the inclusions \(C_{d}\hookrightarrow C_{d'}\)
  for \(d\leq d'\)).
  The fact that all arrows in the sequence \(C_\bullet\colon \omega\to \Cat\)
  have a right adjoint means precisely that
  \(p\colon P\to \omega\)
  is not just a cocartesian fibration (by construction),
  but also a \emph{cartesian} fibration;
  its cartesian straightening
  is exactly the diagram
  \((C_\bullet,R_\bullet)\colon \omega^\op \to \Cat\).
  By the general formula for limits of categories,
  this means that we can identify
  \begin{equation}
    C^R=\lim_{\omega^\op}(C_\bullet,R_\bullet)
    \simeq \Fun^{\mathrm{cart}}_\omega(\omega,P)
    \subseteq \Fun_\omega(\omega,P)
  \end{equation}
  where the right side
  denotes the category of cartesian sections of \(p\colon P\to \omega\)
  as a full subcategory of all sections.

  The inclusions assemble into a natural transformation
  \(I^{<\infty}_\bullet\colon C_\bullet \to C \colon \omega \to \Cat\)
  to the constant diagram with value \(C\).
  After unstraightening, this yields a functor
  \(\overline{I}\colon P\to \omega\times C\) over \(\omega\)
  which is fully faithful because
  each component \(I_d^{<\infty}\colon C_d\hookrightarrow C\)
  of the natural transformation was fully faithful.
  Finally, we pass to categories of sections over \(\omega\),
  and obtain a functor
  \begin{equation}
    \Fun_\omega(\omega, P) \xrightarrow{\overline{I}\circ -}
    \Fun_\omega(\omega, \omega\times C)\simeq \Fun(\omega, C)
  \end{equation}
  by postcomposition, which is again fully faithful.

  By construction, the image consists of those sequences \(X_\bullet\)
  such that \(X_d\in C_d\) for all \(d\in \naturals\).
  Such a sequence corresponds to a $p$-\emph{cartesian} section --- %
  i.e., an object of \(C^R\) --- %
  if and only if each arrow \(X_d\to X_{d+1}\) is $p$-cartesian
  (because each arrow in \(\omega\) is a composite of arrows of the form \(d\to d+1\)).
  which is equivalent to saying that the induced map
  \(X_d\to R_dX_{d+1}\) is an isomorphism
  (because the $p$-cartesian transport functor along $d\to d+1$ is precisely $R_d\colon C_{d+1}\to C_d$).
\end{proof}

\begin{lemma}
  \label{lem:C-eventually-constant}
  \begin{enumerate}
  \item[]
  \item
    For each $d\in \naturals$, the evaluation functor
    $R_d\colon C^R\subset \Fun(\omega, C)\xrightarrow{\ev_d}C_d$
    has a fully faithful left adjoint $I_d^R\colon C_d\hookrightarrow C^R$
    which sends an object \(X \in C_d\)
    to the unique\footnote{
      unique in $C^R$, not in the ambient category $\Fun(\omega,C)$
    } sequence in $C^R$ that has $X$ in degree $d$ and is constant afterwards,
    namely
    \begin{equation}
      I^R_dX \coloneqq
      (R_0 X \to \cdots \to R_{d-1} X \to X \xrightarrow{=} X \xrightarrow{=} \cdots) 
      \in C^R.
    \end{equation}
  \item
    \label{it:I^R-proof}
    By taking the union over \(d\to \infty\),
    we obtain a fully faithful functor
    \begin{equation}
      I^R_{<\infty}\coloneqq\bigcup_d I^R_d\colon C\hookrightarrow C^R,
      \quad X\mapsto R_\bullet X.
    \end{equation}
    which identifies \(C\) with the full subcategory of
    \(C^R\subseteq \Fun(\omega,C)\)
    spanned by the eventually constant sequences.
  \item
    Dually, we have a fully faithful inclusion
    \begin{equation}
      I^L_{<\infty}\colon C\hookrightarrow C^L
    \end{equation}
    which identifies \(C\) with the full subcategory of
    $C^L\subseteq \Fun(\omega,C)$
    spanned by the eventually constant sequences.
  \end{enumerate}
\end{lemma}

\begin{proof}
  \begin{enumerate}
  \item
    We only have to show that for any $X\in C_d$
    and any sequence $X'_\bullet\in C^R$,
    any map $f\colon X\to X'_d$ in $C_d$
    extends uniquely to a dashed transformation as indicated
    \begin{equation}
      \begin{tikzcd}
        R_0X\ar[r]
        \ar[d,dashed]
        &
        \cdots\ar[r]
        &
        R_{d-1}X\ar[r]
        \ar[d,dashed]
        &
        X
        \ar[r,"="]
        \ar[d,"f"]
        &
        X\ar[r,"="]
        \ar[d,dashed]
        &
        \cdots
        \\
        X'_0
        \ar[r]
        &
        \dots
        \ar[r]
        &
        X'_{d-1}
        \ar[r]
        &
        X'_d
        \ar[r]
        &
        \ar[r]
        X'_{d+1}
        \ar[r]
        &\dots
      \end{tikzcd}
    \end{equation}
    This is immediate on the right of $f$
    and follows on the left by the universal property of
    the counit maps $X'_{n}\simeq R_nX'_d\to X'_d$ (for $0\leq n<d$).
  \item
    By construction, the image of $I_d^R$
    consists of those sequences that are constant starting from the index $d$;
    taking the union over all $d\to \infty$
    yields precisely the eventually constant sequences.
  \item
    Dual to \ref{it:I^R-proof}.
  \end{enumerate}
\end{proof}

\begin{construction}
  \label{constr:bimodule-L-R}
  We construct the bimodule (a.k.a.\ profunctor)
  \begin{equation}
    M\colon (C^R)^\op\times C^L \to \An,
    \quad
    X,Y \mapsto \lim_{i,j : \omega^\op} C(X_i,Y_j)
  \end{equation}
  as the composite
  \begin{align}
    (C^R)^\op\times C^L
    &\subseteq
      \Fun(\omega,C)^\op \times \Fun(\omega^\op,C)
    \\
    &\simeq
      \Fun(\omega^\op,C^\op) \times \Fun(\omega^\op,C)
    \\
    &\xrightarrow{-\times -}
      \Fun(\omega^\op\times \omega^\op, C^\op\times C)
    \\
    &\xrightarrow{\Map_C\circ -}
      \Fun(\omega^\op\times \omega^\op, \An)
      \xrightarrow{\lim_{\omega^\op\times\omega^\op}} \An
  \end{align}
  It is immediate from the construction that
  on the full subcategories \(C\subseteq C^R\) and \(C\subseteq C^L\)
  of eventually constant sequences,
  this bimodule is just the identity bimodule
  \begin{equation}
    M|_{C^\op\times C} \simeq C(-,-) \colon C^\op\times C\to \An,
  \end{equation}
  a.k.a.\ the hom-functor.
\end{construction}

\begin{proposition}
  \label{lem:biinductive-adjunction}
  The bimodule \(M\) of \Cref{constr:bimodule-L-R}
  is representable in both variables,
  and corresponds to an adjunction
  \begin{equation}
    \label{eq:adjunction:C-R-L}
    \begin{tikzcd}
      &C_{<\infty}\ar[dl,"I_{<\infty}^R"',hookrightarrow]\ar[dr,hookrightarrow,"I_{<\infty}^L"]
      \\
      C^R\ar[rr,"L"]\ar[rr,phantom,"\bot",bend right=15]&&C^L\ar[ll,"R",bend left]
    \end{tikzcd}
  \end{equation}
  under \(C_{<\infty}=C\),
  where the functors \(L\) and \(R\) are given explicitly via the formulas
  \begin{equation}
    \label{eq:formulas-L-R}
    L(X)_\bullet\simeq \colim_{i : \omega} L_\bullet X_i
    \quad
    \text{and}
    \quad
    R(Y)_\bullet\simeq \lim_{i : \omega^\op} R_\bullet Y_i,
  \end{equation}
  respectively.
\end{proposition}

\begin{proof}[Construction\footnotemark]\footnotetext{
    The reason this is a ``construction'' and not a ``proof''
    is that,
    while the representability statement of 
    \Cref{lem:biinductive-adjunction}
    is indeed just a proposition,
    the explicit identifications \eqref{eq:formulas-L-R}
    correspond exactly to the additional data of universal elements
    in $M(X_\bullet,\colim_i L_\bullet X_i)$ and in $M(\lim_{k}R_\bullet Y_k,Y_\bullet)$,
    respectively.
  }
  We show that for each object $Y_\bullet\in C^L$, the presheaf
  $M(-,Y_\bullet)\colon (C^R)^\op\to \An$
  is represented by $RY\coloneqq \lim_{k:\omega^\op}R_\bullet Y_k$;
  representability in the other variable is analogous.
  Functorially for each $X_\bullet : C^R$, this is witnessed by the isomorphism
  \begin{align}
    \label{eq:iso-bimodule-L-R}
    C^R(X_\bullet,\lim_{k:\omega^\op}R_\bullet Y_k)
    &\simeq
    \int^{i:\omega}C(X_i,\lim_{k}R_i Y_k)
    \\
    &\simeq \int^{i} \lim_{k}C(X_i,R_iY_k)
    \\
    &\simeq \int^{i} \lim_{k}C(X_i,Y_k)
    \\
    &\simeq \lim_{i:\omega^\op} \lim_{k:\omega^\op}C(X_i,Y_k)=M(X_\bullet,Y_\bullet)
  \end{align}
  Here the first line uses the embedding $C^R\subseteq \Fun(\omega^\op,C)$
  and the usual end-formula for mapping anima in functor categories;
  the second and third line use the universal properties of \(\lim_k\) and \(R_i\),
  respectively;
  the last line simplifies the end to a limit
  because it is now constant in the second variable\footnote{
    Recall that an end is just a limit over the twisted arrow category.
    Being constant in the second variable
    just means that the diagram in question factors through the projection
    $\mathrm{Tw}(\omega)\to \omega^\op$;
    this projection is in initial as a cocartesian fibration with weakly contractible fibers,
    hence induces an equivalence on limits.
  }.

  Finally, it is clear from the pointwise formulas
  \eqref{eq:formulas-L-R}
  that $L$ and $R$ 
  send eventually constant sequences to eventually constant sequences.
  Since $M|_{C^\op\times C}$ is just the identity bimodule,
  it follows that the adjunction $L\dashv R$ intertwines $I_{<\infty}^R$ and $I_{<\infty}^L$ as claimed.
\end{proof}

\begin{remark}
  \label{rem:counit-LR-explicit}
  Unraveling the construction one sees that
  the isomorphism \eqref{eq:iso-bimodule-L-R}
  is given explicitly via the formula
  \begin{equation}
    \left(f_\bullet\colon X_\bullet \to \lim_l R_\bullet Y_l\right)
    \mapsto \left(
    X_i \xrightarrow{f_i} \lim_l R_iY_l \xrightarrow{\pr_{l\coloneqq k}} R_iY_k \to Y_k
    \right)_{i,k}
  \end{equation}
  and its inverse by
  \begin{equation}
    \left(f_{i,k}\colon X_i\to Y_k \right)_{i,k}
    \mapsto
    \left(
      X_\bullet\xrightarrow{(\overline{f}_{\bullet,k})_k} \lim_{k}R_\bullet Y_k
    \right),
  \end{equation}
  where $\overline{f}_{\bullet,k}\colon X_\bullet \to R_\bullet Y_k$
  denotes the map induced by $f_{\bullet,k}\colon X_\bullet\to Y_k$.
  Thus it is not hard to explicitly describe the counit \(\epsilon\) and unit \(\eta\)
  of the adjunction \(L\dashv R\)
  by tracing the identity of $R$ and $L$ through the composite isomorphisms
  \begin{equation}
    C^R(R-,R-)\simeq M(R-,-)\simeq C^L(LR-,-)
    \quad\text{and}\quad
    C^R(-,RL-)\simeq M(-,L-)\simeq C^L(L-,L-),
  \end{equation}
  respectively:
  \begin{itemize}
  \item
    For any \(Y_\bullet \in C^L\) and each $d\in\naturals$, the component
    \begin{equation}
      (\epsilon Y)_d\colon 
      (LRY)_d\simeq \colim_{i:\omega} L_d \lim_{j:\omega^\op}R_i Y_j \to Y_d
    \end{equation}
    amounts to the compatible system
    \begin{equation}
      \left(L_d\lim_{j:\omega^\op} R_iY_j\xrightarrow{L_d\pr_{j\coloneqq d}} L_dR_i Y_d\to L_dY_d\simeq Y_d\right)_{i:\omega}
    \end{equation}
  \item
    Dually, for any \(X_\bullet\in C^R\) and each $d\in\naturals$,
    the component
    \begin{equation}
      (\eta X)_d\colon
      X_d\to \lim_{i:\omega^\op}R_d\colim_{j:\omega}L_iX_j\simeq (RLX)_d
    \end{equation}
    amounts to the compatible system
    \begin{equation}
      \left(
          X_d\simeq R_dX_d\to R_dL_i X_d \xrightarrow{R_d\incl_{j\coloneqq d}}R_d\colim_{j:\omega} L_iX_j
      \right)_{i:\omega^\op}
    \end{equation}
  \end{itemize}
  Here ``$\pr$'' and ``$\incl$'' (with appropriate indices)
  denote the respective structure maps from the limit or into the colimit.
\end{remark}

\begin{remark}
  For each fixed \(d\in\naturals\) we have the adjunctions
  \begin{equation}
    \begin{tikzcd}[row sep=large]
      &C_d
      \ar[dr,"I^L_d", bend left]
      \ar[dl,bend right,"I^R_d"']
      \\
      C^R
      \ar[ur,"R_d"']
      \ar[rr,"L"]
      &&
      C^L
      \ar[ul,"L_d"]
      \ar[ll,"R", bend left]
    \end{tikzcd}
  \end{equation}
  by \Cref{lem:biinductive-adjunction} (bottom)
  and 
  by \Cref{lem:C-eventually-constant} (left and right), respectively.
  By composing the adjunctions,
  we conclude that $R_d$ and $L_d$ have adjoints on both sides:
  \begin{equation}
    L_d L \dashv RI^L_d\simeq I_d^R\dashv R_d
    \quad\text{and}\quad
    L_d\dashv I^L_d\simeq LI_d^R\dashv R_dR.
  \end{equation}
  For the middle identifications we use that the adjunction $L\dashv R$
  intertwines the inclusions $I_d^R$ and $I_d^L$
  of sequences that are constant above degree $d$.
\end{remark}

\subsection{... with a bias}
\label{sec:biinductive}

\begin{definition}
  \label{def:bias}
  A \emph{bias} on a biinductive system \(C_\bullet\)
  is a sequence
  \begin{equation}
    C^\simeq \subseteq
    \cdots
    \subseteq  S_n \subseteq
    \cdots
    \subseteq S_1\subseteq S_0\subseteq S_{-1} \coloneqq C
  \end{equation}
  of wide subcategories satisfying the following conditions:
  \begin{enumerate}[label=(S\arabic*), ref=(S\arabic*)]
  \item
    \label{axiom:S-isos}
    The \(S_\bullet\) jointly detect isomorphisms, i.e.,
    \begin{equation}
      C^\simeq=\bigcap_{n\in \naturals}S_n
    \end{equation}
    (the intersection is considered as wide subcategories of $C$).
  \item
    \label{axiom:S-unit}
    For each object \(X\in C\) and each \(d\in \naturals\),
    we have
    \begin{equation}
      \left(
        X\xrightarrow{\eta_d X} L_d X
      \right)
      \in S_d
      \quad \text{and}\quad
      \left(
        R_d X\xrightarrow{\epsilon_d X} X
      \right)
      \in S_d,
    \end{equation}
    where \(\eta_d\) and \(\epsilon_d\) denote the unit and counit
    of the adjunctions
    \(L_d\dashv I_d^{<\infty}\)
    and
    \(I_d^{<\infty}\dashv R_d\), respectively.
  \item
    \label{axiom:S-countable}
    For each \(d,n\in\naturals\),
    the arrows in \(S_n\cap C_d\) are closed under countable compositions.
    This means that we have
    \begin{equation}
      \left(\incl_{i\coloneqq 0}\colon X_0\to \colim_{i:\omega}X_i\right) \in S_n
      \quad
      \text{and}
      \quad
      \left(\pr_{i\coloneqq 0}\colon \lim_{i:\omega^\op}Y_i\to Y_0 \right) \in S_n,
    \end{equation}
    for each diagram
    \(X_\bullet\colon \omega \to C_d\cap S_n\)
    and
    \(Y_\bullet\colon \omega^\op \to C_d\cap S_n\), respectively. 
    (Here the limit/colimit is computed in \(C_d\), which has sequential (co)limits by assumption.)
  \item
    \label{axiom:S-cancellation}
    Given \(n\in \naturals\) and composable arrows
    $X\xrightarrow{f} Y\xrightarrow{g} Z$
    in \(C_n\),
    we have the following cancellation properties:
    \begin{align}
      f\in S_{n-1}\,\, \wedge\,\, gf\in S_n \quad\implies\quad g\in S_n
      \\
      g\in S_{n+1}\,\,\wedge\,\, gf\in S_n \quad\implies\quad f\in S_n
    \end{align}
  \item
    \label{axiom:S-L}
    For each \(n,d\in \naturals\) and each arrow $f$ in $C$, we have 
    \begin{equation}
      f\in S_n\quad\implies \quad L_df\in S_n.
    \end{equation}
  \item
    \label{axiom:S-R}
    For each \(n \in \naturals\) and each arrow $f$ in $C$, we have
    \begin{equation}
      f\in S_n\quad\iff\quad R_n f\in S_n.
    \end{equation}
  \end{enumerate}
\end{definition}

\begin{remark}
  A wide subcategory $S_n$ of $C$ is the same datum as a subanima
  $S_n\subseteq \Mor C$
  that contains all isomorphisms and is closed under composition.
  This translation justifies writing an expression such as
  $f\in S_n$ rather than the more verbose $f \in \Mor (S_n)$
  (given an arrow $f$ of $C$).
\end{remark}

For the remainder of the section,
let \((C_\bullet,S_\bullet)\) be a biased biinductive system.
In the presence of the other axioms,
Axioms~\ref{axiom:S-cancellation} and~\ref{axiom:S-R}
admit the following straightforward strenghtenings.

\begin{lemma}
  \begin{enumerate}[label=(S6'), ref =(S6')]
  \item
    \label{axiom:S-R-new}
    For each \(n\leq d\) and each arrow $f$ in $C$, we have
    \begin{equation}
      f\in S_n\quad\iff\quad R_d f\in S_n
    \end{equation}
    (not just when $d=n$).
  \end{enumerate}
  \begin{enumerate}[label=(S4'), ref =(S4')]
  \item 
    \label{axiom:S-cancellation-new}
    Given composable arrows \(X\xrightarrow{f}Y\xrightarrow{g}Z\) in \(C\)
    and any \(n\in \naturals\), we have:
    \begin{align}
      f\in S_{n-1}\,\, \wedge\,\, gf\in S_n \quad\implies\quad \phantom{R_n}g\in S_n
      \tag{easy cancellation}
      \\
      g\in S_{n+1}\,\,\wedge\,\, gf\in S_n \quad\implies\quad L_nf\in S_n
      \tag{hard cancellation}
    \end{align}
    (not just when \(X,Y,Z\in C_n\subseteq C\)).
  \end{enumerate}
\end{lemma}

\begin{proof}
  \begin{itemize}
  \item[]
  \item
    We get \ref{axiom:S-R-new} as
    \begin{equation}
      f\in S_n \xLeftrightarrow{\ref{axiom:S-R}} R_nf \in S_n \iff R_nR_df\in S_n
        \xLeftrightarrow{\ref{axiom:S-R}} R_df \in S_n,
    \end{equation}
    by applying Axiom~\ref{axiom:S-R} to $f$ to $R_df$,
    and using \(R_nR_d\simeq R_n\) in the middle.
  \item
    We deduce the easy cancellation as
    \begin{align}
      f\in S_{n-1}\,\wedge\, gf\in S_n
      &\iff 
        R_n f\in S_{n-1}\,\wedge\, R_n(gf)\in S_n
      &
        \ref{axiom:S-R-new}
      \\
      &\implies
        R_ng \in S_n
      &
        \ref{axiom:S-cancellation} 
        \\
      &\iff g \in S_n
        &
        \ref{axiom:S-R-new},
    \end{align}
    by applying the already proven \ref{axiom:S-R-new}
    to $f$ and $gf$,
    and applying the original Axiom~\ref{axiom:S-cancellation}
    to the composable arrows
    \(R_nX\xrightarrow{R_nf}R_nY\xrightarrow{R_ng}R_nZ\) in \(C_n\).
  \item
    We deduce the hard cancellation as
    \begin{align}
      g\in S_{n+1}\,\wedge\, gf\in S_n
      &\implies
        L_ng\in S_{n+1}\,\wedge\, L_n(gf)\in S_n
      &
        \ref{axiom:S-L}
      \\&\implies 
      L_nf\in S_n
      &
        \ref{axiom:S-cancellation},
    \end{align}
    by applying Axiom~\ref{axiom:S-L} to $g$ and $gf$,
    and applying the original Axiom~\ref{axiom:S-cancellation}
    to the composable arrows
    \(L_nX\xrightarrow{L_nf}L_nY\xrightarrow{L_ng}L_nZ\) in \(C_n\).
    \qedhere
  \end{itemize}
\end{proof}

We can characterize isomorphisms in \(C^L\) using the system \(S_\bullet\).

\begin{lemma}
  \label{lem:isos-in-C-L}
  Let \(f_\bullet\colon Y_\bullet \to Y'_\bullet\) be an arrow in \(C^L\).
  The following are equivalent:
  \begin{enumerate}
  \item
    $f_\bullet$ is an isomorphism in \(C^L\).
  \item
    \label{it:f-n-iso}
    For all \(d\in\naturals\), the arrow \(f_d\colon Y_d\to Y'_d\)
    is an isomorphism.
  \item
    \label{it:f-n-surj}
    For all \(d\in\naturals\), we have \(f_d\in S_d\).
  \item
    \label{it:f-n-m-surj}
    For all \(n\in\naturals\) there exists \(m\geq n\)
    such that \(L_n f_m\in S_n\).
  \end{enumerate}
\end{lemma}

\begin{proof}
  The second condition is just a reformulation of the first
  using the embedding
  \(C^L\subseteq \Fun(\omega^\op,C)\)
  and the fact that isomorphisms in functor categories are detected objectwise.
  
  Clearly, the last three conditions get progressively weaker,
  so we only have to show \ref{it:f-n-m-surj}$\implies$\ref{it:f-n-iso}.
  Fix \(d\in \naturals\).
  Then for every \(n\geq d\) there exists \(m\geq n\) with
  \(L_nf_m\in S_n\), hence \(f_d\simeq L_d L_n f_m \in S_n\)
  by Axiom~\ref{axiom:S-L};
  hence \(f_d\) is indeed an isomorphism because the \(S_n\) jointly detect isomorphisms
  (Axiom~\ref{axiom:S-isos}).
\end{proof}

\begin{lemma}
  \label{lem:S-structure-maps}
  \begin{enumerate}
  \item[]
  \item
    For 
    each \(Y_\bullet\in C^L\) and each \(i\geq d\) we have
    \begin{equation}
      \left(
        \pr_{j\coloneqq d}\colon (RY)_i\simeq  \lim_{j:\omega^\op} R_i Y_j\to R_iY_d
      \right)
      \in S_d.
    \end{equation}
  \item
    For each \(X_\bullet\in C^R\) and each \(i\geq d\) we have
    \begin{equation}
      \left(
        \incl_{j\coloneqq d}\colon L_i X_d \to \colim_{j:\omega} L_iX_j\simeq (LX)_i
      \right)
      \in S_d.
    \end{equation}
  \end{enumerate}
\end{lemma}

\begin{proof}
  \begin{enumerate}
  \item
    For each \(j\geq d\), the structure 
    map \(y_{j}\colon Y_{j+1}\to Y_{j}\)
    of \(Y_\bullet\) is identified with the unit
    \begin{equation}
      \eta_j Y_{j+1}\colon Y_{j+1}\xrightarrow{\eta_j Y_{j+1}} L_jY_{j+1}\simeq Y_j
    \end{equation}
    by the explicit identification \(C^L\subseteq \Fun(\omega,C)\)
    of \Cref{lem:C-R-L-emb-Fun};
    hence \(y_j\simeq \eta_jY_{j+1}\in S_j\subseteq S_d\) by Axiom~\ref{axiom:S-unit}.
    Since \(i\geq d\), we still have \(R_i(y_j)\in S_d\)
    by the forward direction of Axiom~\ref{axiom:S-R-new}.
    The claim \(\pr_{j\coloneqq d}\in S_d\) then follows from Axiom~\ref{axiom:S-countable},
    because \(\pr_{j\coloneqq d}\) is precisely the countable composition
    \begin{equation}
      \pr_{j\coloneqq d}\colon
      \lim_{j:\omega^\op}R_iY_j \simeq \lim_{j\geq d} R_iY_j
      \to \dots
      \xrightarrow{R_i(y_{d+1})}
      R_i Y_{d+1}\xrightarrow{R_i(y_d)} R_i Y_d
    \end{equation}
    of these arrows in \(S_d\cap C_i\).
  \item
    The proof of the second statement is  dual,
    by invoking Axiom~\ref{axiom:S-L}
    instead of the forward direction of Axiom~\ref{axiom:S-R-new}.
    \qedhere
  \end{enumerate}
\end{proof}

\begin{lemma}
  \label{lem:reformulate-S}
  Let \(f_\bullet\colon X_\bullet\to X'_\bullet\) be an arrow in \(C^R\).
  The following are equivalent:
  \begin{enumerate}
  \item
    \label{it:S-all-m}
    For each \(n\in \naturals\) and each $m>n$, we have \(L_nf_m\in S_n\).
  \item
    \label{it:S-n+1}
    For each \(n\in\naturals\), we have \(L_nf_{n+1}\in S_n\).
  \item
    \label{it:S-original}
    For each \(n\in \naturals\), there exists an \(n\leq m\) with \(L_nf_{m}\in S_n\).
  \end{enumerate}
\end{lemma}

\begin{proof}
  Clearly the conditions get progressively weaker,
  so it suffices to prove the implication
  \ref{it:S-original}\(\implies\) \ref{it:S-all-m}.
  So we consider any map \(f\) satisfying \ref{it:S-original}
  and show that \(L_nf_m\in S_n\) for all \(n<m\).
  
  Fix \(n\in \naturals\). By assumption, we may choose
  \(m_0\geq n\) such that \(L_nf_{m_0}\in S_n\).
  Then for any \(n<m'\leq m_0\) and any \(m_0\leq m'' \)
  we may consider the following two commutative squares in \(C_n\)
  \begin{equation}
    \begin{tikzcd}
      L_nX_{m'}
      \ar[d,"L_nf_{m'}"]
      \ar[r,"\in S_{n+1}"]
      &
      L_nX_{m_0}
      \ar[d,"L_nf_{m_0}\in S_n"]
      \ar[r,"\in S_{n}"]
      &
      L_nX_{m''}
      \ar[d,"L_nf_{m''}"]
      \\
      L_nX'_{m'}
      \ar[r,"\in S_{n+1}"]
      &
      L_nX'_{m_0}
      \ar[r,"\in S_{n}"]
      &
      L_nX'_{m''}
    \end{tikzcd}
  \end{equation}
  All horizontal arrows lie in $S_{n+1}$ or $S_n$ as indicated:
  indeed, these arrows arise by applying $L_n$ to a counit \(X_k\simeq R_k X_l\to X_l\)
  for \(l=m_0 \geq k= m'>n\) (on the left)
  or  \(l =m''\geq k=m_0\geq n\) (on the right);
  this counit lies in
  in \(S_{n+1}\) or \(S_n\), respectively (by Axiom~\ref{axiom:S-unit}),
  and applying $L_n$ does not change this (by Axiom~\ref{axiom:S-L}).
  It follows from the cancellation axioms (Axiom~\ref{axiom:S-cancellation})
  that both \(L_nf_{m'}\) and \(L_nf_{m''}\) lie in \(S_n\).
  (In the edge case \(n=m_0\), there is no \(m'\);
  in this case we only consider the right square of the diagram,
  which still behaves as described.)
  Since \(n<m'\leq m_0\)  and \(m_0\leq m''\) were arbitary,
  we conclude that indeed \(L_nf_m\in S_n\) holds for all \(m>n\).
\end{proof}

\begin{definition}
  \label{def:S}
  We write 
  \begin{equation}
    S\coloneqq\{f_\bullet\colon X_\bullet\to X'_\bullet
    \mid \forall_{n\in \naturals}: L_nf_{n+1}\in S_n\}
  \end{equation}
  for the collection of those arrows in \(C^R\)
  that satisfy any of the equivalent conditions of \Cref{lem:reformulate-S}.
\end{definition}

\begin{lemma}
  \label{lem:S-2-out-of-3}
  The collection \(S\) defines a wide subcategory \(S\subseteq C^R\)
  that satisfies 2-out-of-3.
\end{lemma}

\begin{proof}
  From any of the three characterizations it is clear that all isomorphisms
  lie in \(S\).
  It remains to show that the class \(S\) is closed under composition and cancellation.

  Consider two composable maps
  \(X\xrightarrow{f}X'\xrightarrow{g}X''\).
  It is clear from either characterization \ref{it:S-all-m} or \ref{it:S-n+1}
  of Lemma~\ref{lem:reformulate-S} that \(f,g\in S\) implies \(gf\in S\);
  similarly \(f\in S\) and \(gf\in S\) imply \(g\in S\)
  by the easy cancellation axiom.
  
  So let us assume that \(g,gf\in S\) and show \(f\in S\).
  For any $n\in \naturals$ we consider the composable maps
  \begin{equation}
    \begin{tikzcd}
      L_nX_{n+2}
      \ar[r,"L_nf_{n+2}"]
      &
      L_nX'_{n+2}
      \ar[r,"L_ng_{n+2}"]
      &
      L_nX''_{n+2}
    \end{tikzcd}
  \end{equation}
  where we know
  that the composite \(L_n(gf)_{n+2}\) lies in \(S_n\) 
  (by choosing $m\coloneqq n+2$ in condition \ref{it:S-all-m})
  and that the map \(L_ng_{n+2}\simeq L_nL_{n+1}g_{n+2}\) lies in \(S_{n+1}\)
  (by condition \ref{it:S-n+1} for $n+1$ and Axiom~\ref{axiom:S-L}).
  Thus it follows from the hard cancellation axiom
  that also \(L_n{f_{n+2}}\) lies in \(S_n\).
  Since \(n\in\naturals\) was arbitrary, this implies that \(f\in S\)
  (using characterization \ref{it:S-original}, choosing \(m\coloneqq n+2\)).
\end{proof}

\begin{theorem}
  \label{thm:bias-localization}
  Let \((C_\bullet,S_\bullet)\) be a biased biinductive system on \(C\).
  Then in the adjunction
  \begin{equation}
    L\colon C^R \rightleftarrows C^L : R
  \end{equation}
  of \Cref{lem:biinductive-adjunction} the following hold:
  \begin{enumerate}
  \item
    The right adjoint \(R\) is fully faithful.
  \item
    The left adjoint \(L\) precisely inverts the arrows in \(S\), i.e.,
    \begin{equation}
      L(f)\in (C^L)^\simeq \iff f\in S
    \end{equation}
    for all arrows \(f\) of \(C^R\).
  \end{enumerate}
  In particular, the adjunction exhibits $C^L$ as the localization
  \begin{equation}
    C^R[S^{-1}]\xrightarrow{\simeq} C^L
  \end{equation}
  of \(C^R\) at $S$.
\end{theorem}

\begin{proof}
  \begin{enumerate}
  \item
    We need to show that for each \(Y_\bullet\in C^L\),
    the counit \(\epsilon Y\colon LR Y\to Y \) is an isomorphism.
    By \Cref{lem:isos-in-C-L} it suffices to show that
    for each fixed \(d\in \naturals\) we have 
    \((\epsilon Y)_d \in S_{d}\),
    in fact, our proof will even yield $(\epsilon Y)_d \in S_{d+1}$.

    Now for each \(i\geq d+1\), we consider
    the commutative square
    \begin{equation}
        \cdsquareOpt
        {\lim\limits_{j\geq d+1}R_iY_j}
        {\lim\limits_{j\geq d+1} R_{i+1} Y_j}
        {R_iY_{d+1}}
        {R_{i+1}Y_{d+1}}
        {"t_i",dashed}
        {"S_{d+1}\ni\pr_{j\coloneqq d+1}"'}
        {"\pr_{j\coloneqq d+1}\in S_{d+1}"}
        {"\sim"}
    \end{equation}
    in \(C_{i+1}\),
    where $t_i$ is the map induced by the system
    \((R_iY_j\to R_{i+1}Y_j)_{j\geq d+1}\)
    upon passing to the limit \(j\to \infty\).
    \Cref{lem:S-structure-maps} yields that both vertical maps lie in \(S_{d+1}\).
    Moreover, the lower horizontal map is an isomorphism,
    because \(Y_{d+1}\in C_{d+1}\subseteq C_i\)
    so that \(R_iY_{d+1}\simeq R_{i+1}Y_{d+1}\simeq Y_{d+1}\).
    By the hard cancellation we can thus deduce that,
    while the top horizontal arrow \(t_i\) might not itself lie in \(S_d\),
    we nonetheless have \(L_dt_i\in S_d\).

    Passing to the countable composition of the \((L_dt_i)_{i\geq d+1}\),
    we see (Axiom~\ref{axiom:S-countable})
    that the left vertical arrow in the following triangle
    lies in \(S_d\)
    \begin{equation}
      \begin{tikzcd}[column sep=large]
        {\colim\limits_{i\geq d+1}L_d\lim\limits_{j\geq d+1}R_iY_j}
        \ar[r,"(\epsilon Y)_d"]
        &
        {Y_d}
        \\
        {L_d\lim\limits_{j\geq d+1}R_{d+1}Y_j}
        \ar[u,"S_{d}\ni\incl_{i\coloneqq d+1}"]
        \ar[ur,"s_d"',dashed]
      \end{tikzcd}
    \end{equation}
    By the explicit description of the counit \((\epsilon Y)_d\)
    (\Cref{rem:counit-LR-explicit}) with \(i\coloneqq d+1\),
    we see that the dashed composite \(s_d\) is the map
    \begin{equation}
      L_d\lim_j R_{d+1}Y_j \xrightarrow{L_d\pr_{j\coloneqq d}} L_dR_{d+1}Y_d\simeq Y_d,
    \end{equation}
    which, by factoring $\pr_{j\coloneqq d}$ through $\pr_{j\coloneqq d+1}$,
    is identified with the composite
    \begin{equation}
      L_d\lim_{j\geq d+1}R_{d+1}Y_j\xrightarrow{L_d\pr_{j\coloneqq d+1}}
      L_dR_{d+1}Y_{d+1}\xrightarrow{\sim} L_d R_{d+1} Y_d\simeq Y_d,
    \end{equation}
    where the second arrow is (identified with) the structure isomorphism
    \(L_dY_{d+1}\xrightarrow{\sim} Y_d\).
    Since \(\pr_{j\coloneqq d+1}\in S_{d+1}\) by \Cref{lem:S-structure-maps},
    we also have \(s_d\in S_{d+1}\) by Axiom~\ref{axiom:S-L}.
    Finally we conclude \((\epsilon Y)_d\in S_{d+1}\) by the easy cancellation,
    which is what we set out to prove.
  \item
    We start by showing that each arrow
    \((f_\bullet\colon X_\bullet \to X'_\bullet)\in S\)
    gets sent by \(L\) to an isomorphism,
    i.e., that \((Lf)_d\) is invertible for each \(d\in \naturals\).
    For each \(n\geq d\), consider the commutative square
    \begin{equation}
      \begin{tikzcd}
        {\colim_{i:\omega} L_dX_i}
        \ar[r,"(Lf)_d"]
        &
        {\colim_{i:\omega} L_dX'_i}
        \\
        {L_dX_{n+1}}
        \ar[r,"L_df_{n+1}"]
        \ar[u]
        &
        {L_dX'_{n+1}}
        \ar[u]
      \end{tikzcd}
    \end{equation}
    in \(C_d\) induced by the system \((L_d f_i)_{i\geq n+1}\) upon passing to the colimit \(i\to \infty\).
    Observe the following:
    \begin{itemize}
    \item
      By the assumption \(f\in S\), we have \(L_nf_{n+1}\in S_n\),
      hence also \(L_df_{n+1}\simeq L_dL_nf_{n+1}\in S_n\) by Axiom~\ref{axiom:S-L}.
    \item
      The two vertical arrows lie in \(S_{n-1}\) (in fact, even in $S_{n+1}$), because
      they arise as the countable composition
      (for $i\geq n+1$)
      of the arrows
      \begin{equation}
        L_d(X_{i}\simeq R_iX_{i+1}\xrightarrow{\epsilon_i X_{i+1}} X_{i+1})
        \quad\text{and}\quad
        L_d(X'_{i}\simeq R_iX'_{i+1}\xrightarrow{\epsilon_iX'_{i+1}} X'_{i+1}),
      \end{equation}
      respectively.
      (Here we use Axioms~\ref{axiom:S-unit}, \ref{axiom:S-L} and \ref{axiom:S-countable}
      in that order).
    \end{itemize}
    By composition and easy cancellation,
    we thus deduce that the top horizontal arrow
    \((Lf)_d\) also lies in \(S_n\).
    Since \(n\geq d\) was arbitrary,
    we see that \((Lf)_d\in \bigcap_n S_n\),
    hence \((Lf)_d\) is an isomorphism by Axiom\ref{axiom:S-isos}.\\

    For the converse, it suffices  to show that for each object \(X\in C^R\)
    the unit \(\eta X\colon X\to RLX\) lies in \(S\).
    Indeed, for any \(f_\bullet \colon X_\bullet\to X'_\bullet\)
    with invertible \(Lf\), one can then consider the naturality square 
    \begin{equation}
      \cdsquare
      {X}
      {X'}
      {RLX}
      {RLX'}
      {f}
      {\eta X}
      {\eta X'}
      {RL f}
    \end{equation}
    where the lower horizontal arrow is an isomorphism by assumption
    and the two vertical arrows lie in \(S\);
    since \(S\subset C^R\) is a wide subcategory closed under 2-out-of-3
    (\Cref{lem:S-2-out-of-3}),
    one then also gets \(f\in S\).

    With this in mind, let \(X\in C^R\);
    we need to show that for each \(d\in \naturals\),
    we have \(L_d(\eta X)_{d+1}\in S_d\).
    For each \(i\geq d+1\) we consider the commutative square
    \begin{equation}
      \begin{tikzcd}
        {\colim\limits_{j\geq d+1}L_{i+1}X_j}
        \ar[r,"u_i",dashed]
        &
        {\colim\limits_{j\geq d+1}L_{i}X_j}
        \\
        {L_{i+1}X_{d+1}}
        \ar[r]
        \ar[u,"\incl_{j\coloneqq d+1}"]
        &
        {L_{i} X_{d+1}}
        \ar[u,"\incl_{j\coloneqq d+1}"']
      \end{tikzcd}
    \end{equation}
    in \(C_{i+1}\),
    where $u_i$ is the map induced by the system
    \((L_{i+1}X_{j}\to L_iX_{j})_{j\geq d+1}\)
    upon passing to the colimit \(j\to \infty\).
    \Cref{lem:S-structure-maps} yields that both vertical maps lie in \(S_{d+1}\).
    Moreover, the lower horizontal map is an isomorphism,
    because \(X_{d+1}\in C_{d+1}\subseteq C_i\)
    so that \(X_{d+1}\simeq L_{i+1}X_{d+1}\simeq L_{i}X_{d+1}\).
    By the easy cancellation we can thus deduce that \(u_i\in S_{d+1}\),
    hence also \(R_{d+1}u_i\in S_{d+1}\) by the forward direction of Axiom~\ref{axiom:S-R}.
    Passing to the countable composition of the \((R_{d+1}u_i)_{i\geq d+1}\),
    we see (Axiom~\ref{axiom:S-countable})
    that the right vertical arrow in the following triangle lies in $S_{d+1}$:
    \begin{equation}
      \begin{tikzcd}
        X_{d+1}
        \ar[r,"(\eta X)_{d+1}"]
        \ar[dr,dashed]
        &
        \lim\limits_{i\geq d+1}R_{d+1}\colim_{j\geq d+1}L_i X_j
        \ar[d,"\pr_{i\coloneqq d+1}\in S_{d+1}"]
        \\
        &
        R_{d+1}\colim\limits_{j\geq d+1}L_{d+1}X_j
      \end{tikzcd}
    \end{equation}
    By the explicit description of the unit \((\eta X)_{d+1}\)
    (\Cref{rem:counit-LR-explicit}) with \(i\coloneqq d+1\),
    we see that the dashed diagonal arrow is obtained by applying
    \(R_{d+1}\) to the structure map
    \begin{equation}
      \incl_{j\coloneqq d+1}\colon X_{d+1}\simeq L_{d+1}X_{d+1}\to \colim\limits_{j\geq d+1} L_{d+1}X_j,
    \end{equation}
    hence lies in \(S_d\) (even \(S_{d+1}\))
    by \Cref{lem:S-structure-maps} and the forward direction of Axiom~\ref{axiom:S-R}.
    Finally we conclude \(L_d(\eta X)_{d+1}\in S_{d}\) by the hard cancellation,
    which is what we had to prove.
    \qedhere
  \end{enumerate}
\end{proof}

\subsection{Left and right $\infty$-categories}

We will simply write ``$\infty$-category''
(possibly with additional qualifiers) to speak about (various versions of)
what is usually called an ``$(\infty,\infty)$-category''
or ``$(\infty,\omega)$-category''.

\begin{definition}
    We define 
    \begin{equation}
      \Cat_{<\infty}\coloneqq \bigcup_{d\in \naturals}\Cat_d
      \coloneqq
      \colim(
      \Cat_0\xhookrightarrow{I^1}\Cat_1\hookrightarrow\cdots\xhookrightarrow{I^d}
      \Cat_d\hookrightarrow\cdots
      )
    \end{equation}
    as the category of \emph{finite-dimensional $\infty$-categories}.
    Each finite-dimensional \(\infty\)-category
    is just a $d$-category for some $d\in \naturals$.
\end{definition}

\begin{proposition}
  \label{prop:Cat-bias}
  \begin{enumerate}
  \item
    The sequence
    \begin{equation}
      \Cat_\bullet=
      \left(
        \An=\Cat_0\hookrightarrow \Cat_1\hookrightarrow \cdots 
        \hookrightarrow \Cat_d\hookrightarrow
      \right)
    \end{equation}
    is a biinductive system on $\Cat_{<\infty}$.
  \item
    Let $S_n\coloneqq \{\text{$n$-surjective}\}\subseteq \Cat_{<\infty}$
    denote the wide subcategory of $n$-surjective functors.
    Then $S_\bullet$ is a bias on the biinductive system \(\Cat_\bullet\).
  \end{enumerate}
\end{proposition}

\begin{proof}
  \begin{enumerate}
  \item
    Each functor $\Cat_{d}\hookrightarrow\Cat_{d+1}$ is fully faithful
    and has adjoints $L_d$ and $R_d$ by \Cref{thm:L-I-R}.
    Each $\Cat_d$ is presentable (see \Cref{rem:Cat-d-presentable}),
    hence in particular has sequential limits and colimits.
    The sequence $\Cat_\bullet$ exhausts $\Cat_{<\infty}$ by construction.
  \item
    \ref{axiom:S-isos} is \Cref{lem:d-cat-equiv-n-surj}.
    \ref{axiom:S-unit}, \ref{axiom:S-L} and \ref{axiom:S-R}
    follow from
    \Cref{prop:surjective-core} and \Cref{prop:surjective-localization}.
    \ref{axiom:S-cancellation} is \Cref{prop:cancellation-surj}.

    To conclude the proof, we prove \ref{axiom:S-countable}
    using the explicit description of
    \Cref{lem:sequential-homs} and induction on $n$.
    \begin{itemize}
    \item
      Let $X_0\to X_1\to \dots $ be a sequence of $n$-surjective maps of $d$-categories;
      we need to show that the countable composite $X_0\to X_\infty\coloneqq \colim_i X_i$
      is also $n$-surjective.

      For every object $x_\infty\in X_\infty$ in the sequential colimit,
      there merely exists an index $i\in \omega$ and an object $x_i\in X_i$
      with $x_i\mapsto x_\infty$;
      by the $0$-surjectivity of $X_0\to X_i$ there furthermore merely exists a lift to
      $X_0\ni x_0\mapsto x_i\mapsto x_\infty$;
      this proves the $0$-surjectivity of $\incl_{0}\colon X_0\to X_\infty$.
      If $n\geq 1$, then we consider $x_0,y_0\in X_0$
      and observe that
      $X_0(x_0,y_0)\to X_\infty(\incl_{0}x_0,\incl_{0}y_0)$ is again an $\omega$-indexed composition
      of $(n-1)$-surjective maps, hence $(n-1)$-surjective by induction.
    \item
      Let $\dots \to Y_1\to Y_0$ be a cosequence of $n$-surjective maps of $d$-categories;
      we need to show that the countable composite
      $Y_\infty\coloneqq \lim_i Y_i$ is also $n$-surjective.
      For every object $y_0\in Y_0$ there exist successive
      lifts $Y_{i+1}\ni y_{i+1}\mapsto y_i\in Y_i$ by $0$-surjectivity;
      hence by dependent choice there exists an object
      $y_\bullet\in Y_\infty$
      in the sequential limit with $y_\bullet \mapsto y_0$;
      this proves $0$-surjectivity of $Y\to Y_0$.
      If $n\geq 1$, then we consider $x_\bullet,y_\bullet\in Y_\infty$
      and observe that
      $Y(x_\bullet,y_\bullet)\to Y_0(x_0,y_0)$ is again an $\omega^\op$-indexed composition
      of $(n-1)$-surjective maps, hence $(n-1)$-surjective by induction.
      \qedhere
    \end{itemize}
  \end{enumerate}
\end{proof}

We therefore may specialize the abstract discussion of \Cref{sec:biinductive}
as follows:

\begin{definition}
  \label{def:infty-cat}
  \begin{itemize}
  \item[]
  \item
    The category of \emph{right}\footnote{
      Many authors consider this to be the ``correct''
      notion of $(\infty,\infty)$-categories,
      while the ``left'' one is sometimes considered to be somewhat pathological.
      For now, we treat the two notions on the same footing
      by using symmetric qualifiers ``right'' and ``left''.
      This will change in \Cref{sec:strong},
      where we treat $\Cat^R_\infty$
      as the default universe of discourse and $\Cat^L_\infty$
      as only one of various possible localizations/subcategories thereof.
    } \emph{$\infty$-categories} is 
    \begin{equation}
      \Cat_\infty^R\coloneqq
      \lim(\cdots\xrightarrow{R_d}\Cat_d\to\cdots \xrightarrow{R_1}\Cat_1\xrightarrow{R_0}\Cat_0)
    \end{equation}
  \item
    The category of \emph{left $\infty$-categories} is 
    \begin{equation}
      \Cat_\infty^L\coloneqq
      \lim(\cdots\xrightarrow{L_d}\Cat_d\to\cdots \xrightarrow{L_1}\Cat_1\xrightarrow{L_0}\Cat_0)
    \end{equation}
  \end{itemize}
  (Recall also the explicit description of these limits from \Cref{lem:C-R-L-emb-Fun}.)
\end{definition}

\begin{remark}
  \label{rem:cat-infty-presentable}
  \(\Cat_\infty^R\) and \(\Cat_\infty^L\)
  are the colimit of the sequence
  \begin{equation}
    \Cat_0\xhookrightarrow{I^1} \Cat_1\hookrightarrow \cdots \xhookrightarrow{I^d} \Cat_d\hookrightarrow \cdots
  \end{equation}
  computed in \(\PrL\) and $\PrR$, respectively\footnote{
    Note that the symbols ``$L$'' and ``$R$'' swap roles,
    because colimits in $\PrL$ are computed as limits in $\Cat$ of the right adjoints,
    and colimits in $\PrR$ are computed as limits in $\Cat$ of the left adjoints.
  }.
  In particular, both of them are again presentable categories.
  In contrast, the category $\Cat_{<\infty}$
  does not have sequential limits or colimits.
\end{remark}

\begin{definition}
  \label{def:infty-surjective}
  Let \(F_\bullet\colon C_\bullet \to D_\bullet\)
  be a map of right \(\infty\)-categories.
  We say that $F$ is
  \begin{itemize}
  \item
    \emph{$\infty$-surjective}\footnote{
      Called ``$\omega$-surjective'' by Loubaton, see \cite[Remark~1.2.13]{Loubaton-effectivity}.
    } if,
    for all $n\in \naturals$,
    the map $F_n\colon C_n\to D_n$ of $n$-categories is $n$-surjective.
  \item
    \emph{weakly \(\infty\)-surjective} if, 
    for all $n\in \naturals$,
    the map \(L_nF_{n+1}\colon L_nC_{n+1}\to L_nD_{n+1}\) of $n$-categories
    is $n$-surjective.
  \end{itemize}
\end{definition}

\begin{remark}
  \label{rem:weak-strong-surjective}
  As the terminology suggests, every $\infty$-surjective map
  is weakly $\infty$-surjective.
  Indeed, this follows from the characterization of
  \Cref{lem:reformulate-S}:
  for an $\infty$-surjective map one can just take $m\coloneqq n$
  in condition \ref{it:S-original}.

  For now, the definition of weak surjectivity might be a bit opaque,
  since it is defined in terms of the localization functors $L_n$,
  which can be hard to compute explicitly.
  But in \Cref{lem:L-iso-coind-inv} below
  we will give a more explicit characterization.
\end{remark}

Note that strenghtening the notion of $\infty$-surjective maps
by even one level makes it collapse.

\begin{lemma}
  \label{lem:d+1-surj-equiv}
  Let $F\colon C\to D$ be a map of right $\infty$-categories
  such that $F_d$ is $(d+1)$-surjective for each $d\in \naturals$.
  Then $F$ is an equivalence.
\end{lemma}

\begin{proof}
  Fix $n\in \naturals$; we need to show that $F_n\colon C_n\to D_n$ is an equivalence.
  By \Cref{lem:d-cat-equiv-n-surj} it suffices
  to show that $F_n$ is $d$-surjective for all $d\in \naturals$.
  Since we are assuming that $F_{d}$ is ${d+1}$-surjective,
  it follows from \Cref{prop:surjective-core}\ref{lem:surj-cores-k>d}
  (with $k=d+1$) that also $F_n\simeq R_nF_{d}$ is $(d+1)$-surjective,
  hence in particular $d$-surjective.
\end{proof}

We now obtain the main theorem of this article.

\begin{theorem}[Right vs.\ left $\infty$-categories]
  \label{thm:main}
  We have an adjunction
  \begin{equation}
    \label{eq:adjunction:C-R-L}
    \begin{tikzcd}
      &\Cat_{<\infty}\ar[dl,"I_{<\infty}^R"',hookrightarrow]\ar[dr,hookrightarrow,"I_{<\infty}^L"]
      \\
      \Cat^R_\infty\ar[rr,"L"]\ar[rr,phantom,"\bot",bend right=15]&&
      \Cat^L_\infty\ar[ll,"R",bend left,hookrightarrow]
    \end{tikzcd}
  \end{equation}
  \begin{itemize}
  \item
    compatible with the respective subcategories of
    finite-dimensional $\infty$-categories.
  \item
    The left adjoint $L$ precisely inverts the weakly $\infty$-surjective maps;
  \item
    the right adjoint $R$ is fully faithful.
  \end{itemize}
  It exhibits \(\Cat_\infty^L\) as the localization of $\Cat^R_\infty$
  at the weakly $\infty$-surjective maps.
\end{theorem}

\begin{proof}
  If we consider the biased biinductive system
  $(\Cat_\bullet,S_\bullet\coloneqq\{\bullet\text{-surjective}\})$
  of \Cref{prop:Cat-bias},
  the weakly $\infty$-surjective maps
  are precisely those that lie in the wide subcategory $S$ from \Cref{def:S}.
  Thus the theorem is a direct consequence of \Cref{thm:bias-localization}.
\end{proof}

We note that the localization functor $L$ has a non-trivial kernel
as the following example shows.
We will study the nature of this localization, as well as some other intermediate ones
in \Cref{sec:strong} below.

\begin{example}[$\infty$-category of spans]
  \label{ex:spans}
  For each anima $X$ and each $d\in \naturals$ we consider
  the $d$-category $\Span_d(\An/X)$ of \emph{spans}
  (a.k.a.\ \emph{correspondences}) of animae over $X$;
  see \cite[Definition 5.16]{HaugsengIterated} for a construction.
  It has the following recursive description:
  \begin{itemize}
  \item
    The objects of $\Span_d(\An/X)$ are animae over $X$.
  \item
    If $d\geq 1$, then for each $A,B\in (\An/X)^\simeq$
    the hom-$(d-1)$-categories are
    \begin{equation}
      \label{eq:Span-homs}
      \Span_d(\An/X)(A,B)\simeq \Span_{d-1}(\An/(A\times_X B));
    \end{equation}
    composition is induced by fiber products over $X$.
  \end{itemize}
  Moreover, we may consider the sequence
  \begin{equation}
    \Span_\infty(\An) \coloneqq \left(\Span_0(\An)\to \Span_1(\An)\to \dots\right)
    \in \Fun(\omega, \Cat_{<\infty}),
  \end{equation}
  where the maps $\Span_d(\An/X)\to\Span_{d+1}(\An/X)$
  are induced recursively from the core inclusion
  \begin{equation}
    \Span_0(\An/X)=(\An/X)^\simeq\to \Span_1(\An/X)
  \end{equation}
  in each hom (by replacing $X$ with $A\times_X B$ in the recursive step).
  Note that $\Span_\infty(\An/X)$ is actually a right $\infty$-category
  (in the sense of \Cref{lem:C-R-L-emb-Fun})
  because inductively one has
  $\Span_{d}(\An/X)\simeq R_d\Span_{d+1}(\An/X)$
  for all $d\in \naturals$.
  Finally we observe that the functor
  $\Span_0(\An/X)=(\An/X)^\simeq\to *$
  is trivially $0$-surjective
  (for example, witnessed by $\emptyset\in (\An/X)^\simeq$),
  hence it follows inductively from \ref{eq:Span-homs} that each
  $\Span_d(\An/X)\to *$ is $d$-surjective
  (again, by replacing $X$ by $A\times_X B$ in the induction step);
  in other words, $\Span_\infty(\An/X)\to *$ is $\infty$-surjective.

  We conclude that $\Span_\infty(\An/X)$ is a non-trivial
  right $\infty$-category
  (i.e., componentwise non-trivial as a sequence in $\Fun(\omega,\Cat_{<\infty})$)
  which becomes trivial under the localization functor $L$
  as a left $\infty$-category;
  in particular it does not lie in the image of $R$.
\end{example}

\section{Invertibility and completeness ``at \texorpdfstring{$\infty$}{infinity}''}
\label{sec:strong}

Recall that for finite $d\in\naturals$, a $d$-category $C$
is always assumed to be univalent,
which means that an arrow $x\to y$ in $C$
(of any dimension, between parallel lower-dimensional arrows)
comes from a (unique) identification $x\simeq y$
if and only if it is an isomorphism in the categorical sense.
This in turn means that there is only one reasonable notion of invertibility of arrows,
because there is no distinction between
``invertible up to isomorphism'' and ``invertible up to identification''.

In particular, the isomorphisms form the unique\footnote{
  To see uniqueness, note that any $K\subseteq \Mor_{\geq 1}C$
  satisfying \ref{recursion-formula-E}
  only consists of isomorphisms:
  for any arrow above dimension $d$ this is trivial;
  it follows for all others by downward induction.
} subanima
$K\subseteq \Mor_{\geq 1} C\coloneqq \coprod_{n\geq 1}\Mor_n C$
of arrows that contains all identifications
and satisfies the defining recursion formula:
\begin{enumerate} [label ={(E)}, ref = {(E)}]
\item
  \label{recursion-formula-E}
  An arrow $f\colon x\to y$ (of any dimension $n\geq 1$)
  lies in $K$ if and only if there merely exist
  \begin{itemize}
  \item
    $n$-arrows $g_-,g_+\colon y\to x$
  \item
    $(n+1)$-arrows
    $h_+\colon \id_x\to g_+f$ and $h_-\colon fg_-\to \id_y$ that lie in $K$.
  \end{itemize}
\end{enumerate}

When passing to the limit $d\to \infty$
the situation changes and new notions of invertibility arise.
For example, while the isomorphisms
still form the \emph{minimal} solution to the recursion formula
\ref{recursion-formula-E},
there are now new solutions,
including a unique \emph{maximal} one
consisting of the so called ``coinductive isomorphisms''
(see \Cref{prop:coinductive-equiv-char}).
We also consider variations of this recursion formula
yielding even larger subanimae of arrows
(i.e., weaker notions of invertibility),
such as that of $\omega$-isomorphisms
(see \Cref{def:alpha-invertible}).

Given a new ``notion of invertibility'',
viewed as a suitable\footnote{
  Usually this means that it contains all identifications,
  satisfies some version of the recursion formula
  \ref{recursion-formula-E},
  and is functorial in $C$.
} subanima of arrows $K(C)\subseteq \Mor_{\geq 1} C$
for each $C\in \Cat^R_\infty$,
we then also get new notions of completeness/univalence and surjectivity.
\begin{itemize}
\item
  A (right) $\infty$-category $C$ is called $K$-\emph{complete},
  if $K(C)$ only consists of the isomorphisms.
\item
  A map $F\colon C\to D$ is \emph{$\infty$-surjective up to} arrows in $K$
  if for every $k$-arrow of the form $f\colon Fx\to Fy$ in $D$
  there merely exists a $k$-arrow $f'\colon x\to y$ in $C$
  and a $(k+1)$-arrow $Ff'\to f$ in $K(D)$.
\end{itemize}

These notions in turn yield
a full subcategory of $\Cat^R_\infty$ (of $K$-complete objects)
and a localization of $\Cat^R_\infty$ (at the $\infty$-surjections up to $K$),
respectively.
These do not always agree:
For example, while every right $\infty$-category is tautologically complete 
with respect to identifications,
localizing $\Cat^R_\infty$ at the $\infty$-surjections (up to identifications)
yields the proper full subcategory of coinductively complete objects;
see \Cref{cor:strong}.

The goal of this section is to study various subcategories/localizations
obtained in this way and compare them to the reflective localization
$L\colon \Cat^R_\infty\to \Cat^L_\infty$ 
from our main theorem.

\subsection{Preliminaries}

To reflect its status as the universe of discourse in this section,
we will abbreviate $\Cat_\infty\coloneqq\Cat_\infty^R$
and just call its objects $\infty$-categories
(as opposed to \emph{right} $\infty$-categories).
Recall from \Cref{lem:C-R-L-emb-Fun}
that we consider an $\infty$-category $C$ as a sequence
\begin{equation}
  C = (C_0\to C_1\to \dots\to C_d\to \dots) \in \Fun(\omega, \Cat_{<\infty})
\end{equation}
where each $C_d$ is a $d$-category
and such that $C_d\xrightarrow{\sim} R_d C_{d+1}$.

We can talk about $k$-arrows of an $\infty$-category $C$:
by this we just mean $k$-arrows of $C_k$
(or, equivalently, of $C_d$ for any $d\geq k$).
We set
\begin{equation}
  \Mor_kC\coloneqq \Mor_kC_k\simeq \colim_{d:\omega} \Mor_k C_d
  \quad
  \text{and}
  \quad
  \Par_kC\coloneqq \Par_kC_k\simeq \colim_{d:\omega} \Par_k C_d.
\end{equation}

For each $d\in \naturals$, we have the \emph{(categorical) suspension} functor
defined as the composition
\begin{equation}
  \label{eq:Susp-ladj-hom}
  \Susp\colon \Cat_d\xrightarrow{\Susp_{0,1}} \{0,1\}/\Cat_{d+1}\to \Cat_{d+1},
\end{equation}
where $\Susp_{0,1}$ is the left adjoint to the hom-category functor
\begin{equation}
  (f\colon \{0,1\}\to C)\mapsto C(f0,f1);
\end{equation}
see also \cite[Definition~4.3.21]{GH}.
Given a $d$-category $C$,
the $(d+1)$-category \(\Susp C\) is uniquely determined by the data
$(\Susp C)^\simeq\simeq \{0,1\}$ and
\begin{equation}
  \Cat_{d}
  \ni
  (\Susp C) (i,j)\simeq
  \begin{cases}
    \{\id_i\}, \quad \text{ if } i=j
    \\
    \emptyset, \quad \text{ if } i=1 , j=0
    \\
    C, \quad \text{ if } i=0, j=1
  \end{cases}
\end{equation}
(There is only one way to define the composition structure.
\footnote{
  A more precise way to say this is to consider the map $*\to O(\{0,1\})$
  of operads that selects the color $(0,1)\in \{0,1\}\times \{0,1\}$.
  Restriction and operadic left Kan extension yield an adjunction
  $\Susp_{0,1}\colon\Cat_d\rightleftarrows\flCat[\{0,1\}]{\Cat_{d}}:\ev_{(0,1)}$
  where the left adjoint is fully faithful
  because $*\to O(\{0,1\})$ is.
  The image of the left adjoint consists precisely of those
  $\{0,1\}$-flagged $\Cat_d$-categories
  $D$ with $D(1,0)\simeq \emptyset$ and $D(0,0)\simeq *\simeq D(1,1)$;
  hence such a $D$ is uniquely determined by its evaluation
  $D(0,1)\coloneqq\ev_{(0,1)}(D)$ in $\Cat_d$.
})
The categorical suspension functor extends to $\infty$-categories
via the formula
\begin{equation}
  \label{eq:Susp-infty}
  \Susp\colon \Cat_\infty\to\Cat_\infty,
  \quad
  C_\bullet \mapsto
  \Susp C_{\bullet-1}
  \coloneqq
  (\{0,1\}\to \Susp C_0\to \Susp C_1\to \dots\to \Susp C_{d-1}\to \dots)
\end{equation}
(with the convention $C_{-1}=\emptyset$).

The \emph{$d$-globe} is $\globe{d}\coloneqq \Susp^d*\in \Cat_d$
and its \emph{boundary} is $\boundary\globe{d}\coloneqq\Susp^d\emptyset \in \Cat_{d-1}$.
(The dimension shift arises because the category 
$\Susp(\emptyset)\simeq\sphere{0}$
exceptionally lies in $\Cat_0\subseteq \Cat_1$.)
We denote by $\iota_d\colon \boundary\globe{d}\hookrightarrow \globe{d}$
the canonical ($(d+1)$-fully faithful) inclusion
(induced by $\emptyset\hookrightarrow *$).
The $1$-globe $\globe{1}=\Susp(*)=\{0\to 1\}$ is the walking arrow.

\begin{remark}
  \label{rem:globe-surj}
  Note that a map $\boundary\globe{n}\to C$
  into an $\infty$-category $C$
  amounts to a parallel pair $(x,y)$ of $(n-1)$-arrows;
  an extension to $\globe{n}$ amounts to an $n$-arrow $x\to y$.
  More generally, for $F\colon C\to D$ and \(f\colon Fx\to Fy\), a dashed lift
  \begin{equation}
    \cdsquare[lifting]
    {\boundary\globe{n}}
    {C}
    {\globe{n}}
    {D}
    {{(x,y)}}
    {}{F}
    {f}
  \end{equation}
  amounts to an $n$-arrow $f'\colon x\to y$ with $Ff'\simeq f$.
\end{remark}

\begin{definition}
  A map $F\colon C\to D$ of $\infty$-categories is called
  \emph{$n$-fully faithful}\footnote{
    We follow the terminology of \cite{Loubaton-effectivity}.
    Beware that these maps are called $(n-2)$-faithful
    (without ``fully'') in \cite{LMRSW}.
  } if
  \begin{itemize}
  \item
    $n=0$ and $F$ is an equivalence, or
  \item
    $n\geq 1$ and for all $x,y\in C$ the induced map $C(x,y)\to D(Fx,Fy)$
    is $(n-1)$-fully faithful.
  \end{itemize}
\end{definition}

\begin{theorem}
  \label{thm:factorization-surj-ff}
  [\cite{Loubaton-effectivity}, Definition~1.2.6 or \cite{LMRSW}, Theorem~5.3.7]
  For every $n\in \naturals$, we have a unique factorization system
  \begin{equation}
    (\text{$n$-surjective}, \text{$(n+1)$-fully faithful})
  \end{equation}
  on the category of $\infty$-categories.
\end{theorem}

\begin{remark}
  For $n=0$ this is just the unique factorization system
  \begin{equation}
    (\text{surjective on objects, fully faithful}).
  \end{equation}
  \end{remark}

\begin{corollary}
  \label{cor:surj-colims-cobase}
  Let $n\in \naturals\cup\{\infty\}$.
  \begin{enumerate}
  \item
    The $n$-surjective arrows are closed under
    composition, colimits and cobase change.
  \item
    The suspension of an $n$-surjective map is $(n+1)$-surjective
    (with $\infty+1=\infty$).
  \end{enumerate}
\end{corollary}

\begin{proof}
  \begin{enumerate}
  \item[]
  \item
    For finite $n$ this follows immediately from the fact that
    $n$-surjective arrows form the left class of
    a unique factorization system,
    At infinity, we obtain the claim by considering the intersection
    \begin{align}
      \{\text{$\infty$-surjective}\} =\bigcap_{n\in\naturals}\{\text{$n$-surjective}\}.
    \end{align}
  \item
    Follows directly from the explicit description of categorical suspension.
    \qedhere
  \end{enumerate}
\end{proof}

\subsection{Categorical cell complexes}
\begin{definition}
  \label{def:free-maps}
  A \emph{cell filtration} of a map $j\colon C\to D$ of $\infty$-categories
  is a sequence
  \begin{equation}
    \label{eq:presentation-free-map}
    j\colon C=D^{-1}\to D^0\to\dots\to D^d\to \dots \to \colim_{d\to \infty}D^d \simeq D
  \end{equation}
  such that for each $d\geq 0$ there merely exists a cell-attaching pushout square
  \begin{equation}
    \label{eq:cell-attaching-po}
    \cdsquare[po]
    {T_{d}\times \boundary\globe{d}}
    {D^{d-1}}
    {T_d\times \globe{d}}
    {D^d}
    {\phi_d}{T_d\times \iota_d}{}{\Phi_d}
  \end{equation}
  in $\Cat_\infty$,
  where each $T_d$ is a set;
  the additional choice of such pushout squares
  is called a \emph{cell presentation} of $j$.

  A map $j\colon C\to D$ is called \emph{free}
  if there merely exists a cell filtration for it.
  An $\infty$-category $D$ is called
  a \emph{(categorical) cell complex} if $\emptyset\to D$ is free.
\end{definition}

\begin{remark}[Uniqueness of cell filtration]
  \label{rem:upper-lower-cell-index}
  Let $D$ be a categorical cell complex.
  Then just by virtue of being a (right) $\infty$-category,
  $D$ amounts to a sequence
  \begin{equation}
    \label{eq:subscripts-sequence-cell-complex}
    D_0\to D_1\to \dots \to D_d \to \dots
  \end{equation}
  of $d$-categories with $D_d\xrightarrow{\sim} R_d D_{d+1}$.
  Given a cell filtration $D^\bullet$ of $D$,
  there is a unique transformation $D^\bullet \to D_\bullet$
  of $\omega$-indexed sequences compatible with
  $\colim_d D^d \simeq D\simeq \colim_d D_d$.
  It is not hard to show that this transformation is an equivalence
  if and only if $D_\bullet$ is already a cell filtration of $D$
  if and only if (any choice of) the cell-attaching pushout square
  \eqref{eq:cell-attaching-po}
  (in $\Cat_d$)
  of the cell filtration $D^\bullet$
  is preserved by the $(d-1)$-core functor $R_{d-1}$.

  Applying a general criterion about higher categorical cobase change
  that will apear in forthcoming work by Lyne Moser and T.W.,
  one can prove that $R_{d-1}$ does indeed always preserve pushouts of the form
  \eqref{eq:cell-attaching-po},
  so that $D_\bullet$
  is indeed the unique cell filtration of $D$.
  We shall not use this fact in this article.
\end{remark}

\begin{remark}[Freeness and suspension]
  \label{rem:free-susp}
  From the left adjoint description of the categorical suspension $\Susp$
  it follows that a map $f\colon \Susp C\to Z$
  just amounts to the data of two objects $z_0\coloneqq f(0), z_1\coloneqq f(1)\in Z$
  and a map $C\to Z(z_0,z_1)$.

  Hence by unraveling the universal property of cell attaching pushouts,
  it is not hard to see that suspensions of free maps are free.
  More precisely,
  for every cell presentation of $j\colon C\to D$ with generating cells $T_d$ in dimensions $d$,
  we have a cell presentation of $\Susp j\colon \Susp C\to \Susp D$
  with the same generating cells $T_d$ but now in dimensions $d+1$;
  they follow the old attaching rules,
  except for $d=0$ where the old generating $0$-cells now are generating $1$-cells $0\to 1$;
  the new cell presentation has no generating $0$-cells.

  Similarly, if $D$ is a cell complex then $\Susp D$ is again a cell complex:
  For any cell presentation of $D$
  we obtain a cell presentation of $\Susp D$
  with two new generating $0$-cells $0,1$
  and with the old generating cells shifted up in dimension
  (with the same attaching rules
  except for $d=0$, where the old generating $0$-cells now are $1$-cells $0\to 1$).
\end{remark}

\begin{lemma}
  \label{lem:free-surjective-lifting}
  Free maps have the left lifting property with respect
  to $\infty$-surjective maps,
  i.e., in any square
  \begin{equation}
    \cdsquare[lifting]
    {C}{E}
    {D}{F}
    {}{j}{s}{}
  \end{equation}
  where $j$ is free and $s$ is $\infty$-surjective,
  there merely exists a dashed lift.
\end{lemma}

\begin{proof}
  Since we are concerned with proving a mere existence statement,
  we may choose a cell presentation of $j\colon C\to D$ as in
  \Cref{def:free-maps}.
  Then by the universal property of colimit and pushout,
  it suffices (using dependent choice)
  to show inductively that there merely exist lifts of the form
  \begin{equation}
    \cdsquare[lifting]
    {T_d\times \boundary\globe{d}}
    {E}
    {T_d\times\globe{d}}
    {F}
    {}{}{s}{}
  \end{equation}
  The solid square amounts to, for each $t\in T_d$, 
  a parallel pair $(x_t,y_t)$ of $(d-1)$-arrows in $E$
  and a $d$-arrow $f_t\colon sx_t\to sy_t$ in $F$,
  for each of which there merely exist lifts $e_t\colon x_t\to y_t$ of $f_t$
  by the $\infty$-surjectivity of $s$.
  Since $T_d$ is a set, there merely exists a desired lift assembling them all
  (using the axiom of choice).
\end{proof}

\begin{lemma}
  \label{lem:distinct-0-cells}
  Let $D$ be a cell complex with cell presentation
  $(D^\bullet, T_\bullet, \phi_\bullet)$.
  \begin{enumerate}
  \item
    Let $f,g\in T_k$ be two distinct\footnote{
      We say that two objects $x$ and $y$ of an anima $A$
      are \emph{distinct} if there merely exists a map
      $A\to \{0,1\}$ (or, equivalently, $\pi_0A\to \{0,1\}$)
      with $x\mapsto 0$ and $y\mapsto 1$.
      Constructively (without excluded middle),
      this is usually stronger than just asking for $x\not\simeq y$,
      even when $A$ is a set.
      But if the set $\pi_0A$ has decidable equality
      (e.g., $\naturals$, $\{0,\dots,n\}$,
      or any set whose elements admit distinct natural-number labels)
      then the two notions agree even constructively.
    }
    generating $k$-cells.
    Then $f$ and $g$ are still distinct in $\Mor_k D$.
  \item
    Any (pasting of) generating $k$-cells is never invertible in $D$.
  \end{enumerate}
\end{lemma}

\begin{proof}
  Denote by $\Finset\subseteq \An$ the full subcategory spanned
  by the finite sets, i.e., those animae $A$
  for which there merely exists a (necessarily unique) $n\in \naturals$
  and a (definitely not unique, unless $n=0,1$) equivalence $A\simeq \{1,\dots, n\}$.
  Recall the $\infty$-category $\Span_\infty(\An)$
  of animae (over $*$) from \Cref{ex:spans}.
  We can consider the sub-$\infty$-category
  $S\coloneqq\Span_\infty(\Finset)$,
  whose objects are finite sets and where
  at each level of homs we only allow spans (of spans ... of spans)
  whose apex is a finite set.
  Note that the exact same argument as in 
  \Cref{ex:spans} shows that $S\to *$ is still $\infty$-surjective.
  \begin{enumerate}
  \item
    To show that $f$ and $g$ are distinct in $D$ it suffices to
    show the mere existence of a functor $F\colon D\to S$
    such that $F(f)$ and $F(g)$ are distinct in $S$.

    Since $f,g\in T_k$ are distinct, we may choose a map
    $T_k\to \{0,1\}$ with $f\mapsto 0$ and $g\mapsto 1$.
    Then we can define a map $D^{k}\to S$
    via the universal property of the cell-attaching pushout
    by specifying:
    \begin{itemize}
    \item
      $D^{k-1}\to *\xrightarrow{*}\Finset^\simeq$, and
    \item
      $T_k\to \{0,1\}\xrightarrow{(\emptyset,*)}\Finset^\simeq
      \simeq \End^k_S(*)$.
    \end{itemize}
    Then, as per \Cref{lem:free-surjective-lifting},
    there merely exists a solution $F$ to the lifting problem
    \begin{equation}
      \begin{tikzcd}
        {D^k}
        \ar[r]
        \ar[d]
        &
        {S}
        \ar[d]
        \\
        {D}
        \ar[ur,dashed,"F"]
        \ar[r]
        &
        {*}
      \end{tikzcd}
    \end{equation}
    because the right vertical map is $\infty$-surjective
    and the left vertical map is free.
    Since $F(f)\simeq\emptyset$ and $F(g)\simeq *$
    are distinct in $\Finset^\simeq$,
    they are distinct as $k$-dimensional endomorphisms of $*$ in
    $\Span_\infty(\Finset)$
    (or for $k=0$ just as objects);
    and we even have $F(f)\not\simeq F(g)$ as mere $k$-arrows
    because $*\in \Finset^\simeq$ has no automorphisms\footnote{
      An arrow of the form $f\colon x\to x$ in a category $C$
      can be considered either as an endomorphism
      (i.e., as an object of $\End_C(x)$)
      or just as an arrow
      (i.e., as an object of $\Mor C$),
      yielding two notions of identifications that are different in general.
      For example, \emph{every} automorphism is
      (uniquely identified with) an identity \emph{as an arrow}
      but usually not as an endomorphism.
      But if $x\in C$ has no automorphisms
      (i.e., $C^\simeq(x,x)\simeq *$),
      then the forgetful map $\End_C(x)\to \Mor(C)$ is an embedding
      so it does not matter where one considers identifications.
    }.
    Since $F(f)$ and $F(g)$ are distinct\footnote{
      The astute reader will have noticed that we only proved that $F(f)\not\simeq F(g)$,
      and not (the a priori stronger claim) that $F(f)$ and $F(g)$ are distinct.
      But in this case, these notions agree even constructively, because
      $\pi_0\Finset^\simeq$ is just the set of natural numbers which has decidable equality.
      This is the reason we used finite sets
      rather than animae:
      $\pi_0\An^\simeq$ need not have decidable equality.
    } in $\Mor_k S$, so are $f$ and $g$ in $\Mor_k D$.
  \item
    We can perform the same construction as above
    except that we use the constant map
    $T_k\to *\xrightarrow{\emptyset}\An^\simeq\simeq \End^k_S(*)$
    that sends all generating $k$-cells to $\emptyset$.
    Any pasting of $\emptyset\in \End^k_S(*)$ is $\emptyset$ again,
    which is not invertible;
    we conclude with the same reasoning that generating $k$-cells
    or their pastings cannot be invertible in $D$.
    \qedhere
  \end{enumerate}
\end{proof}

\begin{remark}
  In the setting of \Cref{lem:distinct-0-cells},
  if we only have $f\not\simeq g$ in $T_k$,
  then we can still deduce $f\not\simeq g$ in $\Mor_kD$.
  Indeed, the proof is almost exactly the same with the following modifications:
  \begin{itemize}
  \item
    Use $\Span_\infty(\An)$ rather than $\Span_{\infty}(\Finset)$.
  \item
    Instead of  the map $T_k\to \{0,1\}\to \Finset^\simeq$,
    use the map
    \begin{equation}
      T_k\xrightarrow{T_k(-,g)} \Prop^\simeq\hookrightarrow \An^\simeq
    \end{equation}
    which sends $g\mapsto *$ because $T_k$ is a set
    and $f \mapsto\emptyset$ because $f\not\simeq g$.
  \end{itemize}
\end{remark}

\subsection{Coinductive isomorphisms}

Many of the arguments in this subsection originated in a conversation
with F\'elix Loubaton, whom the authors would like to thank.

\begin{definition}
  An $\infty$-category $D$ is called \emph{$L$-trivial}
  if the terminal map \(D\to *\) is $\infty$-surjective.
\end{definition}

The following lemma justifies the name ``$L$-trivial''.

\begin{lemma}
  An $\infty$-category $C$ is $L$-trivial
  if and only if $LC\simeq *$.
\end{lemma}

\begin{proof}
  By \Cref{thm:main}, $L(C)\to L(*)\simeq *$ is an equivalence
  if and only if $C\to *$ is weakly $\infty$-surjective;
  we have to show that this happens if and only if it is actually $\infty$-surjective.

  Since ``if'' is trivial, we only prove ``only if'':
  For this, let $n\in \naturals$ and consider the composite
  \begin{equation}
    C_n\simeq R_n C_{n+1}\to C_{n+1} \to L_nC_{n+1} \to *,
  \end{equation}
  where the first two maps are $n$-surjective
  by \Cref{prop:surjective-core} and \Cref{prop:surjective-localization}, respectively,
  and the last map is $n$-surjective by the assumption that $C\to *$
  is weakly $\infty$-surjective.
  We conclude that each $C_n\to *$ is $n$-surjective,
  hence $C\to *$ is $\infty$-surjective.
\end{proof}

\begin{definition}
  \label{def:walking-coind}
  A \emph{walking coinductive isomorphism}
  is an $L$-trivial cell complex $E\in \Cat_\infty$
  equipped with a free map
  $u\colon \globe{1}\to E$
  whose image we call the \emph{tautological arrow} of $E$
  and also denote by $u\colon 0\to 1$.
\end{definition}

\begin{remark}
  \label{cor:E-0-neq-1}
  Let $(E,u\colon 0\to 1)$ be a walking coinductive isomorphism.
  It follows immediately from \Cref{lem:distinct-0-cells}
  that $0$ and $1$ are distinct objects of $E$
  and that $u$ is not invertible.
\end{remark}

\begin{definition}
  Let \((E,u)\) be a walking coinductive isomorphism.
  Let $C$ be an $\infty$-category
  and $f\colon x\to y$ an $(n+1)$-arrow.
  We say that $f$ is \emph{coinductively invertible}
  or a \emph{coinductive isomorphism} (with respect to $E$)
  if there merely exists a factorization
  \begin{equation}
    \label{eq:fact-coind-equiv}
    f\colon \Susp^n\globe{1}\xrightarrow{\Susp^n u} \Susp^nE\dashrightarrow C
  \end{equation}
  of $f$ through $\Susp^nu$.
\end{definition}

\begin{warning}
  For an $(n+1)$-arrow $f\colon x\to y$,
  the anima of lifts \eqref{eq:fact-coind-equiv}
  need not be a proposition
  because an arrow can be exhibited as a coinductive isomorphism
  in many distinct ways.
  For example, let $C\coloneq E \amalg_{\globe{1}} E$
  be the pushout of two copies of $E$ along $u\colon \globe{1}\to E$,
  i.e., the cell complex where every generating cell of $E$
  (for some chosen cell-presentation) is duplicated,
  except for $0,1$ and $u$.
  Then the tautological arrow $u\colon \globe{1}\to C$
  can be exhibited in two ways as coinductively invertible,
  namely by either of the inclusions $\iota_i\colon E\to C$ (for $i=0,1$).
  These two are actually distinct
  (as objects of the anima of lifts \eqref{eq:fact-coind-equiv}),
  because for any generating $1$-cell $v$ of $E$
  that is distinct from $u$ 
  (which merely exists, because by $L$-triviality of $E$
  there merely exists an arrow $1\to 0$),
  the two arrows $\iota_0(v),\iota_1(v)$ are distinct generating $1$-cells of $C$,
  hence distinct arrows in $C$ by \Cref{lem:distinct-0-cells}.
  See also \Cref{ex:non-uniqueness-coind} for a more explicit example
  of this phenomenon.
  
  Note that this is a purely $\infty$-dimensional phenomenon:
  we will see in \Cref{lem:finit-is-local}
  that if $C$ is finite-dimensional,
  then coinductive invertibility data
  (lifts along $\Susp^n\globe{1}\to \Susp^nE$)
  is the same as invertibility data
  (lifts along $\Susp^n\globe{1}\to \Susp^n*$)
  and the latter is always unique if it exists.
\end{warning}

\begin{remark}
  The name ``walking coinductive isomorphism''
  suggests a certain uniqueness or canonicity of $(E,u)$.
  While it is \emph{not} true that $(E,u)$ itself is unique
  (for example, as a cell complex),
  we will show momentarily that at least the resulting notion of coinductive invertibility
  does not depend on the chosen walking coinductive isomorphism.
 
  In this subsection, we will also not address whether such
  a walking coinductive equivalence even exists.
  But lest the reader worry that we are building castles in the air,
  we refer to \Cref{constr:standard-E} below,
  where we will construct an explicit
  walking coinductive isomorphism $(E,u)$.
  That specific construction will then allow an alternative
  characterization of coinductive isomorphisms
  which also clarifies their name; see \Cref{prop:coinductive-equiv-char}.
\end{remark}

\begin{lemma}
  \label{lem:L-triv-only-coind-equiv}
  Let $E$ be a walking coinductive isomorphism.
  Let $C$ be an $L$-trivial $\infty$-category.
  Then all arrows of $C$ (of all dimensions $n+1$)
  are coinductive isomorphisms (with respect to $E$).
\end{lemma}

\begin{proof}
  If $C$ is $L$-trivial, then by \Cref{lem:free-surjective-lifting}
  there merely exists a solution
  for each lifting problem
  \begin{equation}
    \cdsquareNA[lifting]
    {\Susp^n\globe{1}}
    {C}
    {\Susp^nE}
    {*}
  \end{equation}
  because the left vertical map is free
  (see \Cref{rem:free-susp})
  and the right one is $\infty$-surjective.
\end{proof}

\begin{lemma}
  \label{lem:coind-equiv-indep}
  The notion of coinductive isomorphism does not depend
  on the choice of walking coinductive isomorphism.
\end{lemma}

\begin{proof}
  Let $(E,u)$ and $(E',u')$ be walking coinductive isomorphisms.
  It suffices to show that $u'$ --- %
  which is the universal coinductive isomorphism with respect to $E'$ --- %
  is also a coinductive isomorphism with respect to $E$.
  But this follows from \Cref{lem:L-triv-only-coind-equiv} because
  $E'$ is $L$-trivial.
  (And vice versa.)
\end{proof}

\begin{lemma}
  Let $C$ be an $\infty$-category.
  \begin{enumerate}
  \item
    Every isomorphism (of any dimension $n+1$)
    is a coinductive isomorphism.
  \item
    If $f$ is a coinductive isomorphism of dimension $(n+1)$ in $C$,
    then, viewed as an $(n+2)$-arrow in $\Susp C$
    it is still a coinductive isomorphism.
  \item
    Coinductive isomorphisms (of any dimension $n+1$)
    are closed under composition.
  \end{enumerate}
\end{lemma}

\begin{proof}
  \begin{enumerate}
  \item
    Let $f\colon x\xrightarrow{\sim}y$ be an isomorphism of dimension $(n+1)$.
    This means that $f\colon \Susp^n\globe{1}\to C$
    factors (uniquely) through $\Susp^n\globe{1}\to\Susp^n*$.
    But then the dashed composite
    \begin{equation}
      \begin{tikzcd}
        \Susp^n\globe{1}\ar[r,"f"]
        \ar[d]
        &
        C
        \\
        \Susp^n E
        \ar[ru,dashed]
        \ar[r]
        &
        \Susp^n *\ar[u]
      \end{tikzcd}
    \end{equation}
    proves that $f$ is a coinductive isomorphism.
  \item
    Immediate from the definition.
  \item
    Consider the universal pair of composable coinductive isomorphisms,
    i.e., the pushout
    \begin{equation}
      \cdsquare[po]
      *
      E
      E
      {E'}
      0
      1 {r}
      {l}
    \end{equation}
    It suffices to show that the composite $r(u)\circ l(u)$
    is a coinductive isomorphism in $E'$,
    because then the same will hold for its image under any map $E'\to C$.
    But $E'$ is $L$-trivial as a pushout $L$-trivials;
    hence \emph{all} of its arrows are coinductive isomorphisms
    by \Cref{lem:L-triv-only-coind-equiv}.
    \qedhere
  \end{enumerate}
\end{proof}

\begin{lemma}
  \label{lem:finit-is-local}
  Let $C$ be a $d$-category for some finite $d$.
  Then $C$ is $(\Susp^n E\to\Susp^n*)$-local for each $n\in \naturals$.
\end{lemma}

\begin{proof}
  Let $f\colon \Susp^nE \to C$.
  If $n\geq 1$, then we may replace $C$ with the $\pred[n]{d}$-category
  $C\arlev{n}(s_{n-1},t_{n-1})$,
  where $s_{n-1}$ and $t_{n-1}$
  are the $(n-1)$-dimensional source and target of $f(u)$;
  thus may assume without loss of generality that $n=0$.

  Then by the adjunction $L_d\dashv I_d$,
  we have an equivalence
  \begin{equation}
    \{\overline{f}\colon E\to *\dashrightarrow C\}
    \simeq
    \{L_d\overline{f}\colon L_dE\to *\dashrightarrow C\}
  \end{equation}
  between animae of lifts.
  But the latter anima is trivial,
  because $L_dE\simeq (LE)_d\simeq *$
  by $L$-triviality of $E$.
\end{proof}

\begin{lemma}
  \label{lemma:char-coind-univalent}
  Let $C$ be an $\infty$-category.
  Let $E$ be a walking coinductive isomorphism.
  The following are equivalent:
  \begin{enumerate}
  \item
    \label{it:coind-complete}
    All coinductive isomorphisms of $C$
    (of all dimensions $n+1$) are isomorphisms.
  \item
    \label{it:E-local}
    $C$ is $(\Susp^n E\to\Susp^n*)$-local for each $n\in \naturals$.
  \end{enumerate}
\end{lemma}

\begin{proof}
  \ref{it:E-local} implies \ref{it:coind-complete}
  because for every coinductive isomorphism $f\colon x\to y$
  there merely exists a lift of the map
  $f\colon \Sigma^n\globe{1}\to C$
  to a map 
  $\overline{f}\colon \Susp^nE\to C$,
  which then factors (uniquely) through $\Susp^nE\to \Susp^n*$ by locality,
  thus exhibiting $f\colon x\to y$ as an isomorphism in $C$.

  Conversely, assume that every coinductive isomorphism of $C$
  (of every dimension)
  is an isomorphism.
  We have to show that every map $\overline{f}\colon \Susp^nE\to C$
  factors uniquely through $\Susp^nE\to \Susp^n*$.

  Since all arrows (of all dimensions)
  of $E$ are coinductive isomorphisms
  (\Cref{lem:L-triv-only-coind-equiv})
  the same is true for their image under $\Susp^n E\to C$;
  hence these images are all isomorphisms by the assumption on $C$.
  In other words, the map $\overline{f}\colon \Susp^n E\to C$
  takes values in $C_n=R_nC\subseteq C$;
  and the same is of course true for any map $\Susp^n *\to C$.
  Hence we may replace $C$ by $C_n$ and assume that $C\in \Cat_n$ is finite dimensional;
  then we are done because every finite-dimensional $\infty$-category is
  $(\Susp^nE\to \Susp^n*)$-local
  by 
  \Cref{lem:finit-is-local}.
\end{proof}

\begin{remark}
  It is clear from the proof
  that the backward implication \ref{it:coind-complete}$\impliedby$\ref{it:E-local}
  of \Cref{lemma:char-coind-univalent} is true even
  for each fixed $n\in \naturals$ individually.
  But note that in the forward implication
  \ref{it:coind-complete}$\implies$\ref{it:E-local},
  to show that $C$ is $(\Susp^n E\to \Susp^n *)$-local
  we really needed to use the coinductive completeness of $C$
  not just for arrows of dimension $(n+1)$
  but for all dimensions above it too
  (because we apply it to all arrows in the image of $\Susp^n E\to C$).

  And indeed, for each fixed $n$,
  the condition of being $(\Susp^n E\to \Susp^n*)$-local
  is much stronger than just requiring the $(n+1)$-dimensional
  coinductive isomorphisms to be invertible.
  For example, consider the $\infty$-category
  $E'\coloneqq E[(\Mor_1E)^{-1}]$
  obtained from $E$ by inverting all $1$-arrows.
  Since all $1$-arrows of $E'$ are invertible, it trivially satisfies 
  condition \ref{it:coind-complete} for $n=0$.
  However $E'$ is not $(E\to *)$-local:
  for example, the defining localization map $E\to E'$
  does not factor through $E\to *$,
  because the generating arrows of $E$ of dimension $\geq 2$ are still\footnote{
    The proof of \Cref{lem:distinct-0-cells} goes through verbatim
    to show that (pastings of) generating $k$-cells of a cell complex $D$
    are non-invertible not just in $D$
    but also in the localization $D[(\Mor_{<k}D)^{-1}]$.
  } non-invertible in $E'$.
\end{remark}

\begin{definition}
  \label{def-coind-compl}
  If $C$ satisfies the equivalent conditions of 
  \Cref{lemma:char-coind-univalent},
  we say that $C$ is \emph{coinductively complete}.
  We denote by $L_\coind\colon \Cat_\infty\rightleftarrows \Cat^\coind_\infty$
  the reflective localization
  onto the full subcategory
  of coinductively complete $\infty$-categories.
\end{definition}

\begin{remark}
  By \Cref{lem:coind-equiv-indep}, the subcategory
  $\Cat^\coind_\infty$ does not depend on the choice
  of walking coinductive isomorphism $(E,u)$.

  For now, the reflector $L_\coind\colon \Cat_\infty\to\Cat^\coind_\infty$
  only exists for formal reasons,
  because the localization is generated by the small (even countable) family
  $\{\Susp^n E\to \Susp^n * \mid n\in \naturals\}$
  and $\Cat_\infty$ is presentable (see \Cref{rem:cat-infty-presentable}).
  In \Cref{prop:localization-coind-approx}
  we will give a direct construction of this reflector
  as the localization at the coinductive isomorphisms (as one might have expected).
\end{remark}

\begin{lemma}
  \label{lem:strongly-surject-reflects-coind-equiv}
  Let $F\colon C\to D$ be a $\infty$-surjective map.
  Then $F$ reflects coinductive isomorphisms.
\end{lemma}

\begin{proof}
  This follows from the mere existence of lifts in the square
  \begin{equation}
    \cdsquare[lifting]
    {\Susp^n\globe{1}}
    {C}
    {\Susp^nE}
    {D}
    {f}{}{F}{\overline{g}}
  \end{equation}
  where the left vertical arrow is free and the right is $\infty$-surjective.
  Indeed, if $\overline{g}$ exhibits $g\coloneqq Ff$ as a coinductive isomorphism,
  then the dashed lift exhibits $f$ as one too.
\end{proof}

\begin{lemma}
  \label{lem:infty-surj-strong-coind}
  Let $F\colon C\to D$ be an $\infty$-surjective map between
  coinductively complete $\infty$-categories.
  Then $F$ is an equivalence.
\end{lemma}

\begin{proof}
  By assumption, in $C$ and $D$ the coinductive isomorphisms
  are precisely the isomorphisms.
  Hence $F$ reflects isomorphisms by \cref{lem:strongly-surject-reflects-coind-equiv}.

  Now fix $d\in \naturals$.
  By assumption, $F_{d+1}$ is $(d+1)$-surjective;
  we now show that the same is true for $F_d\simeq R_dF_{d+1}$:
  We automatically have $d$-surjectivity (by \Cref{prop:surjective-core}).
  So it remains to show that for any parallel pair $(x,y)$ of $d$-arrows in $R_dC$
  and every $(d+1)$-arrow $f\colon Fx\to Fy$ in $R_dD$
  (i.e., an invertible $(d+1)$-arrow $Fx\xrightarrow{\sim} Fy$ in $D$)
  there merely exists a $(d+1)$-arrow $f'\colon x\to y$ in $R_dC$
  (i.e., an invertible $(d+1)$-arrow $f'\colon x\xrightarrow{\sim} y$ in $C$) with $Ff'\simeq f$.
  By the $(d+1)$-surjectivity of $F\colon C\to D$
  there exists an (a priori not necessarily invertible) $f'\colon x\to y$ in $C$
  with $Ff'\simeq f$;
  it is automatically an isomorphism because $F$ reflects isomorphisms.

  Since $d\in \naturals$ was arbitrary, it follows from
  \Cref{lem:d+1-surj-equiv}
  that $F$ is an equivalence.
\end{proof}

\begin{proposition}
  \label{prop:localization-coind-approx}
  Let $A$ be an $\infty$-category.
  Let $\ell\colon A\to A'\coloneqq A[\isos{\coind}^{-1}]$
  denote the localization at all the
  coinductive isomorphisms of $A$ (of all dimensions).
  \begin{enumerate}
  \item    \label{it:localization-coind-equiv-surj}
    The map $\ell\colon A\to A'$ is $\infty$-surjective.
  \item
    The map $\ell\colon A\to A'$ exhibits $A'$ as the coinductive completion of $A$
    (this means that $\ell \colon A\to A'\simeq L_\coind A$ is the counit).
  \end{enumerate}
\end{proposition}

\begin{proof}
  \begin{enumerate}
  \item
    It is a general fact that a localization at a collection of arrows
    can be computed by individually trivializing the elements of a set of representatives
    of those arrows (assuming such a set exists).
    Specifically in our case, this means that we have 
    the outer pushout
    \begin{equation}
      \label{eq:po-arrow-E-localization}
      \begin{tikzcd}
        \coprod_{d\in \naturals} T_{d+1}\times \Susp^d\globe{1}
        \ar[r]
        \ar[d]
        \ar[rr,bend left=15]
        &
        \coprod_{d\in \naturals} T_{d+1}\times \Susp^d E 
        \ar[d]
        \ar[r,dashed]
        &
        A
        \ar[d]
        \\
        \coprod_{d\in\naturals} T_{d+1}\times \Susp^d*
        \ar[r,equal]
        &
        \coprod_{d\in\naturals} T_{d+1}\times \Susp^d*
        \ar[r]
        &
        A'
      \end{tikzcd}
    \end{equation}
    where $T_{d+1}\to \Mor_{d+1}A$ 
    is a set of representatives\footnote{
      Given an anima $X$,
      by a \emph{set of representatives} of $X$
      we mean the choice of a section of the counit map $X\to \pi_0 X$
      (the mere existence of such a section follows from the axiom of connected choice).
      We make such a choice for
      $X\coloneqq \cequiv[d+1]{A}\subseteq \Mor_{d+1}A$
      (the subanima of coinductive isomorphism of dimesion $d+1$).
      The top curved map in \eqref{eq:po-arrow-E-localization}
      is then adjoint to the maps
      $T_{d+1}\coloneqq \pi_0 \cequiv[d+1]{A}\dashrightarrow \cequiv[d+1]{A}\hookrightarrow \Mor_{d+1}A$
      (jointly for all $d$).
    }
    of the coinductive isomorphisms of dimension $d+1$.
    For each $f\in T_{d+1}$,
    the corresponding map $f\colon \Susp^d \globe{1}\to A$
    is a coinductive isomorphism by construction,
    hence merely admits a factorization through $E$;
    hence there merely exists the dashed arrow making the above diagram commute.
    (Note that here we use the axiom of choice,
    both to choose the representatives $T_{d+1}$
    and to assemble the dashed arrow from the individual choices for each $f\in T_{d+1}$.)

    We claim that the right square is again a pushout\footnote{
      Of course, it would suffice to show that the left square is a pushout. It is not.
    }.
    This then concludes the proof, because
    $E\to *$ is $\infty$-surjective
    (by definition of walking coinductive isomorphism)
    and $\infty$-surjective maps
    are closed under suspensions, colimits and cobase change
    (\Cref {cor:surj-colims-cobase}),
    yielding that $A\to A'$ is also $\infty$-surjective.
   
    For each $i\in \naturals\cup\{\infty\}$, we set
    \begin{equation}
      F_d^i\coloneqq
      \begin{cases}
        \Susp^d E, \quad \text{if }d<i\\
        \Susp^d\globe{1},\quad \text{if }d\geq i
      \end{cases}
    \end{equation}
    and prove by induction on $i$ that the induced square
    \begin{equation}
      \cdsquareNA
      {\coprod_{d\in\naturals}T_{d+1}\times F_d^i}
      {A}
      {\coprod_{d\in\naturals}T_{d+1}\times \Susp^d*}
      {A'}
    \end{equation}
    is a pushout;
    note that for $i=\infty$ this 
    is then exactly right square of \eqref{eq:po-arrow-E-localization}
    thus concluding the proof.
    The base case $i=0$ is precisely the original outer pushout square
    \eqref{eq:po-arrow-E-localization};
    the limiting step $i\to \infty$ follows because we have the equivalence
    \begin{equation}
      \colim_{i\to \infty}\coprod_d T_{d+1}\times F^i_d \xrightarrow{\sim} \coprod_{d}T_{d+1}\times F^\infty_d
    \end{equation}
    (trivially, because for each fixed $d$ the colimit stabilizes above $i=d+1$).
    It remains to show the successor step $i\to i+1$.
    We form the pushout
    \begin{equation}
      \label{eq:factor-A-A''}
      \begin{tikzcd}
        {\coprod_{d\neq i}T_{d+1}\times F^i_d}
        \ar[r,equal]
        \ar[d]
        &
        {\coprod_{d\neq i}T_{d+1}\times F^{i+1}_d}
        \ar[d]
        \ar[r]
        \iscoCartesian
        &
        A
        \ar[d]
        \\
        {\coprod_{d\neq i}T_{d+1}\times \Susp^d*}
        \ar[r,equal]
        &
        {\coprod_{d\neq i}T_{d+1}\times \Susp^d*}
        \ar[r]
        &
        A''
      \end{tikzcd}
    \end{equation}
    Then we consider the induced commutative diagram
    \begin{equation}
      \label{eq:factor-A''-A'}
      \begin{tikzcd}
        X=T_{i+1}\times \Susp^{i}\globe{1}
        \ar[r]
        \ar[d]
        &
        Y=T_{i+1}\times \Susp^i E 
        \ar[r]
        \ar[d]
        &
        A''
        \ar[d]
        \\
        Z=T_{i+1}\times \Susp^{i}*
        \ar[r,equal]
        &
        Z=T_{i+1}\times \Susp^{i}*
        \ar[r]
        &
        A'
      \end{tikzcd}
    \end{equation}
    where the outer square is pushout
    (because together with the outer pushout square of \eqref{eq:factor-A-A''}
    it forms the pushout square of the induction hypothesis);
    we need to show that the right square is also pushout
    (because it then combines with the right pushout square of \eqref{eq:factor-A-A''}
    to yield the desired pushout square of the induction claim.).

    Recall that all $(j+1)$-arrows of $Y$
    are coinductive isomorphisms, hence so is their image in $A$.
    By construction all these $(j+1)$-arrows then become invertible in $A''$
    for $j>i$
    (because $T_{j+1}\times F_j^{i+1}=T_{j+1}\times \Susp^j\globe{1}\to A$
    is surjective onto the coinductively invertible $(j+1)$-arrows of $A$),
    which means that the map $Y\to A''$
    factors (uniquely) through the $(i+1)$-core $R_{i+1}A''$.

    Finally, let $D\in \Cat_\infty$ be an arbitrary test-object.
    Then the induced comparison map
    $\Map(A',D)\to \Map(A'',D)\times_{\Map(Y,D)}\Map(Z,D)$
    is the composite of the following equivalences
    \begin{align}
      \Map(A',D)&\xrightarrow{\sim} \Map(A'',D)\times_{\Map(X,D)}\Map(Z,D)
      \\
                &\xleftarrow{\sim} \Map(A'',D)\times_{\Map(X,D_{i+1})}\Map(Z,D_i)
      \\
                &\xleftarrow{\sim} \Map(A'',D)\times_{\Map(Y,D_{i+1})}\Map(Z,D_i)
      \\
                &\xrightarrow{\sim} \Map(A'',D)\times_{\Map(Y,D)}\Map(Z,D)
    \end{align}
    justified as follows:
    \begin{itemize}
    \item
      By definition of the outer pushout \eqref{eq:factor-A''-A'}.
    \item
      Because $X$ and $Z$ are a $(i+1)$-category and $i$-category, respectively.
    \item
      Because every $(i+1)$-category is local for $X\to Y$
      (\Cref{lem:finit-is-local}).
    \item
      Because $Y\to A''$ factors through the $(i+1)$-core $R_{i+1}A''$.
    \end{itemize}
    This establishes the desired right pushout square \eqref{eq:factor-A''-A'}.
  \item
    First we show that $A'$ is coinductively complete.
    Let $f'\colon x'\to y'$ be a coinductive isomorphism
    in $A'$ (of some dimension $k+1$),
    witnessed by a map $\overline{f'}\colon\Susp^k E\to A'$.
    Since $A\to A'$ is $\infty$-surjective by part~\ref{it:localization-coind-equiv-surj},
    and $\Susp^kE$ is free,
    we may lift $\overline{f'}$ to
    (a witness $\overline{f}\colon \Susp^k E\to A$ of)
    a coinductive isomorphism $f\colon x\to y$ in $A$.
    But by construction $\ell\colon A\to A'$ is the localization at all
    coinductive isomorphisms,
    which in particular means that $f'=\ell(f)$ is an isomorphism.

    Finally,
    we have to show that $\ell\colon A\to A'$ is a coinductively complete approximation.
    For this, let $C$ be a coinductively complete $\infty$-category.
    Since any map $F\colon A\to C$ preserves coinductive isomorphisms
    and these are just isomorphisms in $C$,
    the map $F$ factors uniquely through $\ell \colon A\to A'$
    by the defining universal property of the localization.
    In other words, this is saying that $C$ is  $\ell$-local,
    which is precisely what we needed to show.
    \qedhere
  \end{enumerate}
\end{proof}

\begin{proposition}
  \label{prop:L-coind-equivalences}
  The following three classes of maps in $\Cat_\infty$ agree:
  \begin{enumerate}[label=(\arabic*),ref=(\arabic*)]
  \item
    \label{it:L-coind-isos}
    The $L_\coind$-equivalences.
  \item
    \label{it:infty-surj-closure}
    The closure of $\infty$-surjective maps under 2-out-of-3.
  \item
    \label{it:infty-surj-up-to-coind}
    The $\infty$-surjections up to coinductive isomorphisms.
  \end{enumerate}
\end{proposition}

\begin{proof}
  Let $F\colon C\to D$ be a map in $\Cat_\infty$.
  Consider the commutative square
  \begin{equation}
    \label{eq:L-coind-loca-square}
    \cdsquare
    {C}
    {L_\coind C}
    {D}
    {L_\coind D}
    {\ell_C}{F}{\overline{F}\coloneqq L_\coind F}{\ell_D}
  \end{equation}
  induced by the adjunction unit for the localization
  $L_\coind\colon \Cat_\infty\to \Cat^\coind_\infty$.
  By the explicit construction of \Cref{prop:localization-coind-approx}
  of the coinductively complete approximation $A\to L_\coind A$,
  we know that the horizontal maps are $\infty$-surjective.

  We start by showing
  ``\ref{it:L-coind-isos} $\subseteq$ 
  \ref{it:infty-surj-closure}$\cap$\ref{it:infty-surj-up-to-coind}'':
  Let $F$ be an $L_\coind$-equivalence.
  Then the square \eqref{eq:L-coind-loca-square} immediately yields
  $F\in$ \ref{it:infty-surj-closure}.
  To show $F\in$ \ref{it:infty-surj-up-to-coind},
  consider parallel $(k-1)$-arrows $x,y$ in $C$
  and a $k$-arrow $f\colon Fx\to Fy$ in $D$.
  By the $\infty$-surjectivity of $\overline{F}\ell_C\simeq \ell_D F$,
  there exists $f'\colon x\to y$ in $C$ with $\ell_D(Ff')\simeq \ell_D (f)$;
  by the $\infty$-surjectivity of $\ell_D$
  there further exists a $(k+1)$-arrow $\alpha\colon Ff'\to f$ in $D$
  lifting this isomorphism,
  which then is coinductively invertible by 
  \Cref{lem:strongly-surject-reflects-coind-equiv}.

  Next, we show 
  ``\ref{it:infty-surj-up-to-coind} $\subseteq$ \ref{it:L-coind-isos}'':
  We assume that $F$ is $\infty$-surjective up to coinductive isomorphisms
  and show that $L_\coind D$ is $\infty$-surjective, hence an equivalence by 
  \Cref{lem:infty-surj-strong-coind}.
  For this, let $(\overline{x},\overline{y})$
  be a parallel pair of $(k-1)$-arrows in $ L_\coind C$
  and $\overline{f}\colon \overline F \overline{x}\to \overline F \overline{y}$
  a $k$-arrow in $L_\coind D$.
  By the $\infty$-surjectivity of $\ell_C$ and $\ell_D$,
  there exist a lifts of these data to 
  $(x,y)$ in $C$ and $f\colon Fx\to Fy$ in $D$,
  where by assumption there there exist $f'\colon x\to y$
  and a coinductive isomorphism $Ff'\to f$ in $D$;
  their images $\overline{f'}\colon \overline{x}\to \overline{y}$
  and $\overline{F}\overline{f'}\xrightarrow{\sim} \overline{f}$
  in $L_\coind C$ and $L_\coind D$, respectively
  then solve the original lifting problem
  (note that the latter is invertible, because $\ell_D$ inverts coinductive isomorphisms).

  In particular, we have thus shown
  ``\ref{it:L-coind-isos} $=$ \ref{it:infty-surj-up-to-coind}'',
  hence the class
  \ref{it:infty-surj-up-to-coind} 
  is closed under $2$-out-of-$3$ (since $L_\coind$-equivalences clearly are).
  Since $\infty$-surjective maps tautologically lie in
  \ref{it:infty-surj-up-to-coind},
  we thus conclude
  ``\ref{it:infty-surj-closure} $\subseteq$ \ref{it:infty-surj-up-to-coind}'',
  finally concluding the circle of inclusions
  ``
  \ref{it:L-coind-isos}
  $\subseteq$ \ref{it:infty-surj-closure}
  $\subseteq$ \ref{it:infty-surj-up-to-coind}
  $\subseteq$ \ref{it:L-coind-isos}
  ''.
\end{proof}

\begin{corollary}
  \label{cor:strong}
  The reflective localization
  $L_\coind\colon \Cat_\infty\to \Cat^\coind_\infty$
  onto the coinductively complete $\infty$-categories
  is a localization at the $\infty$-surjective maps
  and precisely inverts those maps
  that are $\infty$-surjective up to coinductive isomorphisms.
\end{corollary}

\begin{remark}
  The localization $L_\coind$ inverts strictly more
  than just the $\infty$-surjective maps:
  For example, for any choice $E$ of walking coinductive isomorphism,
  the inclusion $\{0\}\hookrightarrow E$
  is an $L_\coind$-equivalence
  (hence $\infty$-surjective up to coinductive isomorphism)
  because both $E\to *$ and the composite $\{0\}\xrightarrow{\sim} *$
  are $\infty$-surjective;
  but clearly it is not $\infty$-surjective because it does not hit $1\in E$
  (because $0$ and $1$ are distinct in $E$ by \Cref{cor:E-0-neq-1}).
\end{remark}

\begin{corollary}
 \label{cor:im-R-coind}
 The reflective localization $L$ of \Cref{thm:main} factors through $L_\coind$.
 \begin{equation}
   \begin{tikzcd}
     \Cat_\infty
     \ar[r, bend left,"L_\coind"]
     \ar[r,hookleftarrow]
     &
     \Cat_\infty^\coind
     \ar[r,hookleftarrow,"R"]
     \ar[r, bend left,"L"]
     &
     \Cat_\infty^L
   \end{tikzcd}
 \end{equation}
 In other words: $R$ takes values in coinductively complete $\infty$-categories.
\end{corollary}

\begin{proof}
  Every $\infty$-surjective arrow is weakly $\infty$-surjective
  (see \Cref{rem:weak-strong-surjective}),
  hence inverted by $L$.
\end{proof}

\begin{remark}
  For maps $C\to *$ to the singleton
  (or suspensions thereof)
  there is no difference between
  ``weakly $\infty$-surjective''
  and
  ``$\infty$-surjective up to coinductive isomorphisms''
  (and ``$\infty$-surjective'', for that matter).

  Therefore, the two reflective localizations $L$ and $L_\coind$
  kill the same objects and
  invert the same maps of the form
  \begin{equation}
    \label{eq:susp-d-X}
    \Susp^n (C\to *)\quad (\text{for } C\in \Cat_\infty \text{ and } n\in \naturals).
  \end{equation}
  Nonetheless, we will see that these two localizations do not agree;
  see \Cref{cor:E-omega-not-coind}.
  This means that the localization $L$ cannot be generated
  solely by maps of the form \eqref{eq:susp-d-X} ---
  unlike $L_\coind$ which for any choice of walking coinductive isomorphism $(E,u)$
  is generated by the set $\{\Susp^n E\to \Susp^n *\mid n\in \naturals\}$
  (see \Cref{lemma:char-coind-univalent}).
\end{remark}

\subsection{Coinductive characterization}
\label{sec:coinductive-char}

The next proposition explains the name ``coinductive isomorphism''
by actually characterizing them via coinduction.
But first, we give an explicit construction of
a walking coinductive isomorphism with a particularly nice description.

\begin{construction}
  \label{constr:standard-E}
  We define a categorical cell complex $E$ freely generated by
  two objects $0,1$
  and a $d$-arrow $u(\sigma)$ for each sequence
  $\sigma\in\{-,0,+\}^d$
  with $\sigma(i)\in \{-,+\}$ for all $i<d$.
  They are attached as follows:

  For $d=1$:
  \begin{equation}
    u\coloneqq u(0)\colon 0\to 1,
    \quad
    u(-),u(+)\colon 1\to 0.
  \end{equation}
  For $d\geq 2$ and each $\sigma\in \{-,+\}^{d-2}$
  and each $\epsilon\in\{-,+\}$:
  \begin{align}
    u(\sigma,+,0)&\colon \id\to u(\sigma,+)u(\sigma,0)
    \\
    u(\sigma,+,\epsilon)&\colon u(\sigma,+)u(\sigma,0)\to \id
    \\
    u(\sigma,-,0)&\colon u(\sigma,0)u(\sigma,-) \to \id
    \\
    u(\sigma,-,\epsilon)&\colon\id \to u(\sigma,0)u(\sigma,-).
  \end{align}
  Note that these are well-typed,
  because $u(\sigma,0)$
  is anti-parallel to $u(\sigma,-)$ and $u(\sigma,+)$,
  hence the indicated compositions are indeed parallel to a suitable identity.
\end{construction}

\begin{remark}
  Consider the $2$-dimensional complex
  $I$ generated by
  \begin{equation}
    f\colon 0\to 1,\quad g_-,g_+\colon 1\to 0,
    \quad h_+\colon \id_0 \to g_+f,\quad h_-\colon fg_-\to \id_1.
  \end{equation}
  Then comparing the generating cells yields the pushout square
  \begin{equation}
    \label{eq:pushout-E-coinduction}
    \cdsquare[po]
    {{\Susp\globe{1}}\amalg{\Susp\globe{1}}}
    {I}
    {{\Susp E}\amalg {\Susp E}}
    {E}
    {{(h_+,h_-)}}{\Susp u\amalg \Susp u}{}{}
  \end{equation}
  Here the bottom horizontal map is determined in the first component by
  \begin{equation}
    0_E\mapsto \id_{0},
    \quad
    1_E\mapsto u(+)u(0),
    \quad
    u(\sigma)\mapsto u(+,\sigma)
  \end{equation}
  and in the second component by
  \begin{equation}
    0_E\mapsto u(0)u(-),
    \quad
    1_E\mapsto \id_{1},
    \quad
    u(\sigma)\mapsto u(-,\sigma);
  \end{equation}
  the right vertical map sends
  \begin{equation}
    f\mapsto u=u(0),
    \quad
    g_\epsilon\mapsto u(\epsilon),
    \quad
    h_\epsilon\mapsto u(\epsilon,0).
  \end{equation}
  (We use the subscript ``$E$'' to emphasize that we are talking about
  the objects $0,1\in E$, viewed as $1$-arrows in $\Susp E$,
  and not about the objects $0,1\in \Susp E$ with the same name.)
\end{remark}

\begin{remark}
  \label{rem:standard-E-filtration}
  Observe that $E$ is the colimit of the sequence
  \begin{equation}
    E^{(1)}=\{0\to 1\}
    \to I \simeq E^{(2)}
    \to \dots \to E^{(d)}\to \dots,
  \end{equation}
  where $E^{d-1}\subset E^{(d)}\subset E^{d}$
  is the subcomplex obtained by only attaching
  the $d$-arrows of the form $u(\sigma,0)$ for $\sigma\in\{-,+\}^{d-1}$
  (and not $u(\sigma, \epsilon)$ for $\epsilon\in\{-,+\}$).
  Observe that by construction we have pushout squares
  \begin{equation}
    \label{eq:po-coinductive-E-step}
    \cdsquare[po]
    {\coprod\limits_{\sigma\in\{-,+\}^{d-1}}\Susp^{d-1}\globe{1} }
    {E^{(d)}}
    {\coprod\limits_{\sigma\in\{-,+\}^{d-1}}\Susp^{d-1} I}
    {E^{(d+1)}}
    {}{}{}{}
  \end{equation}
  where for each $\sigma\in \{-,+\}^{d-1}$,
  the horizontal maps are determined by
  \begin{equation}
    f\mapsto u(\sigma,0),
    \quad
    g_\epsilon \mapsto u(\sigma,\epsilon),
    \quad
    h_\epsilon\mapsto u(\sigma,\epsilon,0).
  \end{equation}
  Also observe that the pushout square
  \eqref{eq:pushout-E-coinduction}
  restricts to the pushout squares
  \begin{equation}
    \label{eq:pushout-E-coinduction-d}
    \cdsquare[po]
    {{\Susp\globe{1}}\amalg{\Susp\globe{1}}}
    {I}
    {{\Susp E^{(d-1)}}\amalg {\Susp E^{(d-1)}}}
    {E^{(d)}}
    {{(h_+,h_-)}}{\Susp u\amalg \Susp u}{}{}
  \end{equation}
  for each $d\geq 2$.
  In particular, this construction is analogous to
  the construction performed in the strict case in
  \cite[\textsection1.5]{ORsurvey} (and \cite{HLOR}).
\end{remark}

\begin{lemma}
  \label{lem:standard-E}
  For each $d\geq 1$, we have $L_{d-1}E^{(d)}\simeq *$.
\end{lemma}

\begin{proof}
  We prove the claim by induction on $d$.
  \begin{itemize}
  \item
    $d=1$:
    Clearly $E^{(1)}=\globe{1}$ is $L_0$-trivial.
  \item
    $d=2$:
    We have that $E^{(2)}=I$ is $L_1$-trivial,
    because after inverting $h_+,h_-$,
    it is the standard presentation of the walking isomorphism,
    which is trival by univalence.
  \item
    $d>2$:
    Assuming $L_{d-2}E^{(d-1)}\simeq *$,
    we have $L_{d-1}\Susp E^{(d-1)}\xrightarrow{\sim} \Susp *$ by suspending once.
    Thus applying $L_{d-1}$ to the pushout square
    \eqref{eq:pushout-E-coinduction-d}
    yields the pushout square
    \begin{equation}
      \cdsquare[po]
      {{\Susp\globe{1}}\amalg{\Susp\globe{1}}}
      {I}
      {{\Susp *}\amalg{\Susp *}}
      {L_{d-1} E^{(d)}}
      {{(h_+,h_-)}}{}{u}{}
    \end{equation}
    which exhibits $L_{d-1} E^{(d)}$ as the localization of $I$
    at the two generating $2$-arrows $h_+,h_-$.
    It follows that we have $L_{d-1}E^{(d)}\simeq L_1 I\simeq *$,
    as desired.
    \qedhere
  \end{itemize}
\end{proof}

\begin{corollary}
  The categorical cell complex $E$ is $L$-trivial,
  hence we have a walking coinductive isomorphism $(E,u)$.
\end{corollary}

\begin{proof}
  Using the presentation $E\simeq \colim_{d:\omega} E^{(d)}$, we have
  \begin{equation}
    {L_nE}
    \simeq
    {L_n(\colim_{d>n}E^{(d)})}
    \simeq
    {\colim_{d>n} L_n(L_{d-1}E^{(d)})}\simeq {\colim_{d}L_n*}\simeq *
  \end{equation}
  for each $n\in \naturals$.
\end{proof}

\begin{corollary}
  \label{cor:E-not-coind}
  The $\infty$-category $E$ is not coinductively complete,
  hence does not lie in the full subcategory $\Cat^\coind_\infty\subsetneq \Cat_\infty$.
\end{corollary}

\begin{proof}
  The tautological arrow $u\colon 0\to 1$ of $E$
  is coinductively invertible but not invertible (see \Cref{cor:E-0-neq-1}).
\end{proof}

Using the explicit construction of \Cref{constr:standard-E}
it is now easy to see that coinductive invertibility data need not be unique
with respect to a fixed walking coinductive isomorphism.

\begin{example}[Non-uniqueness of coinductive invertibility data]
  \label{ex:non-uniqueness-coind}
  We have a non-trivial automorphism $\tau^1\colon E^1\to E^1$
  that is the identity on $E^0=\{0,1\}$
  and is specified on generating $1$-arrows via
  \begin{equation}
    u\mapsto u, \quad u(-) \mapsto u(+), \quad u(+)\mapsto u(-).
  \end{equation}
  Since $E^1\to E$ is free and $E\to *$ is $\infty$-surjective,
  there merely exists an extension
  \begin{equation}
    \begin{tikzcd}
      E^1\ar[r,"\tau^1"]
      \ar[d]
      & E
      \ar[d]
      \\
      E
      \ar[ur,dashed,"\tau"]
      \ar[r]
      &
      *
    \end{tikzcd}
  \end{equation}
  Since $\tau(u)=u=\id_E(u)$,
  we see that both $\tau\colon E\to E$ and $\id_E\colon E\to E$
  are witnesses that the tautological $1$-arrow $u$ of $E$
  is coinductively invertible (with respect to $E$).
  They are in fact \emph{distinct} witnesses
  (i.e., distinct in the anima of lifts
  $\{u\colon \globe{1}\to E\dashrightarrow E\}$)
  because
  $\tau(u(+))=u(-)$ and $\id_E(u(+))=u(+)$ are distinct $1$-arrows of $E$
  (by \Cref{lem:distinct-0-cells}).
\end{example}

\begin{proposition}[Coinductive characterization of coinductive isomorphisms]
  \label{prop:coinductive-equiv-char}
  Let $C$ be an $\infty$-category and 
  denote by $\cequiv{C}\subseteq \Mor_{\geq 1} C$
  the subanima of coinductive isomorphisms of $C$.
  Then $\cequiv{C}$ is uniquely characterized as the \emph{maximal}
  solution of the recursive formula \ref{recursion-formula-E}.
  More precisely,
  \begin{enumerate}
  \item
    the subanima $\cequiv{C}\subseteq \Mor_{\geq 1}C$
    satisies the recursive formula \ref{recursion-formula-E},
    and
  \item
    if $K\subseteq\Mor_{\geq 1}C$ is a subanima satisfying the forward (``only if'')
    direction of \ref{recursion-formula-E},
    then $K\subseteq \cequiv{C}$.
  \end{enumerate}
\end{proposition}

\begin{proof}
  \begin{enumerate}
  \item
    Working with the walking coinductive isomorphism $(E,u)$
    of \Cref{constr:standard-E},
    this is apparent from the pushout \eqref{eq:pushout-E-coinduction},
    which exactly expresses an equivalence between the data exhibiting
    \begin{itemize}
    \item
      a map $E\to C$ exhibiting $f$ as a coinductive isomorphism and
    \item
      arrows $g_-,g_+,h_-,h_+$ as above plus a map $\Susp E\amalg \Susp E\to C$
      exhibiting $h_-,h_+$ as coinductive isomorphisms;
    \end{itemize}
    a fortiori, we get a logical equivalence between mere existence statements.
    (Note that without loss of generality we may assume $d=1$
    by either passing to suitable hom-categories of $C$
    or equivalently by suspending the pushout \eqref{eq:pushout-E-coinduction}).
  \item
    Let $K\subseteq \Mor_{\geq1} C$
    be a subanima satisfying the forward (``only if'')
    direction of \ref{recursion-formula-E},
    and $f\colon x\to y$ be an arrow in $K$
    which we may assume to be of dimension $1$
    (otherwise pass to an appropriate hom-category
    and note that this preserves the condition \ref{recursion-formula-E}).
    We need to show that $f$ is a coinductive isomorphism.
    We use the filtration $E^{(\bullet)}$ of \Cref{rem:standard-E-filtration}.
    By induction on $d$, we show that there exist successive extensions
    $f^{(d)}\colon E^{(d)}\to C$
    of $f^{(1)}= f\colon \globe{1}\to C$,
    such that $f(u(\sigma,0))\in K$ for each $\sigma\in\{-,+\}^{d-1}$.
    This follows immediately from the pushout \eqref{eq:po-coinductive-E-step}:
    indeed, if the image of $u(\sigma,0)$ is already chosen in $K$ by induction,
    the forward direction of \ref{recursion-formula-E}
    guarantees that there exist choices for
    the images of $u(\sigma, \epsilon)$ and $u(\sigma,\epsilon,0)$
    (for $\epsilon\in\{-,+\}$),
    with the latter again lying in $K$.
    \qedhere
  \end{enumerate}
\end{proof}

\begin{corollary}
  \label{cor:full-dualizable-coind-collapse}
  Let $C$ be a fully dualizable $\infty$-category\footnote{
    It is more common to speak of
    fully dualizable \emph{monoidal} $\infty$-categories,
    where one also requires objects to have duals
    with respect to the monoidal structure.
  }, i.e.,
  assume that every arrow of every dimension $\geq 1$ has both adjoints.
  Then the unit maps
  \begin{equation}
    L^\coind C \xrightarrow{\sim} L C \xrightarrow{\sim}{L_0C}
  \end{equation}
  are equivalences.
\end{corollary}

\begin{proof}
  Since every arrow of $C$ admits both a left and a right adjoint,
  the anima $\Mor_{\geq 1} C$ of \emph{all} arrows
  satisfies the forward (``only if'') direction of \ref{recursion-formula-E}.
  Hence by \Cref{prop:coinductive-equiv-char}
  we have $\Mor_{\geq 1}C \subseteq \cequiv{C}$, i.e.,
  every arrow of $C$ is coinductively invertible.
  It follows from the explicit construction of the coinductive completion from
  \Cref{prop:localization-coind-approx},
  that $L_\coind C$ is the anima obtained from $C$ by inverting all arrows
  of all dimensions, i.e., $L_0C$.
  Since finite-dimensional $\infty$-categories already lie in the image of $R$,
  further applying $L$ does not change $L_0C$ anymore, concluding the proof.
\end{proof}

\begin{example}
  \label{ex:cobordisms}
  The cobordism $\infty$-category
  is fully dualizable\footnote{
    According to the cobordism hypothesis, it is in fact the \emph{free}
    fully dualizable symmetric monoidal $\infty$-category generated by one object.},
  hence becomes a groupoid
  under the localization $L^\coind$,
  hence a fortiori under the localization $L$.
\end{example}

\subsection{\texorpdfstring{$\omega$}{Omega}-complete \texorpdfstring{$\infty$}{infinity}-categories}

One might hope to characterize the image of the full embedding
\begin{equation}
  R\colon \Cat_\infty^L \hookrightarrow \Cat_\infty.
\end{equation}

Since this embedding factors through $\Cat_\infty^\coind$
(by \Cref{cor:im-R-coind}),
each $\infty$-category $C$ in the image needs
to be at least coinductively complete.
The following example due to Henry--Loubaton shows that there are
coinductively complete categories $C$ that do not lie in the image of $R$.
In other words, the localization
\begin{equation}
  L\colon \Cat_\infty^\coind \rightleftarrows \Cat_\infty^L: R
\end{equation}
is not an equivalence.

\begin{example} [\cite{henry-loubaton}, Construction~4.33]
  \label{ex:HL}
  Consider the pushout
  \begin{equation}
    \cdsquareNA[po]
    {\coprod_{d\geq 1}\globe{1}}
    {\coprod_{d\geq 1}E^{(d)}}
    {\globe{1}}
    {E^{(\omega)}}
  \end{equation}
  where $E^{(\bullet)}$
  is the filtration from \Cref{rem:standard-E-filtration}.

  In other words, $E^{(\omega)}$ is the categorical cell complex generated by
  \begin{itemize}
  \item
    Two objects $0, 1$ and an $1$-arrow $u\colon 0\to 1$.
  \item
    For each $i\geq d > 1$ and each $\sigma\colon \{-,+\}^{d-1}$
    two arrows
    \begin{equation}
      u_i(\sigma) \quad\text{and}\quad u_i(\sigma,0)
    \end{equation}
    of dimensions $d-1$ and $d$, respectively.
    They are attached to (composites of) lower-dimensional cells $u_i(\sigma')$
    (with the same index $i$)
    with the same rules as in \Cref{constr:standard-E}
    and the convention that $u=u_1(0)=u_2(0)=\dots$.
  \end{itemize}
  
  We will see below that $E^{(\omega)}$ does not lie in the image of $R$.
  Nevertheless, $E^{(\omega)}$ is coinductively complete.
  Here is a sketch of the argument by Henry and Loubaton;
  see \cite[Corollary~4.35]{henry-loubaton} for more details: 

  To show coinductive completeness,
  it suffices to show that all coinductive isomorphisms
  of dimension $n\geq 2$ are invertible.
  So consider a functor $f\colon \Susp^{n-1} E\to E^{(\omega)}$
  classifying a coinductive isomorphism $f(u)$
  of dimension $n\geq 2$.
  If it is not invertible, $f(u)$ can be written as a pasting
  involving at least one generating $n$-cell of $E^{(\omega)}$,
  say $g=u_i(\sigma)$.
  Then for each generating $j$-cell $u'$ of $\Susp^{n-1}E$
  either $f(u')$ is invertible
  or $f(u')$ can be written as a pasting involving at least
  one generating $j$-cell $g'$ that is attached to (a composite of)
  cells involving $g$;
  in other words, $g'=u_i(\sigma')$.
  Since such a $u_i(\sigma')$ is necessarily of dimension $j\leq i+1$,
  this means that $f\colon \Susp^{n-1}E\to E^{(\omega)}$
  inverts all cells of dimension $j>i+1$,
  i.e., factors through $L_{i+1}(\Susp^{n-1}E)$, which is trivial
  (because $E$ is $L$-trivial).
  We conclude that $f(u)$ is indeed an isomorphism.
\end{example}

\Cref{ex:HL} motivates the following definition,
weakening the notion of coinductive equivalence:

\begin{definition}
  \label{def:alpha-invertible}
  Let $C$ be an $\infty$-category.
  Let $\alpha$ be an ordinal
  and $f\colon x\to y$ a $k$-arrow in $C$.
  \begin{itemize}
  \item
    We say that $f$ is \emph{$\alpha$-invertible} (or an \emph{$\alpha$-isomorphism}),
    if for each $\alpha'< \alpha$
    there merely exist $k$-arrows $g_-,g_+\colon y\to x$
    and $\alpha'$-invertible $(k+1)$-arrows
    \begin{equation}
      h_+\colon \id_x\to g_+f, \quad h_-\colon fg_-\to \id_y.
    \end{equation}
  \item
    We say that $C$ is \emph{$\alpha$-complete}
    if every $\alpha$-isomorphism (of any dimension) is an isomorphism.
    We denote the full subcategory of $\alpha$-complete $\infty$-categories by
    $\Cat^\alpha_\infty\subseteq \Cat_\infty$.
  \end{itemize}
\end{definition}

\begin{remark}
  \Cref{def:alpha-invertible} is sensible because ordinals are well-founded.
  Observe that in the case $\alpha=0$ the condition is vacuous,
  so that \emph{every} arrow is $0$-invertible.
  For increasing $\alpha$, the condition of $\alpha$-invertibility
  becomes increasingly stronger;
  and coinductive invertibility is stronger than all of them.

  Equivalently, one could also replace the strong induction by separate zero, successor and limit cases:
  \begin{itemize}
  \item
    Every arrow is $0$-invertible.
  \item
    A $k$-arrow is $(\alpha+1)$-invertible if it
    merely admits both a left and a right inverse
    up to $\alpha$-invertible $(k+1)$-arrows.
  \item
    For limit ordinals $\lambda$, an arrow is $\lambda$-invertible
    if it is $\alpha$-invertible for all $\alpha<\lambda$.
  \end{itemize}
\end{remark}

\begin{remark}
  With a little sleight of hand one can rewrite $\alpha'<\alpha$
  as $\alpha'+1\leq \alpha$
  and then think of $\infty$ as a ``quantity''
  that is bigger than any ordinal
  and satisfies $\infty+1=\infty$ (unlike any ordinal).
  Doing this yields the following self-referential definition of $\infty$-invertibility:
  $f$ is $\infty$-invertible if and only if
  there exist $g_-,g_+$ and $\infty$-invertible $h_-,h_+$.
  This is precisely the recursion formula \ref{recursion-formula-E}.
  \Cref{prop:coinductive-equiv-char}
  says that resolving this self-referentiality coinductively
  then precisely yields the notion of coinductive isomorphism.
  From this perspective, it is reasonable to call coinductive isomorphisms
  ``$\infty$-isomorphisms'' and coinductive completeness ``$\infty$-completeness''.
\end{remark}

\begin{remark}
  By definition, a $k$-arrow $f$ in $C$ is $n$-invertible for some finite $n$
  if and only if there merely exists an extension of its classifying map
  $f\colon \Susp^{k-1}\globe{1}\to C$
  along $\Susp^{k-1}\globe{1}\hookrightarrow\Susp^{k-1}E^{(n+1)}$;
  it is $\omega$-invertible if there merely exists an extension
  along $\Susp^{k-1}\globe{1}\hookrightarrow\Susp^{k-1}E^{(\omega)}$.
\end{remark}

\begin{remark}
  An $\infty$-category is $0$-complete if and only if it is a $0$-category
  (i.e., a groupoid).
  For general finite $d\in \naturals$
  it is still true that every $d$-category is $d$-complete
  but the converse need not hold anymore.
  For example, the infinite globe
  $\globe{\infty}\coloneqq \colim_{d:\omega}\boundary\globe{d}$
  (which in each dimension $d$ has two non-identity $d$-arrows
  $0_d,1_d\colon 0_{d-1}\to 1_{d-1}$)
  is a $1$-complete $\infty$-category that is not finite-dimensional.
\end{remark}

\begin{lemma}
  \label{lem:coind-inv-via-L_n}
  Let $C$ be an $\infty$-category and $f\colon x\to y$ a $k$-arrow.
  Let $n\geq 0$.
  \begin{enumerate}
  \item
    If $f$ is $n$-invertible
    then $f$ becomes invertible in $L_{k+n-1}C_{k+n}$.
  \item
    If $f$ becomes invertible in $L_{k+n}C_{k+n+1}$
    then $f$ is $n$-invertible.
  \end{enumerate}
\end{lemma}

\begin{proof}
  Let $(s,t)$ be the $(k-2)$-dimensional source and target of $f$.
  After replacing $C$ with $C\arlev{k-1}(s,t)$, we may assume that $k=1$.
  \begin{enumerate}
  \item
    Saying that $f$ is $n$-invertible means that
    there merely exists a factorization
    \begin{equation}
      f\colon \globe{1}\to E^{(n+1)}\to C_{n+1}\subseteq C
    \end{equation}
    Since $L_{n}E^{(n+1)}\simeq *$ by \Cref{lem:standard-E},
    it follows that $f$ becomes invertible in $L_{n}C_{n+1}$.
  \item
    Saying that $f$ becomes invertible in $L_{n+1}C_{n+2}$
    exactly means that we have the following solid commutative square:
    \begin{equation}
      \begin{tikzcd}
        \globe{1}
        \ar[rr,"f"]
        \ar[d,"u"]
        &&
        C_{n+2}
        \ar[d]
        \\
        E^{(n+1)}
        \ar[urr,dashed]
        \ar[r]
        &*\ar[r]& L_{n+1}C_{n+2}
      \end{tikzcd}
    \end{equation}
    Recall that $C_{n+2}\to L_{n+1}C_{n+2}$ is $(n+1)$-surjective
    (\Cref{prop:surjective-localization}).
    Hence there merely exists a dashed lift as indicated,
    because $u\colon \globe{1}\to E^{(n+1)}$ is obtained by freely attaching cells of dimension
    at most $n+1$;
    this means that $f$ is $n$-invertible.
    \qedhere
  \end{enumerate}
\end{proof}

\begin{corollary}
  Let $C$ be an $\infty$-category.
  The $\omega$-isomorphisms in $C$
  are precisely those $k$-arrows (for some $k\geq 1$)
  that become invertible in $L_nC_{n+1}$ for each $n\geq k-1$.
\end{corollary}

We can now give a new characterization of $L$-equivalences:
they are those maps that for all $n\in \naturals$
are $\infty$-surjective up to $n$-invertible arrows.

\begin{proposition}
  \label{lem:L-iso-coind-inv}
  Let $F\colon C\to D$ be a map of $\infty$-categories.
  The following are equivalent:
  \begin{itemize}
  \item
    $F$ is an $L$-equivalence, i.e., weakly $\infty$-surjective.
  \item
    For each $n\in \naturals$,
    each parallel pair $(x,y)$ of $(k-1)$-arrows in $C$,
    and each $k$-arrow $f\colon Fx\to Fy$,
    there merely exists a $k$-arrow $f_n\colon x\to y$
    and an $n$-invertible $(k+1)$-arrow
    $\eta_n\colon Ff_n\to f$.
  \end{itemize}
\end{proposition}

\begin{remark}
  We emphasize that both $f_n$ and $\eta_n$ are allowed to depend on $n$.
  So we are \emph{not} saying that $F$ is $\infty$-surjective
  up to $\omega$-isomorphisms,
  which would require the existence of a single $f'\colon x\to y$
  and a single $\omega$-isomorphism $\eta\colon Ff'\to f$.

  See \Cref{lem:V-W-L-not-omega}
  for a tautological construction of an $L$-equivalence
  that is not surjective up to $\omega$-isomorphisms.
\end{remark}

\begin{proof}
  Recall that the $L$-equivalences are characterized as those $F$
  for which each $L_nF_{n+1}$ is $n$-surjective.

  For the forward direction assume that $F$ is an $L$-equivalence
  and fix $n,k\in \naturals$.
  Consider the induced commutative diagram
  \begin{equation}
    \begin{tikzcd}[column sep=large]
      C_k\ar[r,"\in S_k"]\ar[d,"F_k"]
      &C_{k+1}\ar[d,"F_{k+1}"]\ar[r,"\in S_{k+1}"]
      &C_{k+n+2}\ar[d,"F_{k+n+2}"]\ar[r,"\in S_{k+n+1}"]
      &L_{k+n+1}C_{k+n+2}\ar[d,"L_{k+n+1}F_{k+n+2}\in S_{k+n+1}"]
      \\
      D_k\ar[r,"\in S_k"]&D_{k+1}\ar[r,"\in S_{k+1}"]&D_{k+n+2}\ar[r,"\in S_{k+n+1}"]&L_{k+n+1}D_{k+n+2}
    \end{tikzcd}
  \end{equation}
  where the arrows decorated with ``${\in}S_i$''
  are surjective of the indicated level $i$.
  Let $(x,y)$ be a parallel pair of $(k-1)$-arrows in $C$
  and $f\colon Fx\to Fy$ a $k$-arrow in $D_k$.
  An easy diagram chase yields
  the existence of a $k$-arrow $f_n\colon x\to y$ in $C_k$
  and an identification
  $\overline{\eta_n}\colon [Ff_n]\simeq [f]\colon [x]\to [y]$ in $L_{(k+1)+n}D_{(k+1)+n+1}$
  which admits a lifts to a $(k+1)$-arrow $\eta_n\colon F f_n\to f$ in $D_{k+1}$.
  It follows from \Cref{lem:coind-inv-via-L_n} that $\eta_n$
  is $n$-invertible as desired.

  For the backward direction, let $n\geq 0$.
  We need to show that $L_nF_{n+1}$ is $n$-surjective.
  For this let $0\leq k\leq n$, and consider a parallel pair
  of $(k-1)$-arrows of $L_nC_{n+1}$
  as well as a $k$-arrow between them in $L_nD_{n+1}$;
  by $n$-surjectivity of $C_n\to L_nC_{n+1}$ and $D_n\to L_nD_{n+1}$
  we may assume that they are represented by a parallel pair
  $(x,y)$ in $C_n$ and a $k$-arrow $f\colon Fx\to Fy$ in $D_{n}$, respectively.
  Then by assumption there exists for each $m\geq 0$ a $k$-arrow $f_{m}\colon x\to y$
  and a $m$-invertible $(k+1)$-arrow $\eta_m\colon Ff_m\to f$.
  Setting $m\coloneqq n-k\geq 0$ guarantees that $(k+1)+m-1=n$
  so that $\eta_m$ becomes invertible in $L_nD_{n+1}$ by
  \Cref{lem:coind-inv-via-L_n},
  thus yielding the desired identification $[Ff_m]\simeq [f]$ in $L_nD_{n+1}$.
\end{proof}

\begin{proposition}
  \label{prop:coind->weak-coind}
  Assume that $C$ lies in the image of $R$.
  Then $C$ is $\omega$-complete.
\end{proposition}

\begin{proof}
  Let $f\colon x\to y$ be a $\omega$-invertible $d$-arrow of $C$
  and choose a corresponding map $F\colon \Susp^{d-1}E^{(\omega)}\to C$ with $F(u)=f$.
  Consider the pushout square
  \begin{equation}
    \cdsquare[po]
    {\Susp^{d-1} E^{(\omega)}} {C}
    {\Susp^{d-1}E^{(\omega)}[u^{-1}]} {C[f^{-1}]}
    {F}
    {\Susp^{d-1}\ell}{\ell'}
    {F[u^{-1}]}
  \end{equation}
  where the vertical maps are localizations at $u$ and $f$,
  respectively.
  We first claim that the localization map
  $\ell\colon E^{(\omega)}\to E^{(\omega)}[u^{-1}]$ is an $L$-equivalence,
  hence also its categorical suspension $\Susp^{d-1}\ell$ and the cobase change $\ell'$:
  Indeed, for each $d\geq 1$, we have $L_{d-1}E^{(d)}\simeq *$
  (\Cref{lem:standard-E});
  hence $u$ is invertible in $L_{d-1}E^{(d)}$ and hence in $L_{d-1}E^{(\omega)}$;
  it follows that the localization
  $L_{d-1}\ell \colon L_{d-1}E^{(\omega)}\to L_{d-1}E^{(\omega)}[u^{-1}]$
  is an isomorphism as claimed
  (note that we implicitly use that $L$ commutes with localizations).
  We conclude that the composite
  $C\xrightarrow{\ell'} C[f^{-1}]\to RL(C[f^{-1}])$
  is an $L$-equivalence
  between objects in the image of $R$, hence an equivalence;
  in other words $\ell'$ has a retraction.
  Since $\ell'$ is an epimorphisms (because it is a localization)
  with a retraction, it is an equivalence.
  But this means that $f$ was already invertible in $C$,
  which is exactly what we needed to show.
\end{proof}

\begin{corollary}
  \label{cor:E-omega-not-coind}
  The coinductively complete $\infty$-category $E^{(\omega)}$
  is not $\omega$-complete,
  hence does not lie in the image of $R$.
\end{corollary}

\begin{proof}
  By construction, the tautological arrow $u\colon 0\to 1$ in $E^{(\omega)}$
  is an $\omega$-isomorphism.
  It is not an isomorphism because its image under the tautological map
  $E^{(\omega)}\to E$ (defined by $u_i(\sigma)\mapsto u(\sigma)$)
  is the tautological arrow $u\colon 0\to 1$ of $E$
  which is not an isomorphism by \Cref{cor:E-0-neq-1}.
\end{proof}

At this point, one might wonder whether the necessary criterion of
\Cref{prop:coind->weak-coind}
is also sufficient,
i.e., whether the induced inclusion
$R\colon \Cat^L_\infty\hookrightarrow \Cat_\infty^\omega$
is an equivalence.
The following example shows that this is not the case:

\begin{lemma}
  \label{lem:V-W-L-not-omega}
  We have an inclusion $W\hookrightarrow V$
  of $\infty$-categories with the following properties:
  \begin{enumerate}
  \item
    Both $V$ and $W$ are $\omega$-complete.
  \item
    The inclusion $W\hookrightarrow V$ is an $L$-equivalence.
  \item
    The inclusion $W\hookrightarrow V$ has no retraction,
    hence is not an equivalence.
  \item
    The category $W$ does not lie in the image of $R\colon \Cat^L_\infty\hookrightarrow\Cat^\omega_\infty$.
  \end{enumerate}
  In particular, the inclusion
  $R\colon \Cat^L_\infty\hookrightarrow \Cat_\infty^\omega$
  is not an equivalence.
\end{lemma}

\begin{proof}[Construction]
  We define $V=V^{(\infty)}$ via the following pushout.
  \begin{equation}
    \cdsquareNA[po]
    {\coprod_{n : \naturals}\{1\}}
    {\coprod_{n : \naturals}E^{(n+1)}}
    {*}{V^{(\infty)}}
  \end{equation}
  In other words $V$ is the free $\infty$-category
  obtained by starting with set of objects
  $\{*\}\amalg \naturals$
  and attaching an $n$-invertible arrow $\eta_n\colon n\to *$
  for each $n\in \naturals$.
  Denote by $W\subseteq V$ the full sub-$\infty$-category 
  spanned by the objects $n\in \naturals$.

  As an auxiliary tool, let us introduce,
  for each $d\in \naturals$ the $(d+1)$-category $V^{(d)}$
  constructed in the same way but only with objects
  $\{*\}\amalg\naturals_{\leq d}$;
  then we have
  \begin{equation}
    \label{eq:V-colim-Vd}
    \colim_{d : \omega}V^{(d)}\xrightarrow{\simeq} V
  \end{equation}
  by comparing universal properties.
  Moreover, note that for each $d\in \naturals$,
  the structure map $i^d\colon V^{(d)}\to V$ has a retraction $r^d\colon V \to V^{(d)}$
  obtained by sending all new generating cells to
  (suitably iterated identities of) $*$.

  We can now verify the desired properties:
  \begin{enumerate}
  \item
    It suffices to show that $V$ is $\omega$-complete
    because this property is then immediately inherited by the full sub-$\infty$-category $W$.
    Let $f\colon x \to y$ be an $\omega$-invertible arrow of $V$ of some dimension.
    Then by \eqref{eq:V-colim-Vd} there merely exists a $d$
    and an arrow $f'\colon x'\to y'$ in $V^{(d)}$ with $f\simeq i^d(f')$.
    For the purpose of proving the proposition ``$f$ is invertible'',
    we may choose such a $d$ and such a $f'$.
    Since $f$ is $\omega$-invertible, so is its image $f'$
    under the retraction $r^d\colon V \to V^{(d)}$.
    But $V^{(d)}$ is a finite dimensional $\infty$-category,
    hence $\omega$-complete;
    it follows that $f'$ is invertible,
    and thus also its image $f$ under $i^d$.
  \item
    To show that the inclusion $W\hookrightarrow V$ is an $L$-equivalence,
    we use the characterization of \Cref{lem:L-iso-coind-inv}.
    Since it is fully faithful by definition,
    the only non-trivial case is that of $0$-arrows,
    and specifically that of the object $* \in V$.
    But by construction, for each $n\in \naturals$ we have an $n$-isomorphism
    $\eta_n\colon n\to *$ with $n \in W$.
  \item
    Let us assume a retraction $r\colon V\to W$.
    By definition of $W$, there merely exists a $d \in \naturals$
    with $r(*)\simeq d$; and since we proving a contradiction, we may choose it.
    Then we can apply $r$ to the $(d+1)$-invertible arrow
    $\eta_{d+1}\colon(d+1)\to *$
    to obtain a $(d+1)$-invertible arrow $(d+1)\to d$.
    Further applying (the inclusion $W\subseteq V$ followed by) the retraction
    $r^d\colon V\to V^{(d)}$ this gives a $(d+1)$-invertible
    arrow $*\to d$ in $V^{(d)}$.
    Since $V^{(d)}$ is only a $(d+1)$-category this arrow is then an isomorphism,
    so that we obtain an identification $*\simeq d$.
    But this is a contradiction to \Cref{lem:distinct-0-cells}
    because $V^{(d)}$ has a cell presentation where the $0$-cells
    $*$ and $d$ are distinct.
  \item
    Any $L$-equivalence $W\to V$ from an $L$-local object has a retraction
    (because then the composite $W\to V\to L(V)$ is an equivalence);
    hence $W$ cannot be $L$-local by the previous points.
  \end{enumerate}
\end{proof}

\twnote{maybe add computation of $L$-completion of $V$ and $W$?}
\subsection{Open questions}

We conclude with two questions that we were not yet able to answer.

\begin{question}
  What is the left adjoint to the inclusion
  $\Cat_\infty^\omega\hookrightarrow \Cat_\infty$?
\end{question}

\begin{remark}
  The natural guess is that, for each $C\in \Cat_\infty$,
  the unit should be the localization $C\to C'\coloneqq C[\isos{\omega}^{-1}]$
  of $C$ at the $\omega$-isomorphisms.
  It is clear that this map is an $\omega$-complete equivalence
  (because by definition, any map $C\to D$ into an $\omega$-complete category $D$
  must invert all $\omega$-isomorphisms),
  but we do not know whether $C'$ is actually $\omega$-complete.
  A priori, it might be necessary to repeat this process transfinitely
  by setting
  \begin{equation}
    C_{\alpha+1}\coloneqq (C_\alpha)'
    \quad
    \text{and}
    \quad
    C_{\lambda}\coloneqq \colim_{\alpha<\lambda}C_\alpha
  \end{equation}
  for successor and limit ordinals, respectively.
  Since $\omega$-isomorphisms are determined by a countable amount of data
  (because $E^{(\omega)}$ has countably many generating cells),
  this process will definitely stabilize at the first uncountable ordinal $\omega_1$
  so that $C\to C_{\omega_1}$ is the desired $\omega$-complete approximation.
  The question is whether this process actually stabilizes at an earlier stage
  $C_\alpha$, maybe $\alpha=\omega$ or possibly even $\alpha=1$.
\end{remark}
\begin{question}
  What are the arrows inverted by the localization
  $L_\omega\colon \Cat_\infty\to \Cat^\omega_\infty$?
\end{question}

\begin{remark}
  Again there is a natural guess, namely that the $L_\omega$-equivalences
  are precisely the maps that are $\infty$-surjective up to $\omega$-isomorphisms.
\end{remark}

\bibliographystyle{amsalpha}
\bibliography{ref}

\end{document}